\documentclass[a4paper,10pt,twoside]{amsart}
\usepackage[utf8]{inputenc}
\usepackage{latexsym}
\usepackage[pdftex]{graphicx}
\usepackage{hyperref}
\usepackage{amsmath,amssymb, amsthm, amscd}
\usepackage{multicol}
\usepackage[all]{xy}
\usepackage{etoolbox}
\usepackage{enumitem}
\expandafter\patchcmd\csname\string\proof\endcsname
  {\normalparindent}{0pt }{}{}
\makeatletter
\patchcmd{\@thm}{\thm@headfont{\scshape}}{\thm@headfont{\scshape\bfseries}}{}{}
\patchcmd{\@thm}{\thm@notefont{\fontseries\mddefault\upshape}}{}{}{}
\makeatother
\makeatletter
\patchcmd\@thm
  {\let\thm@indent\indent}{\let\thm@indent\indent}%
  {}{}
\makeatother
\newtheorem{theorem}[equation]{Theorem}
\newtheorem{lemma}[equation]{Lemma}
\newtheorem{proposition}[equation]{Proposition}
\newtheorem{corollary}[equation]{Corollary}
\theoremstyle{definition}
\newtheorem{definition}[equation]{Definition}
\theoremstyle{definition}
\newtheorem{remark}[equation]{Remark}
\theoremstyle{remark}

\numberwithin{subsection}{section}
\numberwithin{equation}{subsection}
\newcommand{\iso}{\xrightarrow{
   \,\smash{\raisebox{-0.50ex}{\ensuremath{\scriptstyle\sim}}}\,}}
\title
[Locally analytic representations]
{Locally analytic representations in the \'{e}tale coverings of the Lubin-Tate moduli space}
\author{Mihir Sheth}
\address{School of Mathematics, Tata Institute of Fundamental Research \\ Homi Bhabha Road, Mumbai - 400005, India.}
\email{mihir@math.tifr.res.in}
\begin{document}
\maketitle
\begin{abstract} 
The Lubin-Tate moduli space $X_{0}^{\textnormal{rig}}$ is a $p$-adic analytic open unit polydisc which parametrizes deformations of a formal group $H_{0}$ of finite height defined over an algebraically closed field of characteristic $p$. It is known that the natural action of the automorphism group $\textnormal{Aut}(H_{0})$ on $X^{\textnormal{rig}}_{0}$ gives rise to locally analytic representations on the topological duals of the spaces $H^{0}(X^{\text{rig}}_{0},(\mathcal{M}^{s}_{0})^{\mathrm{rig}})$ of global sections of certain equivariant vector bundles $(\mathcal{M}^{s}_{0})^{\mathrm{rig}}$ over $X^{\mathrm{rig}}_{0}$. In this article, we show that this result holds in greater generality. On the one hand, we work in the setting of deformations of formal modules over the valuation ring of a finite extension of $\mathbb{Q}_{p}$. On the other hand, we also treat the case of representations arising from the vector bundles $(\mathcal{M}^{s}_{m})^{\mathrm{rig}}$ over the deformation spaces $X^{\mathrm{rig}}_{m}$ with Drinfeld level-$m$-structures. Finally, we determine the space of locally finite vectors in $H^{0}(X^{\text{rig}}_{m},(\mathcal{M}^{s}_{m})^{\mathrm{rig}})$. Essentially, all locally finite vectors arise from the global sections of invertible sheaves over the projective space via pullback along the Gross-Hopkins period map.    
\end{abstract}
\tableofcontents
\section{Introduction} 
\indent The theory of locally analytic representations provides a suitable framework to study continuous $p$-adic representations of $p$-adic reductive groups in the context of the $p$-adic Langlands program. Introduced by Schneider-Teitelbaum and later developed by Emerton, the notion of a \emph{locally analytic representation} $V$ of a $p$-adic Lie group $G$ is (roughly) defined by the property that, for each vector $v\in V$, the orbit map $G\longrightarrow V$, $g\longmapsto g(v)$, is locally on $G$ given by a convergent power series with coefficients in $V$ (cf. \cite{eme04}, \cite{stladist}). Thus the category of locally analytic representations encompasses classical smooth representations, finite dimensional algebraic representations as well as tensor products of these two, the so-called locally algebraic representations. A fundamental theorem of locally analytic representation theory establishes an anti-equivalence between the category of locally analytic representations on vector spaces of compact type and the category of continuous $D(G)$-modules on nuclear Fr\'{e}chet spaces via duality functor, where $D(G)$ is the algebra of \emph{locally analytic distributions} on $G$ (cf. \cite{stladist}, Corollary 3.4).
\\
\indent  First non-trivial examples of locally analytic representations coming from geometry were found by Morita in his investigation of the $p$-adic upper half plane or the \emph{Drinfeld's upper half space} of dimension 1 (cf. \cite{mor}). In general, if $K$ is a finite extension of $\mathbb{Q}_{p}$ then the Drinfeld's upper half space $Y^{\text{rig}}_{0}$ of dimension $h-1$ is obtained by deleting all $K$-rational hyperplanes from the projective space $\mathbb{P}^{h-1}_{K}$. The natural action of $GL_{h}(K)$ on the projective space stabilizes $Y^{\text{rig}}_{0}$. Restricting any $GL_{h}(K)$-equivariant vector bundle $\mathcal{F}$ on $\mathbb{P}^{h-1}_{K}$ to $Y^{\text{rig}}_{0}$ gives rise to a locally analytic $GL_{h}(K)$-representation on the strong topological dual of the nuclear Fr\'{e}chet space $\mathcal{F}(Y^{\text{rig}}_{0})$ of its global sections (cf. \cite{orl}, \cite{stpboundary}). The upper half space and its \'{e}tale coverings give rise to the \emph{Drinfeld's upper half space $Y^{\mathrm{rig}}_{\infty}$ at infinity} which is a moduli space parametrizing certain EL Rapoport-Zink data. The dual space to $Y^{\text{rig}}_{\infty}$ is the \emph{Lubin-Tate moduli space $X^{\textnormal{rig}}_{\infty}$ at infinity} parametrizing the dual EL Rapoport-Zink data (cf. \cite{schwein}, Section 7). Analogous to the general linear group action on the upper half space, there is a natural action of another $p$-adic Lie group $\Gamma$ on the Lubin-Tate moduli space $X^{\text{rig}}_{0}$ and its \'{e}tale covers. While examining this action of $\Gamma$, Kohlhaase first showed that, in this case too, one obtains locally analytic representations on the dual space of the global sections of certain equivariant vector bundles over the Lubin-Tate moduli space provided $K=\mathbb{Q}_{p}$ (cf. \cite{kohliwamo}, Theorem 3.5). The aim of this article is twofold, firstly to generalize Kohlhaase's result to any finite base extension $K$ of $\mathbb{Q}_{p}$ and extend it to the finite \'{e}tale coverings of the Lubin-Tate moduli space, secondly to compute locally finite (algebraic) vectors in the concerned representations in order to understand their structure. \\        
\indent  To describe our results in detail, let $p$ be a prime number and $K$ be a finite extension of $\mathbb{Q}_{p}$ with ring of integers $\mathfrak{o}$, uniformizer $\varpi$ and residue class field $k$. Let us denote by $\breve{K}$ the completion of the maximal unramified extension of $K$ and by $\breve{\mathfrak{o}}$ its ring of integers. Fix a (unique) one-dimensional formal $\mathfrak{o}$-module $H_{0}$ over an algebraic closure $\overline{k}$ of $k$ of finite height $h$. The Lubin-Tate moduli space is a formal scheme $X_{0}$ parametrizing deformations of $H_{0}$ to complete local $\breve{\mathfrak{o}}$-algebras with residue field $\overline{k}$\footnote{We will mostly refer to the generic fibre $X^{\text{rig}}_{0}$ as the Lubin-Tate moduli space.}. Adding level structures to the moduli problem, Drinfeld showed that the formal scheme $X_{m}$ parametrizing deformations equipped with a level-$m$-structure is a finite flat covering of $X_{0}$, and $X_{0}$ is (non-canonically) isomorphic to the formal spectrum $\text{Spf}(\breve{\mathfrak{o}}[[u_{1},\ldots, u_{h-1}]])$ (cf. \cite{dr}, Proposition 4.2 and Proposition 4.3). Passing to the generic fibres of the formal schemes, one obtains a tower of rigid $\breve{K}$-analytic spaces $(X^{\text{rig}}_{m})_{m\in\mathbb{N}_{0}}$ carrying commuting actions of the covering group $GL_{h}(\mathfrak{o})$ and of the automorphism group $\Gamma\cong\mathfrak{o}_{B_{h}}^{\times}$ of $H_{0}$, where $\mathfrak{o}_{B_{h}}$ is the maximal order of the central $K$-division algebra $B_{h}$ of invariant $1/h$. The covering group action on $X^{\text{rig}}_{m}$ factors through the finite group $GL_{h}(\mathfrak{o}/\varpi^{m}\mathfrak{o})$ making $X^{\text{rig}}_{m}$ an \'{e}tale Galois cover of the open unit polydisc $X^{\text{rig}}_{0}$ with Galois group $GL_{h}(\mathfrak{o}/\varpi^{m}\mathfrak{o})$. On the other hand, the $\Gamma$-action on $X^{\text{rig}}_{m}$ is much more complicated and is the one we are interested in. These group actions are of significance, as they realize the local Jacquet-Langlands correspondence on the $l$-adic \'{e}tale cohomology of the Lubin-Tate tower, as conjectured by Carayol (cf. \cite{car}. \cite{strdeform}). \\
\indent  Let us consider the $\Gamma$-equivariant vector bundles over $X^{\text{rig}}_{m}$ induced by the $s$-fold tensor power of the Lie algebra of the universal formal $\mathfrak{o}$-module $\mathbb{H}^{(m)}$ at level $m$ for any integer $s$, and denote by $M^{s}_{m}$ the global sections of these vector bundles. The $\Gamma$-action on the nuclear $\breve{K}$-Fr\'{e}chet space $M^{s}_{m}$ is semi-linear for its action on $\mathcal{O}_{X^{\text{rig}}_{m}}(X^{\text{rig}}_{m})=M^{0}_{m}$. In our first main result, we prove that the strong topological $\breve{K}$-linear dual $(M^{s}_{m})'_{b}$ of $M^{s}_{m}$ is a locally $K$-analytic representation of $\Gamma$ for all $s\in\mathbb{Z}$ and $m\geq 0$ (cf. Theorem \ref{genlaK0} and Theorem \ref{genlaKm}). The proof of local analyticity essentially follows the Kohlhaase's approach in \cite{kohliwamo} and consists of the following two steps:
\begin{enumerate}[leftmargin=*]
\item[1.] The Gross-Hopkins' \emph{$p$-adic period map} $\Phi:X^{\text{rig}}_{0}\longrightarrow\mathbb{P}^{h-1}_{\breve{K}}$ constructed in \cite{gh} can be used to explicitly find out the $\Gamma$-action on the \emph{fundamental domain} $D$ of $X^{\text{rig}}_{0}$. We first show that the explicitly known $\Gamma$-action on the sections $M^{s}_{D}$ over $D$ is locally $K$-analytic by direct computations.
\item[2.] Using the structure theory of the locally $\mathbb{Q}_{p}$-analytic distribution algebra $D(\Gamma_{\mathbb{Q}_{p}})$, we then show that the continuous $\Gamma$-action on $M^{s}_{m}$ extends to a continuous action of $D(\Gamma_{\mathbb{Q}_{p}})$. Finally, to deduce that this action factors through a continuous action of the locally $K$-analytic distribution algebra $D(\Gamma)$, we use the step 1 and the \'{e}taleness of the covering morphisms.
\end{enumerate}

\indent  Our second main result concerns computing the subspace of $(M^{s}_{m})_{\text{lf}}$ of locally finite vectors in the $\Gamma$-representations $M^{s}_{m}$. A \emph{locally finite vector} is a vector contained in a finite dimensional subrepresentation of some open subgroup of $\Gamma$. Consider the $\breve{K}$-linear algebraic representation $B_{h}\otimes_{K_{h}}\breve{K}$ on which $\Gamma$ acts by the left multiplication, and let $\breve{K}_{m}$ denote the $m$-th Lubin-Tate extension of $\breve{K}$ equipped with a smooth $\Gamma$-action via $\mathfrak{o}_{B_{h}}^{\times}\xrightarrow{\textnormal{Nrd}}\mathfrak{o}^{\times}\twoheadrightarrow(\mathfrak{o}/\varpi^{m}\mathfrak{o})^{\times}\cong\text{Gal}(\breve{K}_{m}/\breve{K})$. We show that there is an isomorphism 
\begin{equation*}
(M^{s}_{m})_{\textnormal{lf}}\cong
\breve{K}_{m}\otimes_{\breve{K}}\textnormal{Sym}^{s}(B_{h}\otimes_{K_{h}}\breve{K})\cong\breve{K}_{m}\otimes_{\breve{K}}\mathcal{O}_{\mathbb{P}^{h-1}_{\breve{K}}}(s)(\mathbb{P}^{h-1}_{\breve{K}}).
\end{equation*} of $\Gamma$-representations for all $m\geq 0$ and $s\in\mathbb{Z}$ (cf. Corollary \ref{lf1}, Theorem \ref{lf2}, Theorem \ref{lf3}). Moreover, $(M^{s}_{m})_{\text{lf}}$ is a finite dimensional semi-simple locally algebraic representation. To prove the above isomorphism, we extensively use the action of the Lie algebra of $\Gamma$ obtained from the Gross-Hopkins' period map. The other key ingredients of the proof are the \emph{generic flatness} of the line bundle induced by the Lie algebra of the universal additive extension (cf. \cite{gh}, Section 21), Strauch's result on geometrically connected components of $X^{\text{rig}}_{m}$ (cf. \cite{strgeo}) and Fargues' \emph{cellular decomposition} of the Lubin-Tate tower (cf. \cite{fgl}, Section I.7).\\
\indent We expect our both results to hold in much greater generality. The $\Gamma$-equivariant vector bundles that we consider arise as the pullbacks of the invertible sheaves $\mathcal{O}_{\mathbb{P}^{h-1}_{\breve{K}}}(s)$ on the projective space along the Gross-Hopkins' period map (cf. Remark \ref{our_line_bundle_is_a_pullback_of_O(s)*}). Given any $\Gamma$-equivariant vector bundle $\mathcal{G}$ over the projective space $\mathbb{P}_{\breve{K}}^{h-1}$, we believe that the representations realized on the global sections over $X^{\text{rig}}_{m}$ of its pullback along the period map are dual to locally analytic representations, and the locally algebraic part again comes from the global sections of $\mathcal{G}$. However, we don't have a proof as of now. Another major open question concerning the locally analytic $\Gamma$-representations $(M^{s}_{m})'_{b}$ is whether they are \emph{admissible} or not in the sense of \cite{stadmrep}, Section 6. The similar representations in the example of the Drinfeld's upper half space and its first \'{e}tale covering are known to be admissible (cf. \cite{pss}). However, the presence of the spherical Hecke algebra action on the space of global rigid analytic functions on the Lubin-Tate moduli space raises questions on the admissibility of $(M^{s}_{m})'_{b}$ (cf. \cite{kohliwath}, Proposition 3.3 and Remark 3.5). We would also like to mention the work \cite{lo15} of Chi Yu Lo, showing the analyticity of the action of a certain rigid analytic group associated to $\Gamma$ on a particular closed polydisc of $X^{\textrm{rig}}_{0}$, which will likely be relevant in further investigations of locally analytic representations coming from the Lubin-Tate moduli space. 
\subsection*{Acknowledgements} The results presented in this article form an integral part of the author's Ph.D. thesis conducted at the Fakult\"{a}t f\"{u}r Mathematik, Universit\"{a}t Duisburg-Essen, Germany under the supervision of Jan Kohlhaase. The author is extremely grateful to his supervisor for introducing and explaining the problem, for suggesting valuable ideas, and for many insightful discussions. The author is also thankful to an anonymous referee for many thoughtful remarks and corrections.
\subsection*{Notation and conventions} $\mathbb{N}$ and $\mathbb{N}_{0}$ denote the set of positive integers and the set of non-negative integers respectively. If $\alpha=(\alpha_{1},\ldots ,\alpha_{r})\in\mathbb{N}^{r}_{0}$ is an $r$-tuple of non-negative integers and $T=(T_{1},\ldots ,T_{r})$ is a family of indeterminates for some $r\in\mathbb{N}$, then we set $|\alpha|:=\alpha_{1}+\ldots+\alpha_{r}$, and $T^{\alpha}:=T_{1}^{\alpha_{1}}\cdots T_{r}^{\alpha_{r}}$. Unless stated otherwise, all rings are considered to be commutative with identity. A ring extension $A\subseteq B$ will be denoted by $B|A$, and its degree by $[B:A]$ if it is finite and free. Let $p$ be a fixed prime number and let $K$ be a finite field extension of $\mathbb{Q}_{p}$ with the valuation ring $\mathfrak{o}$. We fix a uniformizer $\varpi$ of $K$ and let $k:=\mathfrak{o}/\varpi\mathfrak{o}$ denote its residue class field of characteristic $p$ and cardinality $q$. The absolute value $\vert\cdot\vert$ of $K$ is assumed to be normalized through $\vert p\vert =p^{-1}$. We denote by $\breve{K}$ the completion of the maximal unramified extension of $K$, and by $\breve{\mathfrak{o}}$ its valuation ring. We denote by $\sigma$ the Frobenius automorphism of an algebraic closure $\overline{k}$ of $k$, as well as its unique lift to a ring automorphism of $\breve{\mathfrak{o}}$ and the induced field automorphism of $\breve{K}$. We also fix an algebraic closure $\overline{\breve{K}}$ of $\breve{K}$ and denote its valuation ring by $\overline{\breve{\mathfrak{o}}}$. The absolute value $|\cdot|$ on $K$ extends uniquely to $\breve{K}$, and to $\overline{\breve{K}}$. For a positive integer $h$, let $K_{h}$ be the unramified extension of $K$ of degree $h$, $\mathfrak{o}_{h}$ be its valuation ring, and $B_{h}$ be the central $K$-division algebra of invariant $1/h$. We fix an embedding $K_{h} \hookrightarrow B_{h}$ and a uniformizer $\Pi$ of $B_{h}$, satisfying $\Pi^{h}=\varpi$. Let $\text{Nrd}:B_{h}\longrightarrow K$ denote the reduced norm of $B_{h}$ over $K$. The symbol $\mathbb{P}^{h-1}_{\breve{K}}$ always denotes the $(h-1)$-dimensional rigid analytic projective space over $\breve{K}$.
\section{Drinfeld's coverings of the Lubin-Tate moduli space and the group actions}
\indent  We begin with a quick introduction to the Lubin-Tate deformation problem equipped with Drinfeld's level structures. Then we prove the main result of this section namely the continuity of the $\Gamma$-action on the universal deformation rings. 
\subsection{Deformations of formal $\mathfrak{o}$-modules with level structures}
\indent Recall from \cite{gh} that a one-dimensional formal $\mathfrak{o}$-module $F$ over a local $\mathfrak{o}$-algebra $A$ is, after having fixed a formal coordinate, given by a formal power series $F(X,Y)\in A[[X,Y]]$, together with a ring homomorphism $[\hspace{.1cm}\cdot\hspace{.1cm}]_{F}:\mathfrak{o}\longrightarrow \textnormal{End}(F)$ such that $[\lambda]_{A}(X)\equiv i_{A}(\lambda)X$ (mod deg 2), where $i_{A}:\mathfrak{o}\longrightarrow A$ is the structure morphism. Let $H_{0}$ be a one-dimensional formal $\mathfrak{o}$-module of finite height $h$ over $\overline{k}$ which is defined over $k$. According to \cite{dr}, Proposition 1.6 and 1.7, the formal module $H_{0}$ is unique up to isomorphism, and one has \begin{equation}\label{endH0}
\text{End}(H_{0})\cong\mathfrak{o}_{B_{h}}
\end{equation}  where $\mathfrak{o}_{B_{h}}$ is the valuation ring of the central $K$-division algebra $B_{h}$ of invariant $1/h$. Let $\mathcal{C}$ be the category of commutative unital complete Noetherian local $\breve{\mathfrak{o}}$-algebras $R=(R,\mathfrak{m}_{R})$ with residue class field $\overline{k}$. The Lubin-Tate deformation problem considers liftings of $H_{0}$ to the objects of $\mathcal{C}$ together with certain additional data defined below. 
\begin{definition} Let $R$ be an object of $\mathcal{C}$ and $H$ be a formal $\mathfrak{o}$-module over $R$, given by a power series $H(X,Y)\in R[[X,Y]]$.
\begin{enumerate}[leftmargin=*]
\item[1.] A pair $(H,\rho)$, where $\rho:H_{0}\iso H\otimes_{R}\overline{k}$ is an isomorphism of formal $\mathfrak{o}$-modules over $\overline{k}$, is called \emph{a deformation of $H_{0}$ to $R$}.
\item[2.] Denote by $(\mathfrak{m}_{R},+_{H})$ the abstract $\mathfrak{o}$-module $\mathfrak{m}_{R}$ in which addition and \linebreak $\mathfrak{o}$-multiplication are defined as $x+_{H}y:=H(x,y)$ and $ax:=[a]_{H}(x)$ respectively for all $x,y\in\mathfrak{m}_{R}$, $a\in\mathfrak{o}$. For a non-negative integer $m$, a \emph{Drinfeld level-$m$-structure on $H$} is a homomorphism  $\phi : \big(\frac{\varpi^{-m}\mathfrak{o}}{\mathfrak{o}}\big)^{h}\longrightarrow (\mathfrak{m}_{R},+_{H})$ of abstract $\mathfrak{o}$-modules such that $\prod_{\alpha\in(\frac{\varpi^{-m}\mathfrak{o}}{\mathfrak{o}})^{h}}(X-\phi(\alpha))$ divides $[\varpi^{m}]_{H}(X)$ in $R[[X]]$.
\item[3.] We call the triple $(H,\rho,\phi)$ \emph{a deformation of $H_{0}$ to $R$ with level-m-structure} if $(H,\rho)$ is a deformation of $H_{0}$ to $R$ and $\phi$ is a Drinfeld level-$m$-structure on $H$. 
\end{enumerate}
\end{definition}
\noindent Two deformations $(H,\rho,\phi)$ and $(H',\rho',\phi')$ of $H_{0}$ to $R$ with level-$m$-structures are isomorphic if there is an isomorphism $f:H\iso H'$ of formal $\mathfrak{o}$-modules over $R$ making the following diagrams commutative.
\begin{displaymath}
\xymatrix@C=13pt{
         H\otimes_{R}\overline{k}\ar[rr]^{f\otimes_{R}\overline{k}}& & H'\otimes_{R}\overline{k}    \\
          & H_{0} \ar[ul]^{\rho}\ar[ur]_{\rho'}&} 
\hspace{1cm}
\xymatrix@C=2pt{
        (\mathfrak{m}_{R},+_{H}) \ar[rr]^{f}& &(\mathfrak{m}_{R},+_{H'})  \\
                   &\big(\frac{\varpi^{-m}\mathfrak{o}}{\mathfrak{o}}\big)^{h} \ar[ul]^{\phi}\ar[ur]_{\phi'}&}
\end{displaymath}
\indent  For any integer $m\geq 0$, consider the set valued functor Def$_{m}:\mathcal{C}\longrightarrow \textnormal{Set}$, which associates to an object $R$ of $\mathcal{C}$ the set of isomorphism classes of deformations of $H_{0}$ to $R$ with level-$m$-structures. For a morphism $\varphi:R\longrightarrow R'$ in $\mathcal{C}$, Def$_{m}(\varphi)$ is defined by sending a class $[(H, \rho, \phi)]$ to the class $[(H\otimes_{R}R',\rho,\varphi\circ\phi)]$. Notice that $\rho:H_{0}\iso H\otimes_{R}\overline{k}\cong (H\otimes_{R}R')\otimes_{R'}\overline{k}$. We denote the triple $(H\otimes_{R}R',\rho,\varphi\circ\phi)$ by $\varphi_{*}(H,\rho,\phi)$ for simplicity. 
\begin{theorem}[Lubin-Tate, Drinfeld]\label{fundthm}\hfill
\begin{enumerate}[leftmargin=*]
\item[\textnormal{1.}] The functor $\textnormal{Def}_{m}$ is representable by a regular local ring $R_{m}$ of dimension $h$ for all $m\geq 0$. 
\item[\textnormal{2.}] For any two integers $0\leq m\leq m'$, the natural transformation $\textnormal{Def}_{m'}\longrightarrow\textnormal{Def}_{m}$ of functors defined by sending a class $[(H,\rho,\phi)]$ to $[(H,\rho,\phi\vert_{(\frac{\varpi^{-m}\mathfrak{o}}{\mathfrak{o}})^{h}})]$ induces a homomorphism of local rings $R_{m}\longrightarrow R_{m'}$ which is finite and flat.
\item[\textnormal{3.}] The ring $R_{0}$ is non-canonically isomorphic to the ring $\breve{\mathfrak{o}}[[u_{1},\ldots,u_{h-1}]]$ of formal power series in $h-1$ indeterminates over $\breve{\mathfrak{o}}$.
\end{enumerate}
\end{theorem}
\begin{proof}
See \cite{dr}, Proposition 4.2 and 4.3.
\end{proof}
\indent  Let us denote the universal deformation of $H_{0}$ to $R_{m}$ with level-$m$-structure by the triple $(\mathbb{H}^{(m)},\rho^{(m)},\phi^{(m)})$. Here $\mathbb{H}^{(m)}=\mathbb{H}^{(0)}\otimes_{R_{0}} R_{m}$, i.e., the universal formal $\mathfrak{o}$-module $\mathbb{H}^{(m)}$ over $R_{m}$ is given by the base change of the universal formal $\mathfrak{o}$-module $\mathbb{H}^{(0)}$ over $R_{0}$ under the map $R_{0}\longrightarrow R_{m}$ induced by Part 2, Theorem \ref{fundthm}. We note that, since $R_{0}\cong\breve{\mathfrak{o}}[[u_{1},\ldots,u_{h-1}]]$ is an integral domain, the flatness of the map $R_{0}\longrightarrow R_{m}$ implies $R_{0}\hookrightarrow R_{m}$ for all $m\geq 0$. \\
\indent  By the universal property, given an object $R$ of $\mathcal{C}$ and a deformation $(H,\rho,\phi)$ of $H_{0}$ to $R$ with level-$m$-structure, there is a unique $\breve{\mathfrak{o}}$-linear local ring \linebreak homomorphism $\varphi : R_{m}\longrightarrow R$ such that $\text{Def}_{m}(\varphi)([(\mathbb{H}^{(m)},\rho^{(m)},\phi^{(m)})])$=\linebreak$[\varphi_{*}(\mathbb{H}^{(m)},\rho^{(m)},\phi^{(m)})]=[(H,\rho,\phi)]$. The unique isomorphism  between the deformations $\varphi_{*}(\mathbb{H}^{(m)},\rho^{(m)},\phi^{(m)})$ and $(H,\rho,\phi)$ over $R$ will be denoted by \linebreak $[\varphi]:\varphi_{*}(\mathbb{H}^{(m)},\rho^{(m)},\phi^{(m)})\iso (H,\rho,\phi)$ (cf. \cite{gh}, Proposition 12.10). 
\subsection{The group actions}
\label{the_group_actions}
For all $m\geq 0$, the functor $\text{Def}_{m}$ admits natural commuting left actions of the groups $\Gamma:=\textnormal{Aut}(H_{0})$ and $G_{0}:=GL_{h}(\mathfrak{o})$ for which the morphisms $\textnormal{Def}_{m'}\longrightarrow\textnormal{Def}_{m}$ of functors mentioned in Part 2, Theorem \ref{fundthm} are equivariant. On $R$-valued points, they are given by 
\begin{equation*}
[(H,\rho,\phi)]\longmapsto[(H,\rho\circ\gamma^{-1},\phi)]\hspace{.2cm}\text{and}\hspace{.2cm}[(H,\rho,\phi)]\longmapsto[(H,\rho,\phi\circ g^{-1})]
\end{equation*} 
for $\gamma\in\Gamma,g\in G_{0}$. Here $g^{-1}\in GL_{h}(\mathfrak{o})$ acts on the free $(\frac{\mathfrak{o}}{\varpi^{m}\mathfrak{o}})$-module $(\frac{\varpi^{-m}\mathfrak{o}}{\mathfrak{o}})^{h}$ by considering it as an $\mathfrak{o}$-module via the natural reduction map $\mathfrak{o}\twoheadrightarrow\frac{\mathfrak{o}}{\varpi^{m}\mathfrak{o}}$. Because of the representability, these actions give rise to commuting left actions of $\Gamma$ and $G_{0}$ on the universal deformation rings $R_{m}$. We use the same letters $\gamma$ and $g$ to denote the automorphisms of $R_{m}$ induced by $\gamma\in\Gamma$ and by $g\in G_{0}$ respectively. It is immediate from the definition that the $G_{0}$-action on $R_{m}$ factors through a quotient by the $m$-th principal congruence subgroup $G_{m}:=1+\varpi^{m}M_{h}(\mathfrak{o})$ of $G_{0}$. For $m'\geq m\geq 0$, the induced action of $G_{m}/G_{m'}$ makes $R_{m'}[\frac{1}{\varpi}]$ \'{e}tale and Galois over $R_{m}[\frac{1}{\varpi}]$ with Galois group $G_{m}/G_{m'}$ (cf. Theorem 2.1.2 (ii), \cite{strdeform}).\\
\indent  The actions of $\Gamma$ and $G_{0}$ on $R_{m}$ induce semilinear actions of $\Gamma$ and $G_{0}$ on the Lie algebra $\text{Lie}(\mathbb{H}^{(m)})$ of the universal formal $\mathfrak{o}$-module $\mathbb{H}^{(m)}$. Recall that the Lie algebra $\text{Lie}(\mathbb{H}^{(m)})$ of $\mathbb{H}^{(m)}$ is the tangent space $\text{Hom}_{R_{m}}((X)/(X)^{2},R_{m})$ of its coordinate ring $R_{m}[[X]]$ (equipped with the trivial Lie bracket). We now describe the $\Gamma$-action on $\text{Lie}(\mathbb{H}^{(m)})$; the $G_{0}$-action is defined likewise. Given $\gamma\in\Gamma$, extend the ring automorphism $\gamma$ of $R_{m}$ to $R_{m}[[X]]$ by sending $X$ to itself. This induces a homomorphism \begin{equation*}
\gamma_{*}:\textnormal{Lie}(\mathbb{H}^{(m)})\longrightarrow\textnormal{Lie}(\gamma_{*}\mathbb{H}^{(m)})
\end{equation*} of additive groups. The isomorphism $[\gamma]:\gamma_{*}\mathbb{H}^{(m)}\iso\mathbb{H}^{(m)}$ also induces a natural $R_{m}$-linear map \begin{equation*}
\textnormal{Lie}([\gamma]):\textnormal{Lie}(\gamma_{*}\mathbb{H}^{(m)})\longrightarrow\textnormal{Lie}(\mathbb{H}^{(m)}).
\end{equation*}  We define $\gamma:\textnormal{Lie}(\mathbb{H}^{(m)})\longrightarrow\textnormal{Lie}(\mathbb{H}^{(m)})$ as the composite of these two maps, i.e., $\gamma:=\textnormal{Lie}([\gamma])\circ\gamma_{*}$.\\
\indent  Given another element $\gamma'\in\Gamma$, let $\gamma'_{*}[\gamma]:\gamma'_{*}(\gamma_{*}\mathbb{H}^{(m)})\iso\gamma'_{*}\mathbb{H}^{(m)}$ be the isomorphism obtained by applying $\gamma'$ to the coefficients of $[\gamma]$. Then $[\gamma']\circ\gamma'_{*}[\gamma]$ is an isomorphism between the formal $\mathfrak{o}$-modules $(\gamma'\gamma)_{*}\mathbb{H}^{(m)}$ and $\mathbb{H}^{(m)}$ over $R_{m}$. Therefore by uniqueness, we have $[\gamma'\gamma]=[\gamma']\circ\gamma'_{*}[\gamma]$. One also checks easily that the following diagram commutes.
$$	
\xymatrixcolsep{4pc}
\xymatrix{
\text{Lie}(\gamma_{*}\mathbb{H}^{(m)})\ar[d]_{\gamma'_{*}}\ar[r]^{\text{Lie}([\gamma])}&\text{Lie}(\mathbb{H}^{(m)})\ar[d]^{\gamma'_{*}}\\
\text{Lie}(\gamma'_{*}(\gamma_{*}\mathbb{H}^{(m)}))\ar[r]_{\text{Lie}(\gamma'_{*}[\gamma])}&\text{Lie}(\gamma'_{*}\mathbb{H}^{(m)})}$$ Then it follows that \begin{align}\label{gamma_action_on_lie_algebra}
\text{Lie}([\gamma'\gamma])\circ(\gamma'\gamma)_{*}&=\text{Lie}([\gamma'])\circ\text{Lie}(\gamma'_{*}[\gamma])\circ\gamma'_{*}\circ\gamma_{*}\\&=\text{Lie}([\gamma'])\circ(\gamma'_{*}\circ\text{Lie}([\gamma])\circ(\gamma'_{*})^{-1})\circ\gamma'_{*}\circ\gamma_{*}\nonumber\\&=\text{Lie}([\gamma'])\circ\gamma'_{*}\circ\text{Lie}([\gamma])\circ\gamma_{*}\nonumber.
\end{align}
\indent  Thus we obtain an action of $\Gamma$ (and of $G_{0}$) on the additive group $\textnormal{Lie}(\mathbb{H}^{(m)})$ which is semilinear for the action of $\Gamma$ (and of $G_{0}$ respectively) on $R_{m}$ because $\gamma_{*}$ is semilinear. Given a positive integer $s$, we denote by $\textnormal{Lie}(\mathbb{H}^{(m)})^{\otimes s}$ the $s$-fold tensor product of $\textnormal{Lie}(\mathbb{H}^{(m)})$ over $R_{m}$ with itself. This is a free $R_{m}$-module of rank 1 with a semi-linear action of $\Gamma$ defined by $\gamma(\delta_{1}\otimes\cdots\otimes\delta_{s}):=\gamma(\delta_{1})\otimes\cdots\otimes\gamma(\delta_{s})$. Set $\textnormal{Lie}(\mathbb{H}^{(m)})^{\otimes 0}:=R_{m}$ and $\textnormal{Lie}(\mathbb{H}^{(m)})^{\otimes s}:=\textnormal{Hom}_{R_{m}}(\textnormal{Lie}(\mathbb{H}^{(m)})^{\otimes (-s)},R_{m})$ if $s$ is a negative integer. In the latter case, a semi-linear action of $\Gamma$ is defined by $\gamma(\varphi)(\delta_{1}\otimes\cdots\otimes\delta_{-s}):=\gamma(\varphi(\gamma^{-1}(\delta_{1})\otimes\cdots\otimes\gamma^{-1}(\delta_{-s})))$. The semi-linear actions of $G_{0}$ on the $s$-fold tensor products are defined similarly. As before,  for all $s\in\mathbb{Z}$, the $G_{0}$-action on $\textnormal{Lie}(\mathbb{H}^{(m)})^{\otimes s}$ factors through $G_{0}/G_{m}$.
\begin{remark}\label{Gamma_and_G0_actions_commute}
Using that the group actions of $\Gamma$ and $G_{0}$ on $R_{m}$ commute, one can show that they commute on $\textnormal{Lie}(\mathbb{H}^{(m)})^{\otimes s}$ as follows. It suffices to show the commutativity for $s=1$. Since the $G_{0}$-action is defined likewise, we may use (\ref{gamma_action_on_lie_algebra}) for $\gamma\in\Gamma$ and $g\in G_{0}$. As a result, we get \begin{align*}
\text{Lie}([g])\circ g_{*}\circ\text{Lie}([\gamma])\circ\gamma_{*}&=\text{Lie}([g\gamma])\circ (g\gamma)_{*}\\&=\text{Lie}([\gamma g])\circ(\gamma g)_{*}\\&=\text{Lie}([\gamma])\circ\gamma_{*}\circ\text{Lie}([g])\circ g_{*}.
\end{align*}
\end{remark} 
\indent  We are primarily interested in the action of $\Gamma$, a $p$-adic Lie gorup, on \linebreak $\text{Lie}(\mathbb{H}^{(m)})^{\otimes s}$. Before describing the underlying Lie group structure of $\Gamma$, we refer the reader to \cite{schplie}, page 38, page 47 and page 89 for the definitions of a locally analytic map, a locally analytic manifold and a locally analytic group respectively. By (\ref{endH0}), we have $\Gamma\cong\mathfrak{o}_{B_{h}}^{\times}$. Recall that the division algebra $B_{h}$ is a $K_{h}$-vector space of dimension $h$ with basis $\lbrace\Pi^{i}\rbrace_{0\leq i\leq h-1}$ whose multiplication is determined by the relations $\Pi^{h}=\varpi$ and $\Pi\lambda=\lambda^{\sigma}\Pi$ for all $\lambda\in K_{h}$ ($\lambda^{\sigma}$ denotes the image of $\lambda$ under the Frobenius automorphism $\sigma$). Thus, any $\gamma\in\Gamma=\mathfrak{o}_{B_{h}}^{\times}$ can be uniquely written as \begin{equation*}
\gamma = \sum_{i=0}^{h-1}\lambda_{i}\Pi^{i}
\end{equation*}
with $\lambda_{0}\in\mathfrak{o}_{h}^{\times}$  and $\lambda_{1},\dots ,\lambda_{h-1}\in \mathfrak{o}_{h}$. The map \begin{align}\label{chart_for_Gamma}
&\hspace{.6cm}\psi:\Gamma\longrightarrow K^{h}_{h}\\&\sum_{i=0}^{h-1}\lambda_{i}\Pi^{i}\longmapsto(\lambda_{0},\lambda_{1}\dots \lambda_{h-1})\nonumber
\end{align} identifies $\Gamma$ with a compact open subset $\mathfrak{o}_{h}^{\times}\times\mathfrak{o}_{h}^{h-1}$ of $K_{h}^{h}$ making it into a compact open locally $K_{h}$-analytic submanifold of $K_{h}^{h}$. The composition map \begin{equation*}
\psi(\Gamma)\times\psi(\Gamma)\xrightarrow{\psi^{-1}\times\psi^{-1}}\Gamma\times\Gamma\xrightarrow{\textnormal{multiplication}}\Gamma\xrightarrow{\psi}\psi(\Gamma)
\end{equation*} from an open subset in $K_{h}^{2h}$ to $K_{h}^{h}$ can be easily seen to be locally $K$-analytic since each component of this map is a composition of a polynomial and a $K$-linear Frobenius automorphism $\sigma$, both being locally $K$-analytic. Therefore, $\Gamma$ is a locally $K$-analytic group. However, notice that $\Gamma$ is not a locally $K_{h}$-analytic group because $\sigma:\mathfrak{o}_{h}^{\times}\longrightarrow\mathfrak{o}_{h}^{\times}$ is not locally $K_{h}$-analytic unless $h=1$.\\
\indent  Now, being a compact and a totally disconnected Hausdorff topological group, $\Gamma$ is a profinite topological group. A basis of neighbourhoods of the identity is given by the normal subgroups $\Gamma_{i}:=1+\varpi^{i}\mathfrak{o}_{B_{h}}=1+\varpi^{i}\textnormal{End}(H_{0})$, $i\geq 1$ of finite index. Let us put $\Gamma_{0}:=\Gamma$. Our aim is to show that the $\Gamma$-action on $\text{Lie}(\mathbb{H}^{(m)})^{\otimes s}$ is continuous, i.e., the action map $\Gamma\times\text{Lie}(\mathbb{H}^{(m)})^{\otimes s}\longrightarrow\text{Lie}(\mathbb{H}^{(m)})^{\otimes s}$ is continuous for the $\mathfrak{m}_{R_{m}}$-adic topology on $\text{Lie}(\mathbb{H}^{(m)})^{\otimes s}$, and for the product of profinite and $\mathfrak{m}_{R_{m}}$-adic topology on the left hand side. But, first we need a couple of lemmas. For any two non-negative integers $n$ and $m$, set $\mathbb{H}_{n}^{(m)}:=\mathbb{H}^{(m)}\otimes_{R_{m}}(R_{m}/\mathfrak{m}_{R_{m}}^{n+1})$. We have $H_{0}\cong\mathbb{H}_{0}^{(m)}$ via $\rho^{(m)}$ for all $m\geq 0$.
\begin{lemma} If $n$ and $m$ are non-negative integers then the homomorphism of $\mathfrak{o}$-algebras $\textnormal{End}(\mathbb{H}_{n+1}^{(m)})\longrightarrow \textnormal{End}(\mathbb{H}_{n}^{(m)})$, induced by reduction modulo $\mathfrak{m}_{R_{m}}^{n+1}$, is injective.
\end{lemma}
\begin{proof}
Let $m\geq 0$ be arbitrary. We show by induction on $n$ that the ring homomorphism $i_{n}:\textnormal{End}(\mathbb{H}_{n}^{(m)})\longrightarrow\textnormal{End}(\mathbb{H}_{0}^{(m)})$, induced by reduction modulo the maximal ideal, is injective for every $n\in\mathbb{N}_{0}$. The case $n=0$ is trivial. Let $n\geq 1$ and assume that $i_{n-1}$ is injective. Since $H_{0}$ is of height $h$, we have $[\varpi]_{\mathbb{H}_{0}^{(m)}}(X)\equiv\overline{u}X^{q^{h}}\mod{\textnormal{deg}\hspace{.1cm} q^{h}+1}$ for some $u\in R_{m}^{\times}$. Then $[\varpi]_{\mathbb{H}_{0}^{(m)}}=i_{n}([\varpi]_{\mathbb{H}_{n}^{(m)}})$ implies that \begin{equation*}
[\varpi]_{\mathbb{H}_{n}^{(m)}}(X)\equiv\overline{\varpi}X+\overline{b_{2}}X^{2}+\cdots+\overline{b_{q^{h}-1}}X^{q^{h}-1}+\overline{u}X^{q^{h}}\mod{\textnormal{deg}\hspace{.1cm} q^{h}+1} \end{equation*} for some $b_{2},\cdots,b_{q^{h}-1}\in\mathfrak{m}_{R_{m}}$.\\
\indent  Now let $f(X)=\sum_{i=1}^{\infty}\overline{a_{i}}X^{i}\in\textnormal{End}(\mathbb{H}_{n}^{(m)})$ such that $i_{n}(f)=0$, i.e., $a_{i}\in\mathfrak{m}_{R_{m}}$ for all $i\geq 1$. We need to show that $a_{i}\in\mathfrak{m}_{R_{m}}^{n+1}$ for all $i\geq 1$. However, the induction hypothesis implies that $a_{i}\in\mathfrak{m}_{R_{m}}^{n}$. Thus $[\varpi]_{\mathbb{H}_{n}^{(m)}}\circ f=0$. Since $[\varpi]_{\mathbb{H}_{n}^{(m)}}\circ f=f\circ[\varpi]_{\mathbb{H}_{n}^{(m)}}$, we get $a_{i}u^{i}\in\mathfrak{m}_{R_{m}}^{n+1}$ by induction $i$ and hence $a_{i}\in\mathfrak{m}_{R_{m}}^{n+1}$ for all $i\geq 1$.
\end{proof}
\indent The above lemma allows us to consider all the $\mathfrak{o}$-algebras $\text{End}(\mathbb{H}^{(m)}_{n})$ as subalgebras of $\text{End}(\mathbb{H}^{(m)}_{0})$.
\begin{proposition}\label{subringprop} For all $n\geq 0$, $m\geq 0$, the subalgebra $\textnormal{End}(\mathbb{H}_{n}^{(m)})$ of $\textnormal{End}(\mathbb{H}_{0}^{(m)})$ contains $\varpi^{n}\textnormal{End}(\mathbb{H}_{0}^{(m)})$.
\end{proposition}
\begin{proof}
Let $m\geq 0$ be arbitrary. We proceed by induction on $n$, the case $n=0$ being trivial. Let $n\geq 1$ and assume the assertion to be true for $n-1$. Let $\varphi\in\varpi^{n}\textnormal{End}(\mathbb{H}_{0}^{(m)})$. By induction hypothesis, we have $\varphi\in\varpi\text{End}(\mathbb{H}^{(m)}_{n-1})$. Now for any $\psi\in\text{End}(\mathbb{H}^{(m)}_{n-1})$, choose a power series $\tilde{\psi}\in(R_{m}/\mathfrak{m}_{R_{m}}^{n+1})[[X]]$ with trivial constant term such that $\tilde{\psi}\mod{\mathfrak{m}_{R_{m}}^{n}}=\psi$. The power series $\varpi\tilde{\psi}=[\varpi]_{\mathbb{H}_{n}^{(m)}}\circ\tilde{\psi}$ is a lift of $\varpi\psi=[\varpi]_{\mathbb{H}_{n-1}^{(m)}}\circ\psi$. We claim that $\varpi\tilde{\psi}\in\text{End}(\mathbb{H}^{(m)}_{n})$ and $(\varpi\psi\longmapsto\varpi\tilde{\psi}):\varpi\text{End}(\mathbb{H}^{(m)}_{n-1})\longrightarrow\text{End}(\mathbb{H}^{(m)}_{n})$ is a well-defined injective map. The proposition then follows from the claim. \\
\indent  First, let us see why $\varpi\tilde{\psi}$ defines an endomorphism of $\mathbb{H}^{(m)}_{n}$. Since $\psi\in\text{End}(\mathbb{H}^{(m)}_{n-1})$, we have\begin{align*}
 0&=\psi(X+_{\mathbb{H}^{(m)}_{n-1}}Y)-_{\mathbb{H}^{(m)}_{n-1}}\psi(X)-_{\mathbb{H}^{(m)}_{n-1}}\psi(Y)\\&=(\tilde{\psi}(X+_{\mathbb{H}^{(m)}_{n}}Y)-_{\mathbb{H}^{(m)}_{n}}\tilde{\psi}(X)-_{\mathbb{H}^{(m)}_{n}}\tilde{\psi}(Y))\mod{\mathfrak{m}_{R_{m}}^{n}}.
\end{align*} Thus all the coefficients of the power series $(\tilde{\psi}(X+_{\mathbb{H}^{(m)}_{n}}Y)-_{\mathbb{H}^{(m)}_{n}}\tilde{\psi}(X)-_{\mathbb{H}^{(m)}_{n}}\tilde{\psi}(Y))$ lie in $\mathfrak{m}^{n}_{R_{m}}/\mathfrak{m}^{n+1}_{R_{m}}$. Since $\varpi\in\mathfrak{m}_{R_{m}}$ and $(\mathfrak{m}_{R_{m}}^{n})^{k}\subseteq \mathfrak{m}_{R_{m}}^{n+1}$ for all integers $k>1$, we get $[\varpi]_{\mathbb{H}_{n}^{(m)}}\circ(\tilde{\psi}(X+_{\mathbb{H}^{(m)}_{n}}Y)-_{\mathbb{H}^{(m)}_{n}}\tilde{\psi}(X)-_{\mathbb{H}^{(m)}_{n}}\tilde{\psi}(Y))=0$. Consequently, $\varpi\tilde{\psi}(X+_{\mathbb{H}^{(m)}_{n}}Y)=\varpi\tilde{\psi}(X)+_{\mathbb{H}^{(m)}_{n}}\varpi\tilde{\psi}(Y)$. Similarly one shows that \begin{equation*}
 0=[\varpi]_{\mathbb{H}_{n}^{(m)}}\circ([a]_{\mathbb{H}^{(m)}_{n}}\circ\tilde{\psi}-_{\mathbb{H}^{(m)}_{n}}\tilde{\psi}\circ[a]_{\mathbb{H}^{(m)}_{n}})=[a]_{\mathbb{H}^{(m)}_{n}}\circ\varpi\tilde{\psi}-_{\mathbb{H}^{(m)}_{n}}\varpi\tilde{\psi}\circ[a]_{\mathbb{H}^{(m)}_{n}} 
\end{equation*} for all $a\in\mathfrak{o}$. Therefore $\varpi\tilde{\psi}\in\text{End}(\mathbb{H}^{(m)}_{n})$.\\
\indent  To see that the above map is well-defined, take another lift $\tilde{\psi}'$ of $\psi$ with trivial constant terms. Then $(\tilde{\psi}'-_{\mathbb{H}_{n}^{(m)}}\tilde{\psi})\mod{\mathfrak{m}_{R_{m}}^{n}}=\psi-_{\mathbb{H}_{n-1}^{(m)}}\psi=0$. Thus $[\varpi]_{\mathbb{H}_{n}^{(m)}}\circ(\tilde{\varphi}'-_{\mathbb{H}_{n}^{(m)}}\tilde{\varphi})=0$ as above. Hence $\varpi\tilde{\psi}'=\varpi\tilde{\psi}$. Finally, the injectivity is clear because $\varpi\tilde{\psi_{1}}=\varpi\tilde{\psi_{2}}$ implies $\varpi\psi_{1}=\varpi\psi_{2}$ after reduction modulo $\mathfrak{m}_{R_{m}}^{n}$.
\end{proof}
\begin{theorem}\label{ctsthm} For all $n\geq 0$, $m\geq 0$, the induced action of $\Gamma_{n+m}$ on $R_{m}/\mathfrak{m}_{R_{m}}^{n+1}$ is trivial. Thus the map $((\gamma,f)\mapsto \gamma(f)):\Gamma\times R_{m}\longrightarrow R_{m}$ is continuous where the left hand side carries the product topology. 
\end{theorem}
\begin{proof}
Let $n$ and $m$ be arbitrary non-negative integers. Let $\gamma\in\Gamma_{n+m}$ and  pr$^{(m)}_{n}:R_{m}\longrightarrow R_{m}/\mathfrak{m}_{R_{m}}^{n+1}$ denote the natural projection. Consider the level-$m$-structure $\phi^{(m)}_{n}:=\textnormal{pr}^{(m)}_{n}\circ\phi^{(m)}$ on $\mathbb{H}_{n}^{(m)}$ and consider the deformation $(\mathbb{H}_{n}^{(m)},\rho^{(m)}\circ\gamma^{-1},\phi^{(m)}_{n})$ of $H_{0}$ to $R_{m}/\mathfrak{m}_{R_{m}}^{n+1}$ with  this level-$m$-structure. Let $\gamma_{n}^{(m)}:R_{m}\longrightarrow R_{m}/\mathfrak{m}_{R_{m}}^{n+1}$ denote the unique ring homomorphism for which there exists an isomorphism $[\gamma_{n}^{(m)}]:(\gamma^{(m)}_{n})_{*}(\mathbb{H}^{(m)},\rho^{(m)},\phi^{(m)})\iso(\mathbb{H}_{n}^{(m)},\rho^{(m)}\circ\gamma^{-1},\phi^{(m)}_{n})$. Note that also the ring homomorphism pr$^{(m)}_{n}\circ\gamma:R_{m}\longrightarrow R_{m}/\mathfrak{m}_{m}^{n+1}$ admits an isomorphism of deformations \begin{align*}
(\text{pr}^{(m)}_{n}\circ\gamma)_{*}(\mathbb{H}^{(m)},\rho^{(m)},\phi^{(m)})&=(\text{pr}^{(m)}_{n})_{*}(\gamma_{*}(\mathbb{H}^{(m)},\rho^{(m)},\phi^{(m)}))\\&\cong(\mathbb{H}_{n}^{(m)},\rho^{(m)}\circ\gamma^{-1},\phi^{(m)}_{n}).
\end{align*} Therefore by uniqueness, we have pr$^{(m)}_{n}\circ\gamma=\gamma^{(m)}_{n}$ and  $[\gamma^{(m)}_{n}]=[\gamma]\mod{\mathfrak{m}_{R_{m}}^{n+1}}$.\\
\indent  Since the map $(\sigma\mapsto\rho^{(m)}\circ\sigma\circ(\rho^{(m)})^{-1})$ is an isomorphism End$(H_{0})\iso\textnormal{End}(\mathbb{H}^{(m)}_{0})$ of $\mathfrak{o}$-algebras, Proposition~\ref{subringprop} shows that $\rho^{(m)}\circ\gamma^{-1}\circ(\rho^{(m)})^{-1}\in 1+\varpi^{m}\textnormal{End}(\mathbb{H}^{(m)}_{n})\subseteq\textnormal{Aut}(\mathbb{H}^{(m)}_{n})$. We claim that $(\rho^{(m)}\circ\gamma^{-1}\circ(\rho^{(m)})^{-1})\circ\phi^{(m)}_{n}=\phi^{(m)}_{n}$. Write $\rho^{(m)}\circ\gamma^{-1}\circ(\rho^{(m)})^{-1}=1+\varepsilon\varpi^{m}$ for some $\varepsilon\in\textnormal{End}(\mathbb{H}^{(m)}_{n})$ and let $\alpha\in(\frac{\varpi^{-m}\mathfrak{o}}{\mathfrak{o}})^{h}$ be arbitrary. Then \begin{align*}
(\rho^{(m)}\circ\gamma^{-1}\circ(\rho^{(m)})^{-1})(\phi^{(m)}_{n}(\alpha))&=(1+\varepsilon\varpi^{m})(\phi^{(m)}_{n}(\alpha))\\&=\phi^{(m)}_{n}(\alpha)+_{\mathbb{H}^{(m)}_{n}}\varepsilon(\varpi^{m}(\phi^{(m)}_{n}(\alpha)))\\&=\phi^{(m)}_{n}(\alpha)+_{\mathbb{H}^{(m)}_{n}}\varepsilon(\phi^{(m)}_{n}(\varpi^{m}\alpha))\\&=\phi^{(m)}_{n}(\alpha)+_{\mathbb{H}^{(m)}_{n}}\varepsilon(\phi^{(m)}_{n}(0))\\&=\phi^{(m)}_{n}(\alpha).
\end{align*} Therefore, the automorphism $\rho^{(m)}\circ\gamma^{-1}\circ(\rho^{(m)})^{-1}$ of $\mathbb{H}^{(m)}_{n}$ defines an isomorphism of deformations \begin{align*}
&(\mathbb{H}^{(m)}_{n},\rho^{(m)},\phi^{(m)}_{n})\\&\cong(\mathbb{H}^{(m)}_{n},(\rho^{(m)}\circ\gamma^{-1}\circ(\rho^{(m)})^{-1})\circ\rho^{(m)},(\rho^{(m)}\circ\gamma^{-1}\circ(\rho^{(m)})^{-1})\circ\phi^{(m)}_{n})\\&=(\mathbb{H}_{n}^{(m)},\rho^{(m)}\circ\gamma^{-1},\phi^{(m)}_{n}). 
\end{align*} However, $(\mathbb{H}^{(m)}_{n},\rho^{(m)},\phi^{(m)}_{n})=\textnormal{(pr}^{(m)}_{n})_{*}(\mathbb{H}^{(m)},\rho^{(m)},\phi^{(m)})$. By uniqueness again, we have pr$^{(m)}_{n}=\textnormal{pr}^{(m)}_{n}\circ\gamma=\gamma^{(m)}_{n}$. This implies that $\Gamma_{n+m}$ acts trivially on $R_{m}/\mathfrak{m}_{R_{m}}^{n+1}$ and $[\gamma]\mod{\mathfrak{m}_{R_{m}}^{n+1}}=\rho^{(m)}\circ\gamma^{-1}\circ(\rho^{(m)})^{-1}$.
\end{proof}
\indent  The $R_{m}$-module $\textnormal{Lie}(\mathbb{H}^{(m)})^{\otimes s}$ is complete and Hausdorff for the $\mathfrak{m}_{R_{m}}$-adic topology because it is free of finite rank. By the semi-linearity of the $\Gamma$-action, the $R_{m}$-submodules $\mathfrak{m}_{R_{m}}^{n}\textnormal{Lie}(\mathbb{H}^{(m)})^{\otimes s}$ are $\Gamma$-stable for any non-negative integer $n$.
\begin{theorem}\label{ctsthm2}
Let $s$, $n$, $m$ be integers with $n\geq 0$ and $m\geq 0$. The induced action of $\Gamma_{2n+m+1}$ on $\textnormal{Lie}(\mathbb{H}^{(m)})^{\otimes s}/\mathfrak{m}_{R_{m}}^{n+1}\textnormal{Lie}(\mathbb{H}^{(m)})^{\otimes s}$ is trivial. Thus the map $((\gamma,\delta)\mapsto \gamma(\delta)):\Gamma\times \textnormal{Lie}(\mathbb{H}^{(m)})^{\otimes s}\longrightarrow \textnormal{Lie}(\mathbb{H}^{(m)})^{\otimes s}$ is continuous where the left hand side carries the product topology.
\end{theorem}
\begin{proof}
If we assume the assertion to be true for $s=1$, then by the definition of the action, it is easy to see that it holds for all positive $s$. On the other hand, let $\overline{\varphi}\in\textnormal{Lie}(\mathbb{H}^{(m)})^{\otimes -1}/\mathfrak{m}_{R_{m}}^{n+1}\textnormal{Lie}(\mathbb{H}^{(m)})^{\otimes -1}$ and $\gamma\in\Gamma_{2n+m+1}$. Then by assumption, $\gamma(\delta)-\delta\in\mathfrak{m}_{R_{m}}^{n+1}\text{Lie}(\mathbb{H}^{(m)})$. Write $\gamma^{-1}(\delta)=\delta+\sum_{i=1}^{r}\alpha_{i}\eta_{i}$ with $\alpha_{i}\in\mathfrak{m}_{R_{m}}^{n+1}$ and $\eta_{i}\in\text{Lie}(\mathbb{H}^{(m)})$. Then \begin{align*}
(\varphi-\gamma(\varphi))(\delta)=\varphi(\delta)-\gamma(\varphi)(\delta)&=\varphi(\delta)-\gamma(\varphi(\gamma^{-1}(\delta)))\\&=\varphi(\delta)-\gamma(\varphi(\delta+\sum_{i=1}^{r}\alpha_{i}\eta_{i}))\\&=\varphi(\delta)-\gamma(\varphi(\delta))-\sum_{i=1}^{r}\gamma(\alpha_{i})\gamma(\varphi(\eta_{i})).
\end{align*}  Since $2n+m+1\geq n+m$, by Theorem \ref{ctsthm}, we have $\varphi(\delta)-\gamma(\varphi(\delta))\in\mathfrak{m}_{R_{m}}^{n+1}$. Also $\gamma(\alpha_{i})\in\mathfrak{m}_{R_{m}}^{n+1}$. Therefore $
(\varphi-\gamma(\varphi))(\delta)\in\mathfrak{m}_{R_{m}}^{n+1}$. If $\delta_{0}$ is a basis of $\text{Lie}(\mathbb{H}^{(m)})$ over $R_{m}$, and $\psi\in\textnormal{Lie}(\mathbb{H}^{(m)})^{\otimes -1}$ is defined by $\psi(\delta_{0})=1$, then $\varphi-\gamma(\varphi)=(\varphi-\gamma(\varphi))(\delta_{0})\psi\in\mathfrak{m}_{R_{m}}^{n+1}\textnormal{Lie}(\mathbb{H}^{(m)})^{\otimes -1}$. Thus $\overline{\varphi}=\gamma(\overline{\varphi})$. A similar argument like this can be used to show that the assertion is true for all higher negative $s$. Hence it is sufficient to prove the theorem for $s=1$. \\  
\indent  Let $\gamma\in\Gamma_{2n+m+1}$. By identifying $\textnormal{Lie}(\mathbb{H}^{(m)})/\mathfrak{m}_{R_{m}}^{n+1}\textnormal{Lie}(\mathbb{H}^{(m)})=\textnormal{Lie}(\mathbb{H}^{(m)}_{n})$, Theorem~\ref{ctsthm} and its proof show that the map $\gamma\mod{\mathfrak{m}_{R_{m}}^{n+1}}:\textnormal{Lie}(\mathbb{H}^{(m)}_{n})\longrightarrow\textnormal{Lie}(\mathbb{H}^{(m)}_{n})$ is given by $\textnormal{Lie}(\rho^{(m)}\circ\gamma^{-1}\circ(\rho^{(m)})^{-1})$ where $\rho^{(m)}\circ\gamma^{-1}\circ(\rho^{(m)})^{-1}\in 1+\varpi^{2n+m+1}\textnormal{End}(\mathbb{H}^{(m)}_{0})$ $\subseteq 1+\varpi^{n+m+1}\textnormal{End}(\mathbb{H}^{(m)}_{n})$. Therefore it suffices to show that the natural action of $1+\varpi^{n+m+1}\textnormal{End}(\mathbb{H}^{(m)}_{n})\subset\textnormal{End}(\mathbb{H}^{(m)}_{n})$ on $\textnormal{Lie}(\mathbb{H}_{n}^{(m)})$ is trivial. However, if $\varphi\in\textnormal{End}(\mathbb{H}^{(m)}_{n})$ and $\delta\in\textnormal{Lie}(\mathbb{H}^{(m)}_{n})$, then \begin{align*}
(\textnormal{Lie}(1+\varpi^{n+m+1}\varphi)(\delta))(\overline{X})&=\delta(\overline{(1+\varpi^{n+m+1}\varphi)(X)})\\&=\delta(\overline{X+_{\mathbb{H}^{(m)}_{n}}\varpi^{n+m+1}\varphi(X)})\\&=\delta(\overline{X+\varpi^{n+m+1}\varphi(X)})\\&=\delta(\overline{X})\end{align*} because $\varpi^{n+m+1}\in\mathfrak{m}_{R_{m}}^{n+1}$.
\end{proof}
\begin{remark}\label{iwasawa_alg_action} The $\Gamma$-action on $\textnormal{Lie}(\mathbb{H}^{(m)})^{\otimes s}$ gives rise to an action of the group ring $\breve{\mathfrak{o}}[\Gamma]$ on $\textnormal{Lie}(\mathbb{H}^{(m)})^{\otimes s}$. By Theorem \ref{ctsthm2}, the induced action of $\breve{\mathfrak{o}}[\Gamma]$ on \linebreak $\textnormal{Lie}(\mathbb{H}^{(m)})^{\otimes s}/\mathfrak{m}_{R_{m}}^{n+1}\textnormal{Lie}(\mathbb{H}^{(m)})^{\otimes s}$ factors through $(\breve{\mathfrak{o}}/\varpi^{n+1}\breve{\mathfrak{o}})[\Gamma/\Gamma_{2n+m+1}]$ such that the following diagram with the horizontal action maps and the vertical reduction maps commutes for all $n$.
\begin{displaymath}
\xymatrix@C=70pt{
\frac{\breve{\mathfrak{o}}}{(\varpi^{n+1})}\big[\frac{\Gamma}{\Gamma_{2n+m+1}}\big]\times\frac{\textnormal{Lie}(\mathbb{H}^{(m)})^{\otimes s}}{\mathfrak{m}_{R_{m}}^{n+1}\textnormal{Lie}(\mathbb{H}^{(m)})^{\otimes s}}\ar[d]\ar[r]&\frac{\textnormal{Lie}(\mathbb{H}^{(m)})^{\otimes s}}{\mathfrak{m}_{R_{m}}^{n+1}\textnormal{Lie}(\mathbb{H}^{(m)})^{\otimes s}}\ar[d]\\
\frac{\breve{\mathfrak{o}}}{(\varpi^{n})}\big[\frac{\Gamma}{\Gamma_{2(n-1)+m+1}}\big]\times\frac{\textnormal{Lie}(\mathbb{H}^{(m)})^{\otimes s}}{\mathfrak{m}_{R_{m}}^{n}\textnormal{Lie}(\mathbb{H}^{(m)})^{\otimes s}}\ar[r]&\frac{\textnormal{Lie}(\mathbb{H}^{(m)})^{\otimes s}}{\mathfrak{m}_{R_{m}}^{n}\textnormal{Lie}(\mathbb{H}^{(m)})^{\otimes s}}}
\end{displaymath}
Taking projective limits over $n$, we obtain an action of the Iwasawa algebra $\breve{\mathfrak{o}}[[\Gamma]]$ on $\textnormal{Lie}(\mathbb{H}^{(m)})^{\otimes s}$ that extends the action of $\Gamma$. 
\end{remark}
\subsection{Rigidification and the equivariant vector bundles}
\label{Rigidification and the equivariant vector bundles}
\indent Berthelot's rigidification functor associates to every locally Noetherian adic formal $\breve{\mathfrak{o}}$-scheme whose reduction is a scheme locally of finite type over Spec$(\overline{k})$, a rigid $\breve{K}$-analytic space (cf. \cite{jong}, Section 7). For an affine formal $\breve{\mathfrak{o}}$-scheme $\text{Spf}(A)$, there is a bijection between the closed points of its generic fibre $\text{Spec}(A\otimes_{\breve{\mathfrak{o}}}\breve{K})$ and the points of the associated rigid analytic space. Let's denote by $X^{\textnormal{rig}}_{m}$ the rigidification of the affine formal $\breve{\mathfrak{o}}$-scheme $X_{m}=\textnormal{Spf}(R_{m})$ under Berthelot's functor, and by $R^{\textnormal{rig}}_{m}:=\mathcal{O}_{X^{\textnormal{rig}}_{m}}(X^{\textnormal{rig}}_{m})$ the $\breve{K}$-algebra of the global rigid analytic functions on $X^{\textnormal{rig}}_{m}$.\\
\indent  By functoriality, $X^{\text{rig}}_{m}$ and $R^{\textnormal{rig}}_{m}$ carry commuting (left) actions of $\Gamma$ and $G_{0}$, thus an action of the product group $\Gamma\times G_{0}$, and the $G_{0}$-action factors through $G_{0}/G_{m}$. For $m'\geq m\geq 0$, let \begin{equation*}
\pi_{m',m}:X^{\text{rig}}_{m'}\longrightarrow X^{\text{rig}}_{m}
\end{equation*} denote the morphism of rigid analytic spaces induced by Part 2, Theorem \ref{fundthm} and by functoriality. It follows from the properties of the rigidification functor that the morphism $\pi_{m',m}$ is a finite \'{e}tale Galois covering with Galois group $G_{m}/G_{m'}$ (cf. \cite{jong}, Section 7). Consequently, the ring extension $R^{\text{rig}}_{m'}\big|R^{\text{rig}}_{m}$ is finite Galois with Galois group $G_{m}/G_{m'}$. We note that all covering morphisms are $(\Gamma\times G_{0})$-equivariant.\\
\indent  It follows from the isomorphism $R_{0}\cong\breve{\mathfrak{o}}[[u_{1},\ldots,u_{h-1}]]$ that $X^{\textnormal{rig}}_{0}$ is isomorphic to the $(h-1)$-dimensional rigid analytic open unit polydisc over $\breve{K}$, and the isomorphism $R_{0}\cong\breve{\mathfrak{o}}[[u_{1},\ldots,u_{h-1}]]$ extends to an isomorphism \begin{equation}
\text{$R^{\textnormal{rig}}_{0}\cong\Big\lbrace\sum_{\alpha\in\mathbb{N}_{0}^{h-1}}c_{\alpha}u^{\alpha}\mid c_{\alpha}\in\breve{K}$ and $\lim_{|\alpha|\to\infty}|c_{\alpha}|r^{|\alpha|}=0$ for all $0<r<1$\Big\rbrace}
\end{equation} of $\breve{K}$-algebras. This allows us to view $R^{\textnormal{rig}}_{0}$ as a topological $\breve{K}$-Fr\'{e}chet algebra whose topology is defined by the family of norms $\|\cdot\|_{l}$, given by \begin{equation*}
\Big\|\sum_{\alpha\in\mathbb{N}_{0}^{h-1}}c_{\alpha}u^{\alpha}\Big\|_{l}:=\sup_{\alpha\in\mathbb{N}_{0}^{h-1}}\lbrace\vert c_{\alpha}\vert\vert \varpi\vert^{\vert \alpha\vert /l}\rbrace
\end{equation*} for any positive integer $l$. Let $R^{\textnormal{rig}}_{0,l}$ be the completion of $R^{\textnormal{rig}}_{0}$ with respect to the norm $\|\cdot\|_{l}$. Then \begin{equation*}
R^{\text{rig}}_{0,l}\cong\Big\lbrace\sum_{\alpha\in\mathbb{N}_{0}^{h-1}}c_{\alpha}u^{\alpha}\big|c_{\alpha}\in\breve{K},\lim_{|\alpha|\to\infty}|c_{\alpha}||\varpi|^{|\alpha|/l}=0\Big\rbrace
\end{equation*} is the $\breve{K}$-Banach algebra of rigid analytic functions on the affinoid subdomain \begin{equation*}
\text{$\mathbb{B}_{l}:=\lbrace x\in X^{\textnormal{rig}}_{0}\mid |u_{i}(x)|\leq |\varpi|^{1/l}$ for all $1\leq i\leq h-1\rbrace$}
\end{equation*} of $X^{\textnormal{rig}}_{0}$. Further, $R^{\textnormal{rig}}_{0}\cong\varprojlim_{l}R^{\textnormal{rig}}_{0,l}$ is the topological projective limit of the $\breve{K}$-Banach algebras $R^{\textnormal{rig}}_{0,l}$.\\
\indent  Since $R_{0}$ is a local ring, the finite flat $R_{0}$-module $R_{m}$ is free by \cite{matcrt}, Theorem 7.10, and has rank $r=|G_{0}/G_{m}|$. Let $R^{\text{rig}}_{m,l}$ denote the affinoid $\breve{K}$-algebra of the rigid analytic functions on the affinoid subdomain $\mathbb{B}_{m,l}:=\pi_{m,0}^{-1}(\mathbb{B}_{l})$ of $X^{\text{rig}}_{m}$. Then by \cite{jong}, Lemma 7.2.2, we have \begin{equation}\label{rigidification@l_is_a_basechange}
R^{\text{rig}}_{m,l}\cong R_{m}\otimes_{R_{0}}R^{\text{rig}}_{0,l}
\end{equation}  as $R_{m}|R_{0}$ is finite. Let us fix a basis $\lbrace e_{1},\cdots, e_{r}\rbrace$ of $R_{m}$ over $R_{0}$ and view it as an $R^{\textnormal{rig}}_{0, l}$-basis of $R^{\textnormal{rig}}_{m,l}=R_{m}\otimes_{R_{0}}R^{\textnormal{rig}}_{0,l}$. The next lemma shows that $R^{\textnormal{rig}}_{m,l}$ is a $\breve{K}$-Banach algebra with respect to the norm $\|(f_{1}e_{1}+\cdots +f_{r}e_{r})\|_{l}:=\max_{1\leq i\leq r}\lbrace\|f_{i}\|_{l}\rbrace$, where $f_{i}\in R^{\textnormal{rig}}_{0,l}$ for all $i$, by showing that it is indeed an algebra norm.   
\begin{lemma}\label{lnorm_is_algebra_norm} Let $f=f_{1}e_{1}+\cdots +f_{r}e_{r}$, $g=g_{1}e_{1}+\cdots +g_{r}e_{r}\in R^{\textnormal{rig}}_{m,l}$. Then $\| fg\|_{l}\leq\| f\|_{l}\| g\|_{l}$.
\end{lemma}
\begin{proof} 
Let $e_{i}e_{j}=\sum_{k=1}^{r}a_{ijk}e_{k}$ for all $1\leq i,j \leq r$. Note that $a_{ijk}\in R_{0}=\breve{\mathfrak{o}}[[u_{1},\dots ,u_{h-1}]]$ and thus $\| a_{ijk}\|_{l}\leq 1$ for all $1\leq i,j,k \leq r$. Also $\|\cdot\|_{l}$ is multiplicative on $R^{\textnormal{rig}}_{0,l}$. Therefore \begin{align*}
\| fg\|_{l}=\max_{1\leq k\leq r}\Big\lbrace\Big\|\sum_{1\leq i,j\leq r}f_{i}g_{j}a_{ijk}\Big\|_{l}\Big\rbrace &\leq\max_{1\leq k\leq r}\Big\lbrace\max_{1\leq i,j\leq r}\| f_{i}g_{j}a_{ijk}\|_{l}\Big\rbrace\\&\leq\max_{1\leq i,j\leq r}\lbrace\| f_{i}\|_{l}\| g_{j}\|_{l}\rbrace\leq\| f\|_{l}\| g\|_{l}.
\end{align*}
\end{proof}
\indent  It then follows from \cite{bgr}, (6.1.3), Proposition 2 that the affinoid topology on $R^{\text{rig}}_{m,l}$ coincides with Banach topology given by the aforementioned norm $\|\cdot\|_{l}$. The natural maps $R^{\textnormal{rig}}_{m,l+1}\longrightarrow R^{\textnormal{rig}}_{m,l}$ induced from $R^{\textnormal{rig}}_{0,l+1}\longrightarrow R^{\textnormal{rig}}_{0,l}$ endow the projective limit \begin{equation*}
 R^{\textnormal{rig}}_{m}\cong\varprojlim_{l}R^{\textnormal{rig}}_{m,l}
\end{equation*} with the structure of a $\breve{K}$-Fr\'{e}chet algebra. Indeed, this projective limit is isomorphic to the $\breve{K}$-algebra of global rigid analytic functions on $X^{\textnormal{rig}}_{m}$ by \cite{jong}, Lemma 7.2.2. Thus we have \begin{equation}\label{rigidification_is_a_basechange}
R^{\textnormal{rig}}_{m}\cong R_{m}\otimes_{R_{0}}R^{\textnormal{rig}}_{0}
\end{equation}  as $R_{m}$ is a finite free $R_{0}$-module, and $R^{\textnormal{rig}}_{m,l}$ can be viewed as the Banach completion of $R^{\textnormal{rig}}_{m}$ with respect to the norm $\|\cdot\|_{l}$ defined as before by $\|(f_{1}e_{1}+\cdots +f_{r}e_{r})\|_{l}:=\max_{1\leq i\leq r}\lbrace\|f_{i}\|_{l}\rbrace$, with $f_{i}\in R^{\textnormal{rig}}_{0}$. \\
\indent  The $\Gamma$-action on $X^{\text{rig}}_{0}$ stabilizes the affinoid subdomains $\mathbb{B}_{l}$ for all positive integers $l$. Indeed, since $\gamma(u_{i})$ belongs to the maximal ideal $(\varpi,u_{1},\dots ,u_{h-1})$ of $R_{0}$, $\|\gamma(u_{i})\|_{l}\leq|\varpi|^{1/l}$ for all $1\leq i\leq h-1$. This implies that $\|\gamma(f)\|_{l}\leq\|f\|_{l}$ for all $f\in R^{\text{rig}}_{0}$. Thus the $\Gamma$-action on $R^{\text{rig}}_{0}$ extends to its completion $R^{\text{rig}}_{0,l}$ for all positive integers $l$. As a consequence, the affinoid subdomains $\mathbb{B}_{m,l}$ of $X^{\text{rig}}_{m}$ are stable under the $(\Gamma\times G_{0})$-action for all $m$ and $l$, and the isomorphism (\ref{rigidification@l_is_a_basechange}) is $(\Gamma\times G_{0})$-equivariant for the diagonal $(\Gamma\times G_{0})$-action on the right. Similarly the isomorphism (\ref{rigidification_is_a_basechange}) is also $(\Gamma\times G_{0})$-equivariant for the diagonal $(\Gamma\times G_{0})$-action on the right.
\begin{remark} By \cite{bgr}, (6.1.3), Theorem 1, the $\breve{K}$-algebra automorphism of an affinoid $\breve{K}$-algebra $R^{\text{rig}}_{m,l}$ corresponding to $(\gamma,g)\in\Gamma\times G_{0}$ is automatically continuous for its $\breve{K}$-Banach topology defined by the norm $\|\cdot\|_{l}$. Since the $\breve{K}$-Fr\'{e}chet topology of $R^{\text{rig}}_{m}$ is given by the family of norms $\|\cdot\|_{l}$, $l\in\mathbb{N}$, the group $\Gamma\times G_{0}$ acts on $R^{\text{rig}}_{m}$ by continuous $\breve{K}$-algebra automorphisms for all $m\geq 0$.
\end{remark} 
\indent  Now recall from \cite{gh}, Section 15 that a $\Gamma$-equivariant vector bundle $\mathcal{M}$ on the formal scheme $X_{m}$ is a locally free $\mathcal{O}_{X_{m}}$-module $\mathcal{M}$ of finite rank equipped with a (left) $\Gamma$-action that is compatible with the $\Gamma$-action on $X_{m}$. Since $X_{m}$ is formally affine, a $\Gamma$-equivariant vector bundle $\mathcal{M}$ on $X_{m}$ is completely determined by its global sections $\mathcal{M}(X_{m})$. Hence, for all $s\in\mathbb{Z}$ and $m\geq 0$, the free $R_{m}$-module $\textnormal{Lie}(\mathbb{H}^{(m)})^{\otimes s}$ of rank 1 equipped with a semilinear $\Gamma$-action gives rise to a $\Gamma$-equivariant line bundle \begin{equation*}
\mathcal{M}^{s}_{m}:=\mathcal{L}\textnormal{ie}(\mathbb{H}^{(m)})^{\otimes s}
\end{equation*} on $X_{m}$. Its rigidification $(\mathcal{M}^{s}_{m})^{\textnormal{rig}}$ is a locally free $\mathcal{O}_{X^{\textnormal{rig}}_{m}}$-module of rank 1 by \cite{jong}, 7.1.11. Let \begin{equation*}
 M^{s}_{m}:=(\mathcal{M}^{s}_{m})^{\textnormal{rig}}(X^{\textnormal{rig}}_{m})
\end{equation*} denote its global sections. Because of the fact that $X_{m}$ is affine, the natural map \begin{equation}\label{G_eqv_iso_eqn1}
R^{\textnormal{rig}}_{m}\otimes_{R_{m}}\textnormal{Lie}(\mathbb{H}^{(m)})^{\otimes s}\longrightarrow M^{s}_{m}
\end{equation} is an isomorphism.
By functoriality, $\Gamma$ acts on $(\mathcal{M}^{s}_{m})^{\textnormal{rig}}$ in such a way that the map (\ref{G_eqv_iso_eqn1}) is $\Gamma$-equivariant for the diagonal $\Gamma$-action on the left and for the $\Gamma$-action induced by functoriality on the right. In particular, the $\Gamma$-action on $M^{s}_{m}$ is semilinear for its action on $R^{\text{rig}}_{m}$, and thus $(\mathcal{M}^{s}_{m})^{\text{rig}}$ is a rigid $\Gamma$-equivariant line bundle on $X^{\text{rig}}_{m}$. In a similar fashion, it can be seen that $(\mathcal{M}^{s}_{m})^{\text{rig}}$ is a rigid $G_{0}$-equivariant line bundle on $X^{\text{rig}}_{m}$, and the actions of $\Gamma$ and $G_{0}$ commute (cf. Remark \ref{Gamma_and_G0_actions_commute}). By functoriality again, the $G_{0}$-action on $(\mathcal{M}^{s}_{m})^{\text{rig}}$ factors through the quotient group $G_{0}/G_{m}$.\\
\indent  For all $s,m$ and $l$, set $M^{s}_{m,l}:=(\mathcal{M}^{s}_{m})^{\text{rig}}(\mathbb{B}_{m,l})$. Then $M^{s}_{m,l}$ is a free $R^{\textnormal{rig}}_{m,l}$-module of rank 1 for which the natural $R^{\textnormal{rig}}_{m,l}$-linear map \begin{equation}\label{G_eqv_iso_eqn3}
R^{\textnormal{rig}}_{m,l}\otimes_{R_{m}}\textnormal{Lie}(\mathbb{H}^{(m)})^{\otimes s}\longrightarrow M^{s}_{m,l}
\end{equation} is an isomorphism (cf. \cite{jong}, 7.1.11), and is $(\Gamma\times G_{0})$-equivariant for the diagonal $(\Gamma\times G_{0})$-action on the left. Endowing $M^{s}_{m}$ and $M^{s}_{m,l}$ with the natural topologies of finitely generated modules over $R^{\textnormal{rig}}_{m}$ and  $R^{\textnormal{rig}}_{m,l}$ respectively, makes them a $\breve{K}$-Fr\'{e}chet space and a $\breve{K}$-Banach space respectively. One then has a topological isomorphism \begin{equation*}
M^{s}_{m}\cong\varprojlim_{l}M^{s}_{m,l}
\end{equation*} for the projective limit topology on the right, and the group $\Gamma\times G_{0}$ acts on $M^{s}_{m}$ by continuous $\breve{K}$-vector space automorphisms for all $s$ and $m$.
\section{Locally analytic representations at higher levels}
\indent  We saw in the previous section that the group $\Gamma=\text{Aut}(H_{0})$ acts on the Fr\'{e}chet space  $M^{s}_{m}$ of the global sections of the rigid $\Gamma$-equivariant line bundle $(\mathcal{M}^{s}_{m})^{\text{rig}}$ over $X^{\text{rig}}_{m}$ by continuous vector space automorphisms. The goal of this section is to show that the strong topological $\breve{K}$-linear dual $(M^{s}_{m})'_{b}$ of $M^{s}_{m}$ with the induced $\Gamma$-action is a locally $K$-analytic representation of $\Gamma$ for all $s\in\mathbb{Z}$ and levels $m\geq 0$.\footnote{In the notation $(M^{s}_{m})'_{b}$, the superscript $'$ indicates the continuous linear dual of the topological vector space $M^{s}_{m}$, while the subscript $b$ stands for bounded and implies that $(M^{s}_{m})'_{b}$ is equipped with the topology of bounded convergence, i.e., the strong topology.} Recall from \cite{stladist}, Section 3 that a locally $K$-analytic representation $V$ of $\Gamma$ (over $\breve{K}$) is a barrelled Hausdorff locally convex $\breve{K}$-vector space equipped with a $\Gamma$-action by continuous linear endomorphisms such that for each $v\in V$, the orbit map $\rho_{v}=(g\mapsto g(v))\in C^{an}(\Gamma, V)$, i.e., $\rho_{v}$ is a $V$-valued locally $K$-analytic function on $\Gamma$.\\
\indent As stated in the introduction, to show the local $K$-analyticity, we are going to make use of the Gross-Hopkins' period morphism $\Phi$ and of the fundamental domain $D$. The next subsection is intended to give a brief overview of these two technical tools.
\subsection{The period morphism and the Gross-Hopkins fundamental domain}
\label{Phi&D}
Let $\mathbb{E}^{(m)}$ denote \emph{the universal additive extension} of the universal formal $\mathfrak{o}$-module $\mathbb{H}^{(m)}$. It sits in the exact sequence \begin{equation}\label{uniaddextses}
0\longrightarrow \textnormal{Lie}((\mathbb{H}^{(m)})')\xrightarrow{\alpha^{(m)}}\textnormal{Lie}(\mathbb{E}^{(m)})\xrightarrow{\beta^{(m)}}\textnormal{Lie}(\mathbb{H}^{(m)})\longrightarrow 0
\end{equation} where $(\mathbb{H}^{(m)})'=\mathbb{G}_{a}\otimes_{R_{m}}\text{Hom}_{R_{m}}(\text{Ext}(\mathbb{H}^{(m)},\mathbb{G}_{a}), R_{m})$ is the additive formal $\mathfrak{o}$-module of dimension $h-1$ associated with the free $R_{m}$-module dual to the module $\text{Ext}(\mathbb{H}^{(m)},\mathbb{G}_{a})$ of extensions of $\mathbb{H}^{(m)}$ by $\mathbb{G}_{a}$ (cf. \cite{gh}, Section 11). The universality induces commuting semi-linear actions of $\Gamma$ and $G_{0}$ on $\textnormal{Lie}((\mathbb{H}^{(m)})')$ and on $\textnormal{Lie}(\mathbb{E}^{(m)})$ for which the maps $\alpha^{(m)}$ and $\beta^{(m)}$ are equivariant, and the $G_{0}$-action factors through $G_{0}/G_{m}$. By rigidification, the short exact sequence (\ref{uniaddextses}) gives rise to an exact sequence \begin{equation*}
 0\longrightarrow(\mathcal{L}\text{ie}((\mathbb{H}^{(m)})')^{\otimes s})^{\text{rig}}\longrightarrow(\mathcal{L}\text{ie}(\mathbb{E}^{(m)})^{\otimes s})^{\text{rig}}\longrightarrow(\mathcal{M}^{s}_{m})^{\text{rig}}\longrightarrow 0
\end{equation*} of the corresponding rigid $(\Gamma\times G_{0})$-equivariant line bundles on $X^{\text{rig}}_{m}$ for all non-negative $s$, and for negative $s$ in the opposite direction. Since $X_{m}$ is an affine formal scheme, by taking global sections, we get an exact sequence \begin{equation}\label{exact_seq_of_global_sections_of_eqv_bundles}
0\longrightarrow R^{\text{rig}}_{m}\otimes_{R_{m}}\text{Lie}((\mathbb{H}^{(m)})')^{\otimes s}\longrightarrow R^{\text{rig}}_{m}\otimes_{R_{m}}\text{Lie}(\mathbb{E}^{(m)})^{\otimes s}\longrightarrow M^{s}_{m}\longrightarrow 0
\end{equation} of $\breve{K}$-linear $(\Gamma\times G_{0})$-representations for $s\geq 0$, and in the opposite direction for $s<0$.
\begin{proposition}\label{Lie(E0)_is_generically_flat} The $\Gamma$-equivariant line bundle $\mathcal{L}\textnormal{ie}(\mathbb{E}^{(0)})$ is generically flat, i.e., there exists a basis $\lbrace c_{0},c_{1},\ldots ,c_{h-1}\rbrace$ of $R^{\textnormal{rig}}_{0}\otimes_{R_{0}}\textnormal{Lie}(\mathbb{E}^{(0)})$ over $R^{\textnormal{rig}}_{0}$ such that the $\breve{K}$-subspace of $R^{\textnormal{rig}}_{0}\otimes_{R_{0}}\textnormal{Lie}(\mathbb{E}^{(0)})$ spanned by $c_{i}$'s is $\Gamma$-stable. Let $B_{h}\otimes_{K_{h}}\breve{K}$ be the $h$-dimensional $\breve{K}$-linear $\Gamma$-representation where the action of $\Gamma\cong\mathfrak{o}_{B_{h}}^{\times}$ is given by left multiplication. Then we have an isomorphism 
\begin{equation}\label{Lie(E0)_is_generically_flat_iso}
R^{\textnormal{rig}}_{0}\otimes_{R_{0}}\textnormal{Lie}(\mathbb{E}^{(0)})\cong R^{\textnormal{rig}}_{0}\otimes_{\breve{K}}(B_{h}\otimes_{K_{h}}\breve{K})
\end{equation} of $R^{\textnormal{rig}}_{0}[\Gamma]$-modules with $\Gamma$ acting diagonally on both sides.
\end{proposition}
\begin{proof} See \cite{gh}, Proposition 21.8 and Proposition 22.4. 
\end{proof} 
\begin{corollary}\label{Lie(Ems)_generically_flat}
The $(\Gamma\times G_{0})$-equivariant line bundle $\mathcal{L}\textnormal{ie}(\mathbb{E}^{(m)})^{\otimes s}$ is generically flat for all $m\geq 0$ and $s\in\mathbb{Z}$.
\end{corollary}
\begin{proof} 
Since $\mathbb{H}^{(m)}=\mathbb{H}^{(0)}\otimes_{R_{0}}R_{m}$, it follows from the universality of $\mathbb{E}^{(m)}$ that $\mathbb{E}^{(m)}=\mathbb{E}^{(0)}\otimes_{R_{0}}R_{m}$. The isomorphism $\text{Lie}(\mathbb{E}^{(m)})\cong R_{m}\otimes_{R_{0}}\text{Lie}(\mathbb{E}^{(0)})$ of $R_{m}[\Gamma\times G_{0}]$-modules gives rise to an isomorphism $R^{\text{rig}}_{m}\otimes_{R_{m}}\text{Lie}(\mathbb{E}^{(m)})\cong R^{\text{rig}}_{m}\otimes_{R_{0}}\text{Lie}(\mathbb{E}^{(0)})$ of $R^{\text{rig}}_{m}[\Gamma\times G_{0}]$-modules. Then using (\ref{rigidification_is_a_basechange}) and (\ref{Lie(E0)_is_generically_flat_iso}), we have an isomorphism \begin{equation}\label{Lie(Em)_is_generically_flat_iso}
R^{\textnormal{rig}}_{m}\otimes_{R_{m}}\textnormal{Lie}(\mathbb{E}^{(m)})\cong R^{\textnormal{rig}}_{m}\otimes_{\breve{K}}(B_{h}\otimes_{K_{h}}\breve{K})
\end{equation} of $R^{\text{rig}}_{m}[\Gamma\times G_{0}]$-modules, where $\Gamma$ and $G_{0}$ act diagonally on both sides. The action of $G_{0}$ on $B_{h}\otimes_{K_{h}}\breve{K}$ by convention is trivial. The corollary follows after taking tensor powers on both sides.
\end{proof}
\indent  Let $v_{i}$ denote the images of the basis elements $c_{i}$ under the map $R^{\text{rig}}_{0}\otimes_{R_{0}}\text{Lie}(\mathbb{E}^{(0)})\longrightarrow M^{1}_{0}$. According to \cite{gh}, Proposition 23.2, the global sections $\lbrace v_{i}\rbrace_{0\leq i\leq h-1}$ of the line bundle $(\mathcal{M}^{1}_{0})^{\text{rig}}$ have no common zeros on $X^{\text{rig}}_{0}$, and are linearly independent over $\breve{K}$. If $\mathbb{V}$ denotes the $\breve{K}$-subspace of $M^{1}_{0}$ spanned by them, then $\mathbb{V}$ is $\Gamma$-stable, and is isomorphic to $B_{h}\otimes_{K_{h}}\breve{K}$ as a $\Gamma$-representation. Let $\mathbb{P}(\mathbb{V})$ be the projective space of all hyperplanes in $\mathbb{V}$, then the map \begin{align*}
\Phi:X&^{\text{rig}}_{0}\longrightarrow\mathbb{P}(\mathbb{V})\\& x\longmapsto\lbrace v\in \mathbb{V}|v(x)=0\rbrace
\end{align*} is an \'{e}tale surjective morphism of rigid analytic spaces, if $\mathbb{P}(\mathbb{V})$ is identified with the $(h-1)$-dimensional rigid analytic projective space $\mathbb{P}^{h-1}_{\breve{K}}$ (cf. \cite{gh}, Proposition 23.5). The morphism $\Phi:X^{\text{rig}}_{0}\longrightarrow\mathbb{P}^{h-1}_{\breve{K}}$ is called \emph{the period morphism}. In homogeneous projective coordinates, it is given by $\Phi(x)=[\varphi_{0}(x):\ldots:\varphi_{h-1}(x)]$ where $\varphi_{0},\ldots,\varphi_{h-1}\in R^{\textnormal{rig}}_{0}$ are certain global rigid analytic functions without any common zero. These functions can be constructed from the logarithm $g_{0}(X)=\sum_{n\geq 0}a_{n}X^{q^{n}}$ of the universal formal $\mathfrak{o}$-module $\mathbb{H}^{(0)}$ over $R_{0}$ as the limits \begin{align}\label{varphi_i_formulae}
&\varphi_{0}:=\lim_{n\to\infty}\varpi^{n}a_{nh}\\
&\varphi_{i}:=\lim_{n\to\infty}\varpi^{n+1}a_{nh+i}, \hspace{.2cm}\text{if}\hspace{.2cm}1\leq i\leq h-1\nonumber
\end{align} in the Fr\'{e}chet topology of $R^{\text{rig}}_{0}$ (cf. \cite{gh}, (21.6) and (21.13)).\\
\indent  An important property of the period morphism $\Phi$ is that it is $\Gamma$-equivariant for the $\Gamma$-action on $\mathbb{P}^{h-1}_{\breve{K}}$ by fractional linear transformations via the following injective group homomorphism (cf. \cite{kohliwath}, Remark 1.4):  \begin{align}\label{Gamma_as_a_subgroup_of_GLh}
&\hspace{2.5cm} j:\Gamma\hookrightarrow GL_{h}(K_{h})\nonumber \\
&\sum_{i=0}^{h-1}\lambda_{i}\Pi^{i}\longmapsto \begin{pmatrix}
\lambda_{0} & \varpi\lambda_{1} & \varpi\lambda_{2} & \cdots & \cdots & \varpi\lambda_{h-1} \\
\lambda_{h-1}^{\sigma} & \lambda_{0}^{\sigma} & \lambda_{1}^{\sigma} & \cdots & \cdots & \lambda_{h-2}^{\sigma}\\
\lambda_{h-2}^{\sigma^{2}} & \varpi\lambda_{h-1}^{\sigma^{2}} & \lambda_{0}^{\sigma^{2}} & \cdots & \cdots & \lambda_{h-3}^{\sigma^{2}}\\
\vdots  & \vdots  & \ddots & \ddots & &\vdots  \\
\vdots  & \vdots  & &\ddots & \ddots & \vdots  \\ 
\lambda_{1}^{\sigma^{h-1}} & \varpi\lambda_{2}^{\sigma^{h-1}} & \cdots & \cdots & \varpi\lambda_{h-1}^{\sigma^{h-1}} & \lambda_{0}^{\sigma^{h-1}}
\end{pmatrix}
\end{align} 
\indent  The Gross-Hopkins fundamental domain $D$ is the affinoid subdomain of $X^{\textnormal{rig}}_{0}$ defined as follows: \begin{equation}\label{DefnD}
D:=\left\lbrace \text{$x\in X^{\textnormal{rig}}_{0} \hspace{.1cm}\Big\vert\hspace{.1cm} |u_{i}(x)|\leq |\varpi|^{(1-\frac{i}{h})}$ for all $1\leq i \leq h-1$} \right\rbrace \end{equation}   
According to \cite{gh}, Lemma 23.14, the function $\varphi_{0}$ does not have any zeroes on $D$, hence is a unit in $\mathcal{O}_{X^{\textnormal{rig}}_{0}}(D)$. Setting $w_{i}:=\frac{\varphi_{i}}{\varphi_{0}}$ for $1\leq i\leq h-1$, \cite{gh}, Lemma 23.14 implies that the affinoid $\breve{K}$-algebra $\mathcal{O}_{X^{\textnormal{rig}}_{0}}(D)$ is isomorphic to the generalized Tate algebra.  \begin{align}\label{expression_of_O(D)}
&\mathcal{O}_{X^{\textnormal{rig}}_{0}}(D)\cong \breve{K} \langle\varpi^{-(1-\frac{1}{h})}w_{1},\dots ,\varpi^{-(1-\frac{h-1}{h})}w_{h-1}\rangle\\&:=\Big\{\sum_{\alpha\in\mathbb{N}_{0}^{h-1}}c_{\alpha}w^{\alpha}\in\breve{K}[[w_{1},\dots ,w_{h-1}]]\hspace{.1cm}\Big\vert\hspace{.1cm}\lim_{\vert \alpha\vert\to\infty}\vert c_{\alpha}\vert\vert\varpi\vert^{\sum_{i=1}^{h-1}\alpha_{i}(1-\frac{i}{h})}=0\Big\}\nonumber
\end{align} 
\indent  It follows from \cite{fgl}, Remarque I.3.2 that $D$ is stable under the $\Gamma$-action on $X^{\text{rig}}_{0}$. Also, the $\Gamma$-equivariant period morphism $\Phi$ restricts to an isomorphism $\Phi:D\iso\Phi(D)$ over $D$ (cf. \cite{gh}, Corollary 23.15). As a result, we have an explicit formula for the $\Gamma$-action on $\mathcal{O}_{X^{\textnormal{rig}}_{0}}(D)$ similar to the one of Devinatz-Hopkins (cf. \cite{kohliwath}, Proposition 1.3):
\begin{proposition}\label{dh} Fix $i$ with $1\leq i\leq h-1$, and let $\gamma=\sum_{j=0}^{h-1}\lambda_{j}\Pi^{j}\in\Gamma$. Then \begin{equation}\label{dheq}
\gamma(w_{i})=\frac{\sum_{j=1}^{i}\lambda_{i-j}^{\sigma^{j}}w_{j}+\sum_{j=i+1}^{h}\varpi\lambda_{h+i-j}^{\sigma^{j}}w_{j}}{\lambda_{0}+\sum_{j=1}^{h-1}\lambda_{h-j}^{\sigma^{j}}w_{j}}.
\end{equation}
The group $\Gamma$ acts on $\mathcal{O}_{X^{\textnormal{rig}}_{0}}(D)$ by continuous $\breve{K}$-algebra endomorphisms extending its action on $R^{\textnormal{rig}}_{0}$.
\end{proposition}
\begin{proof} This is straightforward since $\gamma$ acts on the projective homogeneous coordinates $[\varphi_{0}:\ldots:\varphi_{h-1}]$ through right multiplication with the matrix $j(\gamma)$ in  (\ref{Gamma_as_a_subgroup_of_GLh}). By \cite{bgr}, (6.1.3), Theorem 1, the induced $\breve{K}$-algebra endomorphism $\gamma$ of the affinoid $\breve{K}$-algebra $\mathcal{O}_{X^{\textnormal{rig}}_{0}}(D)$ is automatically continuous.  
\end{proof}
\begin{remark}\label{our_line_bundle_is_a_pullback_of_O(s)}
\emph{A rigidified extension} $(E,s)$ of $\mathbb{H}^{(0)}$ by $\mathbb{G}_{a}$ is an extension $E$ of $\mathbb{H}^{(0)}$ by $\mathbb{G}_{a}$ together with a section $s:\text{Lie}(\mathbb{H}^{(0)})\longrightarrow\text{Lie}(E)$. The set $\text{RigExt}(\mathbb{H}^{(0)},\mathbb{G}_{a})$ of isomorphism classes of rigidified extensions of $\mathbb{H}^{(0)}$ by $\mathbb{G}_{a}$ is a free $R_{0}$-module of rank $h$, and has a basis $\lbrace g_{0},g_{1},\ldots, g_{h-1}\rbrace$ where $g_{0}\in R_{0}[[X]]$ is the logarithm of $\mathbb{H}^{(0)}$, and $g_{i}:=\frac{\partial g_{0}}{\partial u_{i}}$ for $1\leq i\leq h-1$ (cf. \cite{gh}, Proposition 9.8). Moreover, the $R_{0}$-module $\omega(\mathbb{E}^{(0)})$ of invariant differentials on the universal additive extension is isomorphic to $\text{RigExt}(\mathbb{H}^{(0)},\mathbb{G}_{a})$ (cf. \cite{gh}, (11.4)). Thus, $R^{\textnormal{rig}}_{0}\otimes_{R_{0}}\textnormal{Lie}(\mathbb{E}^{(0)})\cong\text{Hom}_{R_{0}}(\text{RigExt}(\mathbb{H}^{(0)},\mathbb{G}_{a}),R^{\textnormal{rig}}_{0})$. The functions $\varphi_{i}$ in (\ref{varphi_i_formulae}) are precisely $c_{i}(g_{0})$, and the basis $dg_{0}$ of $\omega(\mathbb{H}^{(0)})$ is mapped to $g_{0}$ under the natural map $\omega(\mathbb{H}^{(0)})\longrightarrow\omega(\mathbb{E}^{(0)})$. As a result, the global sections $v_{i}$ and $v_{j}$ of the line bundle $(\mathcal{M}^{1}_{0})^{\text{rig}}$ (see paragraph after Corollary \ref{Lie(Ems)_generically_flat}) are related by the relation $\varphi_{j}v_{i}=\varphi_{i}v_{j}$ for all $0\leq i,j\leq h-1$. Consequently, we have $\varphi_{j}^{s}v_{i}^{s}=\varphi_{i}^{s}v_{j}^{s}$ in $M^{s}_{0}$. Let $U_{i}\subset X^{\text{rig}}_{0}$ be the non-vanishing locus of $\varphi_{i}$, then on $U_{i}\cap U_{j}$, we get $v_{i}^{s}=\frac{\varphi_{i}^{s}}{\varphi_{j}^{s}} v_{j}^{s}$ and $v_{j}^{s}=\frac{\varphi_{j}^{s}}{\varphi_{i}^{s}} v_{i}^{s}$. The $U_{i}$'s cover $X^{\text{rig}}_{0}$ as the functions $\varphi_{i}$'s do not vanish simultaneously at any point on $X^{\text{rig}}_{0}$. This means that $(\mathcal{M}^{s}_{0})^{\text{rig}}\big|_{U_{i}}\cong\mathcal{O}_{X^{\text{rig}}_{0}}\big|_{U_{i}}\varphi_{i}^{s}$ for all $0\leq i\leq h-1$. In particular, for $i=0$, we have an isomorphism $M^{s}_{D}:=(\mathcal{M}^{s}_{0})^{\text{rig}}(D)\cong\mathcal{O}_{X^{\textnormal{rig}}_{0}}(D)\varphi_{0}^{s}$ of $\mathcal{O}_{X^{\textnormal{rig}}_{0}}(D)$-modules which is also $\Gamma$-equivariant. The $\Gamma$-action on $M^{s}_{D}$ is semilinear for its action on $\mathcal{O}_{X^{\textnormal{rig}}_{0}}(D)$.
\end{remark}
\begin{remark}\label{our_line_bundle_is_a_pullback_of_O(s)*} The discussion in Remark \ref{our_line_bundle_is_a_pullback_of_O(s)} shows that the generating global sections $v_{i}$'s of the line bundle $(\mathcal{M}^{1}_{0})^{\text{rig}}$ are the pullbacks $\Phi^{*}(\varphi_{i})$ of $\varphi_{i}$'s along the period morphism $\Phi$ for all $0\leq i\leq h-1$. As a consequence, it follows that \begin{equation*}
(\mathcal{M}^{1}_{0})^{\text{rig}}\cong\Phi^{*}\mathcal{O}_{\mathbb{P}^{h-1}_{\breve{K}}}(1).
\end{equation*} By the general properties of the inverse image functor, we then have \begin{equation*}
(\mathcal{M}^{s}_{0})^{\text{rig}}\cong\Phi^{*}\mathcal{O}_{\mathbb{P}^{h-1}_{\breve{K}}}(s) \hspace{6mm} \textnormal{for all $s\in\mathbb{Z}$.}
\end{equation*} 
\end{remark}
\subsection{Local analyticity of the $\Gamma$-action on $M^{s}_{D}$}
\label{step1}
\indent In this subsection, we show that the orbit map $(\gamma\mapsto\gamma(f\varphi_{0}^{s})):\Gamma\longrightarrow M^{s}_{D}$ explicitly given by Proposition \ref{dh} and Remark \ref{our_line_bundle_is_a_pullback_of_O(s)} is locally $K$-analytic for all $f\varphi_{0}^{s}\in M^{s}_{D}$.\\
\indent  Let $M_{h}(K_{h})$ denote the ring of $h\times h$ matrices with entries from $K_{h}$. It carries an induced topology from the identification with $K_{h}^{h^{2}}$, which endows it with a structure of a locally analytic $K_{h}$-manifold. The subset $GL_{h}(K_{h})$ of invertible matrices is open and forms a locally $K_{h}$-analytic group. Consider the subgroup $P$ of $GL_{h}(K_{h})$ defined as follows.
\[
  P:= \left\{ a=(a_{ij})_{0\leq i,j\leq h-1}\in GL_{h}(\mathfrak{o}_{h})\ \middle\vert \begin{array}{l}
    \hspace{.3cm} a_{ij},a_{0k}\in \varpi\mathfrak{o}_{h}\hspace{.2cm} \textnormal{for all} \\ 1\leq i,j,k \leq h-1 \hspace{.2cm}\textnormal{with}\hspace{.2cm} i>j
  \end{array}\right\}
\]
It is conjugate to a standard Iwahori subgroup of $GL_{h}(K_{h})$. The conditions on the entries of a matrix in $P$ force all of its diagonal entries to lie in $\mathfrak{o}_{h}^{\times}$. Since $P$ contains the ball of radius $\vert\varpi^{2}\vert$ around $a$ for any $a\in P$, $P$ is open in $GL_{h}(K_{h})$. Thus, $P$ is a locally $K_{h}$-analytic subgroup of $GL_{h}(K_{h})$. 
The inclusion map $j:\Gamma\hookrightarrow GL_{h}(K_{h})$ mentioned in (\ref{Gamma_as_a_subgroup_of_GLh}) has image in $P$.
\begin{lemma}\label{j_is_locKan} The inclusion map $j:\Gamma\hookrightarrow P$
in (\ref{Gamma_as_a_subgroup_of_GLh}) is locally $K$-analytic. 
\end{lemma}
\begin{proof} The global chart for $P$ induced from that for $M_{h}(K_{h})$ sends $a$ in $P$ to 
\begin{equation*}
(a_{00},a_{01},\dots a_{0(h-1)},a_{10},a_{11},\dots ,a_{(h-1)(h-2)},a_{(h-1)(h-1)})
\end{equation*} in $K_{h}^{h^{2}}$. Recall the global chart $\psi$ for $\Gamma$ from (\ref{chart_for_Gamma}). Using the global charts for both groups, it is easy to see that the corresponding map from the open subset $\psi(\Gamma)$ in $K_{h}^{h}$ to $K_{h}^{h^{2}}$ is locally $K$-analytic since each component of this map is either a linear polynomial or a $K$-linear Frobenius automorphism $\sigma$ or a composition of both, all being locally $K$-analytic. As before, we remark that $j$ is generally not locally $K_{h}$-analytic because $\sigma:\mathfrak{o}_{h}^{\times}\longrightarrow\mathfrak{o}_{h}^{\times}$ is not locally $K_{h}$-analytic unless $h=1$.
\end{proof}
\indent  The algebra $\mathcal{O}_{X^{\textnormal{rig}}_{0}}(D)$ is a $\breve{K}$-Banach algebra with respect to the multiplicative norm $\|\cdot\|_{D}$defined as follows: for $f=\sum_{\alpha\in\mathbb{N}_{0}^{h-1}}c_{\alpha}w^{\alpha}\in\mathcal{	O}_{X^{\text{rig}}_{0}}(D)$, $\|f\|_{D}:=\sup_{\alpha\in\mathbb{N}_{0}^{h-1}}\vert c_{\alpha}\vert \vert\varpi\vert^{\sum_{i=1}^{h-1}\alpha_{i}(1-\frac{i}{h})}$ (cf. \cite{bgr}, Section 6.1.5, Proposition 1 and 2). Let $P$ act on $\mathcal{O}_{X^{\textnormal{rig}}_{0}}(D)$ by $\breve{K}$-linear ring automorphisms by defining \begin{equation}\label{actI}
a(w_{i}):=\frac{a_{0i}+\sum_{j=1}^{h-1}a_{ji}w_{j}}{a_{00}+\sum_{j=1}^{h-1}a_{j0}w_{j}}
\end{equation} 
for $a\in P$ and for $1\leq i\leq h-1$. This gives an action of $P$ on $\mathcal{O}_{X^{\textnormal{rig}}_{0}}(D)$ by continuous $\breve{K}$-linear ring automorphisms which, when restricted to $\Gamma$ via $j$, coincides with the $\Gamma$-action on $\mathcal{O}_{X^{\textnormal{rig}}_{0}}(D)$ (cf. Proposition \ref{dh}). Indeed, note that $a_{00}+\sum_{j=1}^{h-1}a_{j0}w_{j}=a_{00}(1+\sum_{j=1}^{h-1}a_{00}^{-1}a_{j0}w_{j})\in(\mathcal{O}_{X^{\textnormal{rig}}_{0}}(D))^{\times}$ is a unit of norm 1, and $\big\|a_{0i}+\sum_{j=1}^{h-1}a_{ji}w_{j}\big\|_{D}=\|w_{i}\|_{D}$ by the strict triangle inequality. Altogether, $\|a(w_{i})\|_{D}=\|w_{i}\|_{D}$ which ensures that $P$ acts on $\mathcal{O}_{X^{\textnormal{rig}}_{0}}(D)$ via $\sum_{\alpha\in\mathbb{N}_{0}^{h-1}}c_{\alpha}w^{\alpha}\longmapsto\sum_{\alpha\in\mathbb{N}_{0}^{h-1}}c_{\alpha}a(w_{1})^{\alpha_{1}}\ldots a(w_{h-1})^{\alpha_{h-1}}$ in a well-defined way.\\
\indent  We now show that the above action is locally $K$-analytic.
\begin{lemma}\label{invla} 
The map $\iota:P\longrightarrow P$, $(a_{ij})_{0\leq i,j\leq h-1}\longmapsto(\iota(a)_{ij})_{0\leq i,j\leq h-1}$ defined by \begin{equation*}
\iota(a)_{ij}=
\begin{cases}
a_{ij}^{-1}, &\text{if $i=j=0$;}\\
a_{ij}, &\text{otherwise}
\end{cases}
\end{equation*}  is locally $K_{h}$-analytic, and thus locally $K$-analytic.
\end{lemma}
\begin{proof}
This follows from \cite{schplie}, Proposition 13.6, and the fact that $K_{h}^{\times}$ is a locally $K_{h}$-analytic group. The local $K$-analyticity of $\iota$ follows due to restriction of scalars.
\end{proof}
\begin{proposition}\label{laofI} The action of $P$ on $\mathcal{O}_{X^{\textnormal{rig}}_{0}}(D)$ is locally $K_{h}$-analytic, and thus locally $K$-analytic, i.e., the orbit maps of the action are locally $K$-analytic. 
\end{proposition}
\begin{proof}
By Lemma \ref{invla}, it is enough to show that, for each $f\in\mathcal{O}_{X^{\textnormal{rig}}_{0}}(D)$, the map $\iota(a)\longmapsto a(f)$ from $P$ to $\mathcal{O}_{X^{\textnormal{rig}}_{0}}(D)$ is locally $K_{h}$-analytic. Consider the open neighbourhood $U$ of $0$ in $K_{h}^{h^{2}}$ defined as follows.
\[
  U:= \left\{ x=(x_{1},x_{2},\dots ,x_{h^{2}})\in \mathfrak{o}_{h}^{h^{2}}\ \middle\vert \begin{array}{l}
  x_{i}, x_{qh+r}\in\varpi\mathfrak{o}_{h}\hspace{.15cm}  \textnormal{for all}\hspace{.15cm} 2\leq i\leq h \\ \hspace{.15cm}\textnormal{and}\hspace{.15cm} \textnormal{for all}\hspace{.15cm} q\geq r\hspace{.15cm} \textnormal{with} \hspace{.15cm}q,r>1 
  \end{array}\right\}
\]
Let $T=(T_{1},T_{2},\dots ,T_{h^{2}})$, and let $\mathcal{F}_{U}(K_{h}^{h^{2}},\mathcal{O}_{X^{\textnormal{rig}}_{0}}(D))$ denote the set of power series in $T$ with coefficients from $\mathcal{O}_{X^{\textnormal{rig}}_{0}}(D)$ which converge on $U$.
Like $\mathcal{O}_{X^{\textnormal{rig}}_{0}}(D)$, $\mathcal{F}_{U}(K_{h}^{h^{2}},\mathcal{O}_{X^{\textnormal{rig}}_{0}}(D))$ is also a $\breve{K}$-Banach algebra with respect to the multiplicative norm $\Big\|\sum_{\alpha\in\mathbb{N}_{0}^{h^{2}}}f_{\alpha}T^{\alpha}\Big\|_{U}:=$ \begin{align*}
\sup_{\alpha\in\mathbb{N}_{0}^{h^{2}}}\| f_{\alpha}\|_{D}\vert\varpi\vert^{(\alpha_{2}+\alpha_{3}+\dots +\alpha_{h}+\alpha_{2h+2}+\alpha_{3h+2}+\alpha_{3h+3}+\dots +\alpha_{(h-1)h+h-1})}
\end{align*} (cf. \cite{schplie}, Proposition 5.3).
Under the global chart of $P$ in the proof of \ref{j_is_locKan}, we now show that, for a monomial $w^{\alpha}\in\mathcal{O}_{X^{\textnormal{rig}}_{0}}(D)$, the map $\iota(a)\longmapsto a(w^{\alpha})$ belongs to $\mathcal{F}_{U}(K_{h}^{h^{2}},\mathcal{O}_{X^{\textnormal{rig}}_{0}}(D))$ for every $\alpha\in \mathbb{N}_{0}^{h-1}$.\\
\indent  By (\ref{actI}), we have \begin{align}\label{3bra}
a(w^{\alpha})\nonumber&=a(w_{1})^{\alpha_{1}}\dots a(w_{h-1})^{\alpha_{h-1}}\\\nonumber&=\Bigg(\frac{a_{01}+\sum_{j=1}^{h-1}a_{j1}w_{j}}{a_{00}+\sum_{j=1}^{h-1}a_{j0}w_{j}}\Bigg)^{\alpha_{1}}\dots\Bigg(\frac{a_{0(h-1)}+\sum_{j=1}^{h-1}a_{j(h-1)}w_{j}}{a_{00}+\sum_{j=1}^{h-1}a_{j0}w_{j}}\Bigg)^{\alpha_{h-1}}
\\\nonumber&=\Bigg(\prod_{i=1}^{h-1}\Big(a_{0i}+\sum_{j=1}^{h-1}a_{ji}w_{j}\Big)^{\alpha_{i}}\Bigg)(a_{00}^{-1})^{\vert \alpha\vert}\Big(1+a_{00}^{-1}\sum_{j=1}^{h-1}a_{j0}w_{j}\Big)^{-\vert \alpha\vert}\\&=\Bigg(\prod_{i=1}^{h-1}\Big(a_{0i}+\sum_{j=1}^{h-1}a_{ji}w_{j}\Big)^{\alpha_{i}}\Bigg)(a_{00}^{-1})^{\vert \alpha\vert}\Bigg(\sum_{l=0}^{\infty}\Big(-a_{00}^{-1}\sum_{j=1}^{h-1}a_{j0}w_{j}\Big)^{l}\Bigg)^{\vert \alpha\vert}
\end{align}
Thus the expression of $a(w^{\alpha})$ is a product of $(a_{00}^{-1})^{\vert \alpha\vert}$ and two big brackets. The first big bracket in (\ref{3bra}) is a product of polynomials in $a_{ij}$'s with coefficients from $\mathcal{O}_{X^{\textnormal{rig}}_{0}}(D)$, and hence is the evaluation at $\iota(a)$ of an element in $\mathcal{F}_{U}(K_{h}^{h^{2}},\mathcal{O}_{X^{\textnormal{rig}}_{0}}(D))$. Similarly, $(a_{00}^{-1})^{\vert \alpha\vert}$ is the evaluation at $\iota(a)$ of the monomial $T_{1}^{\vert \alpha\vert}$ which belongs to $\mathcal{F}_{U}(K_{h}^{h^{2}},\mathcal{O}_{X^{\textnormal{rig}}_{0}}(D))$. The second big bracket is the $\vert \alpha\vert$-th power of a certain geometric series. The $l$-th term in that series is the evaluation of the polynomial $\big(-T_{1}\sum_{j=1}^{h-1}T_{jh+1}w_{j}\big)^{l}$ at $\iota(a)$, and \begin{equation*}
\Big\|\Big(-T_{1}\sum_{j=1}^{h-1}T_{jh+1}w_{j}\Big)^{l}\Big\|_{U}=\Big(\Big\|-T_{1}\sum_{j=1}^{h-1}T_{jh+1}w_{j}\Big\|_{U}\Big)^{l}=\vert\varpi\vert^{\frac{l}{h}}
\end{equation*} Hence, the series $\sum_{l=0}^{\infty}(-T_{1}\sum_{j=1}^{h-1}T_{jh+1}w_{j})^{l}$ converges in $\mathcal{F}_{U}(K_{h}^{h^{2}},\mathcal{O}_{X^{\textnormal{rig}}_{0}}(D))$, and the map $\iota(a)\longmapsto a(w^{\alpha})\in\mathcal{F}_{U}(K_{h}^{h^{2}},\mathcal{O}_{X^{\textnormal{rig}}_{0}}(D))$ for every $\alpha\in \mathbb{N}_{0}^{h-1}$.\\
\indent  Let us calculate the norms $\|\cdot\|_{U}$ of the above power series corresponding to the terms in the expression (\ref{3bra}) or find an upper bound for them. First, $\| T_{1}^{|\alpha|}\|_{U}=1$. Since \begin{equation*}
\Big\|\sum_{l=0}^{\infty}\Big(-T_{1}\sum_{j=1}^{h-1}T_{jh+1}w_{j}\Big)^{l}\Big\|_{U}\leq\sup_{l\geq 0}\Big\|\Big(-T_{1}\sum_{j=1}^{h-1}T_{jh+1}w_{j}\Big)^{l}\Big\|_{U}=\sup_{l\geq 0}\vert \varpi\vert^{\frac{l}{h}}=1,
\end{equation*} the power series corresponding to the second big bracket in (\ref{3bra}) has the norm $\leq$ 1. The first big bracket is obtained by evaluating $\prod_{i=1}^{h-1}\Big(T_{i+1}+\sum_{j=1}^{h-1}T_{jh+i+1}w_{j}\Big)^{\alpha_{i}}$ at $\iota(a)$, and \begin{align*}
\Big\|\prod_{i=1}^{h-1}\Big(T_{i+1}+\sum_{j=1}^{h-1}T_{jh+i+1}w_{j}\Big)^{\alpha_{i}}\Big\|_{U}&=\prod_{i=1}^{h-1}\Big\|\Big(T_{i+1}+\sum_{j=1}^{h-1}T_{jh+i+1}w_{j}\Big)\Big\|_{U}^{\alpha_{i}}\\&=\prod_{i=1}^{h-1}\vert\varpi\vert^{\alpha_{i}(1-\frac{i}{h})}. 
\end{align*} Therefore, the power series corresponding to the first big bracket has the norm $\vert\varpi\vert^{\sum_{i=1}^{h-1}\alpha_{i}(1-\frac{i}{h})}$. So, for every $\alpha\in \mathbb{N}_{0}^{h-1}$, the map $\iota(a)\longmapsto a(w^{\alpha})$ is given by an element in $\mathcal{F}_{U}(K_{h}^{h^{2}},\mathcal{O}_{X^{\textnormal{rig}}_{0}}(D))$ whose norm is bounded above by $\vert\varpi\vert^{\sum_{i=1}^{h-1}\alpha_{i}(1-\frac{i}{h})}$.\\
\indent  Now given $f=\sum_{\alpha\in\mathbb{N}_{0}^{h-1}}c_{\alpha}w^{\alpha}\in\mathcal{O}_{X^{\textnormal{rig}}_{0}}(D)$, $a(f)=\sum_{\alpha\in\mathbb{N}_{0}^{h-1}}c_{\alpha}(a (w^{\alpha}))$, and for every $\alpha\in \mathbb{N}_{0}^{h-1}$, the map $\iota(a)\longmapsto c_{\alpha}(a (w^{\alpha}))$ is represented by a power series in $\mathcal{F}_{U}(K_{h}^{h^{2}},\mathcal{O}_{X^{\textnormal{rig}}_{0}}(D))$ having norm $\leq\vert c_{\alpha}\vert\vert\varpi\vert^{\sum_{i=1}^{h-1}\alpha_{i}(1-\frac{i}{h})}$. Since, \linebreak $\lim_{\vert \alpha\vert\to\infty}\vert c_{\alpha}\vert\vert\varpi\vert^{\sum_{i=1}^{h-1}\alpha_{i}(1-\frac{i}{h})}=0$, we see that the map $\iota(a)\longmapsto a(f)$ from $P$ to $\mathcal{O}_{X^{\textnormal{rig}}_{0}}(D)$ is given by a convergent power series in $\mathcal{F}_{U}(K_{h}^{h^{2}},\mathcal{O}_{X^{\textnormal{rig}}_{0}}(D))$. As $a$ is arbitrary, this implies that the action of $P$ on $\mathcal{O}_{X^{\textnormal{rig}}_{0}}(D)$ is locally $K_{h}$-analytic, and thus locally $K$-analytic by the restriction of scalars.
\end{proof}
\begin{proposition}\label{laofG} The $\breve{K}$-Banach space $\mathcal{O}_{X^{\textnormal{rig}}_{0}}(D)$ is a locally $K$-analytic representation of $\Gamma$.\end{proposition}
\begin{proof}
This follows from the Lemma \ref{j_is_locKan} and Proposition \ref{laofI}.
\end{proof}
\begin{theorem}\label{laofGonMD}
Let $s$ be any integer, then the $\breve{K}$-Banach space $M^{s}_{D}$ is a locally $K$-analytic representation of $\Gamma$.
\end{theorem}
\begin{proof}
Due to Remark \ref{our_line_bundle_is_a_pullback_of_O(s)}, we have a $\Gamma$-equivariant, $\mathcal{O}_{X^{\textnormal{rig}}_{0}}(D)$-linear isomorphism $M^{s}_{D}\cong\mathcal{O}_{X^{\textnormal{rig}}_{0}}(D).\varphi_{0}^{s}$ of free $\mathcal{O}_{X^{\textnormal{rig}}_{0}}(D)$-modules of rank 1. Then $M^{s}_{D}$ obtains a structure of a $\breve{K}$-Banach space with respect to the norm defined as $\| f\varphi^{s}_{0}\|_{M^{s}_{D}}:=\| f\|_{D}$. Since the $\Gamma$-action on $M^{s}_{D}$ is semilinear for its action on $\mathcal{O}_{X^{\textnormal{rig}}_{0}}(D)$, we have $\gamma(f\varphi^{s}_{0})=\gamma(f)\gamma(\varphi^{s}_{0})$ for all $\gamma\in\Gamma$ and $f\in\mathcal{O}_{X^{\textnormal{rig}}_{0}}(D)$. Now, as mentioned in the proof of Proposition \ref{dh}, $\gamma(\varphi^{s}_{0})=\gamma(\varphi_{0})^{s}=(\lambda_{0}\varphi_{0}+\lambda^{\sigma}_{h-1}\varphi_{1}+\dots +\lambda^{\sigma^{h-1}}_{1}\varphi_{h-1})^{s}=(\lambda_{0}+\lambda^{\sigma}_{h-1}w_{1}+\dots +\lambda^{\sigma^{h-1}}_{1}w_{h-1})^{s}\varphi_{0}^{s}$. So the orbit map from $\Gamma$ to $M^{s}_{D}$ is given by sending $\gamma$ to $\gamma(f\varphi^{s}_{0})=\gamma(f)(\lambda_{0}+\lambda^{\sigma}_{h-1}w_{1}+\dots +\lambda^{\sigma^{h-1}}_{1}w_{h-1})^{s}\varphi^{s}_{0}$. The map $\gamma\mapsto\gamma(f)$ is locally $K$-analytic by Proposition \ref{laofG}, and the map $\gamma\mapsto(\lambda_{0}+\lambda^{\sigma}_{h-1}w_{1}+\dots +\lambda^{\sigma^{h-1}}_{1}w_{h-1})^{s}\varphi^{s}_{0}$ is also locally $K$-analytic since it is given by a linear polynomial in the coordinates of $\gamma$. Thus, the orbit map, being a product of these two maps, is locally $K$-analytic. Therefore, $M^{s}_{D}$ is a locally analytic $\Gamma$-representation for all integers $s$.
\end{proof}
\subsection{Local analyticity of the $\Gamma$-action on $M^{s}_{0}$}
\label{step2}
\indent Let $D(\Gamma)$ denote the algebra of $\breve{K}$-valued locally $K$-analytic distributions on $\Gamma$  (cf. \cite{stladist}, Section 2). The strong topological duality induces an anti-equivalence between the category of locally $K$-analytic representations of $\Gamma$ on the $\breve{K}$-vector spaces of compact type and the category of continuous $D(\Gamma)$-modules on the nuclear $\breve{K}$-Fr\'{e}chet spaces (cf. \cite{stladist}, Corollary 3.4). Using the local $K$-analyticity of the $\Gamma$-action on $M^{s}_{D}$ obtained in the previous subsection, we now show that, at level $m=0$, the induced $\Gamma$-action on the vector space $(M^{s}_{0})'_{b}$ of compact type is locally $K$-analytic by showing that its strong topological dual $M^{s}_{0}$ is a continuous $D(\Gamma)$-module.\\
\indent  The continuity of the $\Gamma$-action on the universal deformation ring $R_{0}$ (cf. Theorem \ref{ctsthm}) leads to a continuous $\Gamma$-action on $R^{\text{rig}}_{0}$. This is implied by the next proposition.
\begin{proposition}\label{mainprop} Let $n$ and $l$ be integers with $n\geq 0$ and $l\geq 1$. If $\gamma\in\Gamma_{n}$, and if $f\in R^{\textnormal{rig}}_{0}$, then $\|\gamma(f)-f\|_{l}\leq\vert \varpi\vert^{n/l}\| f\|_{l}$.
\end{proposition}
\begin{proof}
Note that $R^{\text{rig}}_{0,l}$ is a generalized Tate algebra over $\breve{K}$ in the variables \linebreak $(\varpi^{-1/l}u_{i})_{1\leq i\leq h-1}$. Then by \cite{bgr}, (6.1.5), Proposition 5, we have \begin{equation*}
 \| g\|_{l}=\sup\big\lbrace\vert g(x)\vert\hspace{.1cm}\big\vert\hspace{.1cm} x\in\mathbb{B}_{l}\big(\overline{\breve{K}}\big)\big\rbrace\hspace{.2cm}\text{for any}\hspace{.2cm} g\in R^{\textnormal{rig}}_{0,l},
\end{equation*} where \begin{equation*}
\mathbb{B}_{l}\big(\overline{\breve{K}}\big)=\big\lbrace x\in\big(\overline{\breve{K}}\big)^{h-1}\hspace{.1cm}\big\vert\hspace{.1cm}\text{$\vert x_{i}\vert\leq\vert\varpi\vert^{1/l}$ for all $1\leq i \leq h-1$}\big\rbrace.
\end{equation*} Let us first prove the assertion for $f=u_{i}$ for some $1\leq i\leq h-1$. If $x\in\mathbb{B}_{l}\big(\overline{\breve{K}}\big)$ and $y=(y_{j}):=(\gamma(u_{j})(x))$, then we need to show that $\vert x_{i}-y_{i}\vert\leq\vert\varpi\vert^{(n+1)/l}$ because $\|u_{i}\|_{l}=|\varpi|^{1/l}$. Consider the commutating diagram \begin{displaymath}
\xymatrix{
        R_{0}\ar[dr]_{f\mapsto f(y)}\ar[rr]^{\gamma} & & R_{0}\ar[dl]^{f\mapsto f(x)}  \\
        & \overline{\breve{\mathfrak{o}}} &  }
\end{displaymath}
of homomorphisms of $\breve{\mathfrak{o}}$-algebras. Choosing $z\in\overline{\breve{\mathfrak{o}}}$ with $\vert z\vert = \vert\varpi\vert^{1/l}$, we have $x_{j}\in z\overline{\breve{\mathfrak{o}}}$ for all $j$. Further, $\varpi\in z\overline{\breve{\mathfrak{o}}}$ because $l\geq 1$. As a consequence, the right oblique arrow of the above diagram maps $\mathfrak{m}_{R_{0}}=(\varpi,u_{1},\ldots,u_{h-1})$ to $z\overline{\breve{\mathfrak{o}}}$. Note that $\gamma(u_{j})\in\mathfrak{m}_{R_{0}}$ so we obtain $y_{j}\in z\overline{\breve{\mathfrak{o}}}$ as well. Therefore, also the left oblique arrow maps $\mathfrak{m}_{R_{0}}$ to $z\overline{\breve{\mathfrak{o}}}$. Now consider the induced diagram 
\begin{displaymath}
\xymatrix{
        R_{0}/\mathfrak{m}_{R_{0}}^{n+1}\ar[dr]\ar[rr]^{\gamma} & & R_{0}/\mathfrak{m}_{R_{0}}^{n+1}\ar[dl]  \\
        & \overline{\breve{\mathfrak{o}}}/(z^{n+1}) &  }
\end{displaymath}
According to Theorem \ref{ctsthm}, the upper horizontal arrow is the identity. It follows that $x_{i}-y_{i}\in z^{n+1}\overline{\breve{\mathfrak{o}}}$, i.e., $\vert x_{i}-y_{i}\vert\leq\vert\varpi\vert^{(n+1)/l}$.
\\
\indent  We now prove the assertion for $f=u^{\alpha}$ by induction on $|\alpha|$. The case $|\alpha|=0$ is trivial. Let $|\alpha|>0$. Choose an index $i$ with $\alpha_{i}>0$. Define $\beta_{j}:=\alpha_{j}$ if $j\neq i$, and $\beta_{i}:=\alpha_{i}-1$. Then for $x\in\mathbb{B}_{l}\big(\overline{\breve{K}}\big)$,\begin{align*}
\vert \gamma(u^{\alpha})(x)-u^{\alpha}(x)\vert=\vert y^{\alpha}-x^{\alpha}\vert&=\vert y_{i}y^{\beta}-x_{i}x^{\beta}\vert\\&\leq\max\lbrace\vert y_{i}\vert\vert y^{\beta}-x^{\beta}\vert,\vert y_{i}-x_{i}\vert\vert x^{\beta}\vert\rbrace.
\end{align*} Now $\vert y_{i}\vert\vert y^{\beta}-x^{\beta}\vert\leq\vert\varpi\vert^{1/l}\|\gamma(u^{\beta})-u^{\beta}\|_{l}\leq\vert\varpi\vert^{(n+1)/l}\| u^{\beta}\|_{l}=\vert\varpi\vert^{n/l}\| u^{\alpha}\|_{l}$ by the induction hypothesis and $\vert y_{i}-x_{i}\vert\vert x^{\beta}\vert\leq\vert\varpi\vert^{(n+1)/l}\vert\varpi\vert^{\vert \beta\vert/l}=\vert\varpi\vert^{n/l}\| u^{\alpha}\|_{l}$ as seen above. Thus we obtain $\vert \gamma(u^{\alpha})(x)-u^{\alpha}(x)\vert\leq\vert\varpi\vert^{n/l}\|u^{\alpha}\|_{l}$ for all $x\in\mathbb{B}_{l}\big(\overline{\breve{K}}\big)$ as required.\\
\indent  Therefore if $f=\sum_{\alpha\in\mathbb{N}^{h-1}_{0}}c_{\alpha}u^{\alpha}\in R^{\textnormal{rig}}_{0}$, then by continuity of $\gamma$, we get \begin{align*}
\|\gamma(f)-f\|_{l}=\Big\|\sum_{\alpha\in\mathbb{N}^{h-1}_{0}}c_{\alpha}(\gamma(u^{\alpha})-u^{\alpha})\Big\|_{l}&\leq\sup_{\alpha\in\mathbb{N}^{h-1}_{0}}|c_{\alpha}|\|\gamma(u^{\alpha})-u^{\alpha}\|_{l}\\&\leq\sup_{\alpha\in\mathbb{N}^{h-1}_{0}}|c_{\alpha}||\varpi|^{n/l}\|u^{\alpha}\|_{l}\\&=|\varpi|^{n/l}\|f\|_{l}.
\end{align*}
\end{proof}
\indent  We write $\Gamma_{\mathbb{Q}_{p}}$ for $\Gamma$ when viewed as a locally $\mathbb{Q}_{p}$-analytic group, and $\mathfrak{g}_{\mathbb{Q}_{p}}$ for its Lie algebra $\mathfrak{g}$ when considered as a $\mathbb{Q}_{p}$-vector space. Let $d:=[K:\mathbb{Q}_{p}]$. Since $\Gamma_{\mathbb{Q}_{p}}$ is a compact locally $\mathbb{Q}_{p}$-analytic group of dimension $t:=dh^{2}$, it contains an open subgroup $\Gamma_{o}$ which is a uniform pro-$p$ group of rank $t$ (cf. \cite{ddms}, Theorem 8.32). The subgroups in its lower $p$-series $P_{i}(\Gamma_{o})$ $(i\geq 1)$ form a basis of open neighbourhoods of the identity in $\Gamma_{o}$ and are also uniform pro-$p$ groups of rank $t$ (cf. \cite{ddms}, Proposition 1.7, Proposition 1.11 (i), Theorem 3.6 (i) and Proposition 4.4). Let $n$ be a positive integer such that $\Gamma_{n}\subseteq\Gamma_{o}$. As $\Gamma_{n}$ is open in $\Gamma_{o}$, it contains $\Gamma_{*}:=P_{i}(\Gamma_{o})$ for some $i\geq 1$. In what follows, we view $\Gamma_{*}$ as a locally $\mathbb{Q}_{p}$-analytic group.\\
\indent  Let us denote by $\Lambda(\Gamma_{*}):=\breve{\mathfrak{o}}[[\Gamma_{*}]]$ the Iwasawa algebra of $\Gamma_{*}$ over $\breve{\mathfrak{o}}$. Set $b_{i}:=\gamma_{i}-1\in\Lambda(\Gamma_{*})$ and $b^{\alpha}:=b_{1}^{\alpha_{1}}\cdots b_{t}^{\alpha_{t}}$ for any $\alpha\in\mathbb{N}_{0}^{t}$ where $\lbrace\gamma_{1},\dots ,\gamma_{t}\rbrace$ is a minimal topological generating set of $\Gamma_{*}$. By \cite{ddms}, Theorem 7.20, any element $\mu\in\Lambda(\Gamma_{*})$ admits a unique expansion of the form \begin{equation*}
\text{$\mu =\sum_{\alpha\in\mathbb{N}_{0}^{t}}d_{\alpha}b^{\alpha}$ with $d_{\alpha}\in\breve{\mathfrak{o}}$ $\forall$ $\alpha\in\mathbb{N}_{0}^{t}$}
\end{equation*}
For any $l\geq 1$, this allows us to define the $\breve{K}$-norm $\|\cdot\|_{l}$ on the algebra $\Lambda(\Gamma_{*})_{\breve{K}}:=\Lambda(\Gamma_{*})\otimes_{\breve{\mathfrak{o}}}\breve{K}$ through \begin{equation}\label{lnorm}
\Big\|\sum_{\alpha\in\mathbb{N}_{0}^{t}}d_{\alpha}b^{\alpha}\Big\|_{l}:=\sup_{\alpha\in\mathbb{N}_{0}^{t}}\lbrace\vert d_{\alpha}\vert\vert \varpi\vert^{\vert \alpha\vert /l}\rbrace
\end{equation} By \cite{stadmrep}, Proposition 4.2, the norm $\|\cdot\|_{l}$ on $\Lambda(\Gamma_{*})_{\breve{K}}$ is submultiplicative. As a consequence, the completion \begin{equation*}
\Lambda(\Gamma_{*})_{\breve{K},l}=\Big\lbrace\sum_{\alpha\in\mathbb{N}_{0}^{t}}d_{\alpha}b^{\alpha}\mid d_{\alpha}\in\breve{K} , \lim_{\vert \alpha\vert\to\infty}\vert d_{\alpha}\vert\vert \varpi\vert^{\vert \alpha\vert /l}=0\Big\rbrace
\end{equation*} of $\Lambda(\Gamma_{*})_{\breve{K}}$ with respect to $\|\cdot\|_{l}$ is a $\breve{K}$-Banach algebra. The natural inclusions $\Lambda(\Gamma_{*})_{\breve{K},l+1}\hookrightarrow\Lambda(\Gamma_{*})_{\breve{K},l}$ endow the projective limit \begin{equation*}D(\Gamma_{*})=\varprojlim_{l}\Lambda(\Gamma_{*})_{\breve{K},l}\end{equation*} with the structure of a $\breve{K}$-Fr\'{e}chet algebra. By Amice's theorem, the above projective limit is indeed equal to the algebra of $\breve{K}$-valued locally $\mathbb{Q}_{p}$-analytic distributions on $\Gamma_{*}$ (cf. \cite{stadmrep}, Section 4). By fixing coset representatives $\lbrace\gamma'_{1}=1,\gamma'_{2},\ldots,\gamma'_{s}\rbrace$ of $\Gamma_{*}$ in $\Gamma_{\mathbb{Q}_{p}}$, the natural topological isomorphism $C^{an}(\Gamma_{\mathbb{Q}_{p}},\breve{K})\cong\prod_{i=1}^{s}C^{an}(\gamma'_{i}\Gamma_{*},\breve{K})$ of locally convex $\breve{K}$-vector spaces induces a topological isomorphism \begin{equation}\label{top_iso_of_distribution_alg}
D(\Gamma_{\mathbb{Q}_{p}})\cong\bigoplus_{i=1}^{s}\delta_{\gamma'_{i}} D(\Gamma_{*})\hspace{1.5cm}\text{($\delta_{\gamma'_{i}}$'s are Dirac distributions)}
\end{equation} by dualizing (cf. \cite{fe99}, Korollar 2.2.4). This defines a $\breve{K}$-Fr\'{e}chet algebra structure on $D(\Gamma_{\mathbb{Q}_{p}})$ given by the family of norms $\|\delta_{\gamma'_{1}}\mu_{1}+\cdots+\delta_{\gamma'_{s}}\mu_{s}\|_{l}:=\max_{i=1}^{s}\lbrace\|\mu_{i}\|_{l}\rbrace$ with $l\geq 1$ (cf. \cite{stadmrep}, Theorem 5.1).\\
\indent  Note that $D(\Gamma_{\mathbb{Q}_{p}})$ is not the same as the distribution algebra $D(\Gamma)$ of $\breve{K}$-valued locally $K$-analytic distributions on $\Gamma$. In fact, the natural embedding $C^{an}(\Gamma,\breve{K})\hookrightarrow C^{an}(\Gamma_{\mathbb{Q}_{p}},\breve{K})$ induces a map $D(\Gamma_{\mathbb{Q}_{p}})\longrightarrow D(\Gamma)$ which is a strict surjection and a homomorphism of $\breve{K}$-algebras by \cite{kohlinvdist}, Lemma 1.3.1. According to \cite{kohlinvdist}, Lemma 1.3.2 and Lemma 1.3.3, the kernel $I$ of the surjection $D(\Gamma_{\mathbb{Q}_{p}})\twoheadrightarrow D(\Gamma)$ is the closure of the ideal generated by all elements of the form $i(\lambda\mathfrak{x})-\lambda i(\mathfrak{x})$ with $\mathfrak{x}\in\mathfrak{g}_{\mathbb{Q}_{p}}$, $\lambda\in K$ and $i:\mathfrak{g}_{\mathbb{Q}_{p}}\hookrightarrow D(\Gamma_{\mathbb{Q}_{p}})$ denoting the natural inclusion as explained on page 10 of \cite{stladist}.
\begin{theorem}\label{laK0} The action of $\Gamma_{\mathbb{Q}_{p}}$ on $R^{\textnormal{rig}}_{0}$ extends to a continuous action of the $\breve{K}$-Fr\'{e}chet algebra $D(\Gamma_{\mathbb{Q}_{p}})$, which then factors through a continuous action of $D(\Gamma)$ on $R^{\textnormal{rig}}_{0}$. Hence the action of $\Gamma$ on the strong continuous $\breve{K}$-linear dual $(R^{\textnormal{rig}}_{0})'_{b}$ of $R^{\textnormal{rig}}_{0}$ is locally $K$-analytic.
\end{theorem}
\begin{proof}
First, we show that $R^{\textnormal{rig}}_{0,l}$ is a topological Banach module over the $\breve{K}$-Banach algebra $\Lambda(\Gamma_{*})_{\breve{K},l}$ for all $l\geq 1$. To show this, let us prove by induction on $\vert \alpha\vert$ that $\| b^{\alpha}(f)\|_{l}\leq\| b^{\alpha}\|_{l}\| f\|_{l}$ for any $f\in R^{\textnormal{rig}}_{0}$. This is clear if $\vert \alpha\vert=0$. Let $|\alpha|>0$ and let $i$ be the minimal index such that $\alpha_{i}>0$. Define $\beta_{j}:=\alpha_{j}$ if $j\neq i$, and $\beta_{i}:=\alpha_{i}-1$. Since $\Gamma_{*}\subseteq \Gamma_{n}$, Proposition \ref{mainprop} and the induction hypothesis imply \begin{align*}
\| b^{\alpha}(f)\|_{l}&=\| ((\gamma_{i}-1)b^{\beta})(f)\|_{l}=\| (\gamma_{i}-1)(b^{\beta}(f))\|_{l}\leq\vert\varpi\vert^{n/l}\| b^{\beta}(f)\|_{l}\\&\leq\vert\varpi\vert^{1/l}\| b^{\beta}\|_{l}\| f\|_{l}\leq\vert\varpi\vert^{1/l}\vert\varpi\vert^{\vert \beta\vert /l}\| f\|_{l}=\vert\varpi\vert^{(\vert \beta\vert+1) /l}\| f\|_{l}=\| b^{\alpha}\|_{l}\| f\|_{l}  
\end{align*} as required.\\
\indent  By Remark \ref{iwasawa_alg_action}, we then have $\|\mu(f)\|_{l}\leq\|\mu\|_{l}\| f\|_{l}$ for all $\mu\in\Lambda(\Gamma_{*})_{\breve{K}}$ and $f\in R_{0}[\frac{1}{\varpi}]=R_{0}\otimes_{\breve{\mathfrak{o}}}\breve{K}$. Hence the map $\Lambda(\Gamma_{*})_{\breve{K}}\times R_{0}[\frac{1}{\varpi}]\longrightarrow R_{0}[\frac{1}{\varpi}]$ $((\mu,f)\mapsto\mu(f))$ is continuous if $\Lambda(\Gamma_{*})_{\breve{K}}$ and $R_{0}[\frac{1}{\varpi}]$ are endowed with the respective $\|\cdot\|_{l}$-topologies, and if the left hand side carries the product topology. Since $R_{0}[\frac{1}{\varpi}]$ is dense in $R^{\textnormal{rig}}_{0,l}$, we obtain a map $\Lambda(\Gamma_{*})_{\breve{K},l}\times R^{\textnormal{rig}}_{0,l}\longrightarrow R^{\textnormal{rig}}_{0,l}$ by passing to completions. By continuity, it gives $R^{\textnormal{rig}}_{0,l}$ the structure of a topological Banach module over the $\breve{K}$-Banach algebra $\Lambda(\Gamma_{*})_{\breve{K},l}$. Taking projective limits over $l$, we obtain a continuous map $D(\Gamma_{*})\times R^{\textnormal{rig}}_{0}\longrightarrow R^{\textnormal{rig}}_{0}$ giving $R^{\textnormal{rig}}_{0}$ the structure of a continuous module over $D(\Gamma_{*})$. Because of the topological isomorphism (\ref{top_iso_of_distribution_alg}), $R^{\textnormal{rig}}_{0}$ becomes a continuous module over $D(\Gamma_{\mathbb{Q}_{p}})$.\\
\indent  To see that the $D(\Gamma_{\mathbb{Q}_{p}})$-action on $R^{\textnormal{rig}}_{0}$ factors through a continuous action of $D(\Gamma)$, it suffices to check  that $i(\lambda\mathfrak{x})(f)=\lambda i(\mathfrak{x})(f)$ for all $\lambda\in K$, $\mathfrak{x}\in\mathfrak{g}_{\mathbb{Q}_{p}}\subseteq D(\Gamma_{\mathbb{Q}_{p}})$ and $f\in R^{\textnormal{rig}}_{0}$. Here we make use of Theorem \ref{laofGonMD}. Being a locally $K$-analytic $\Gamma$-representation, $\mathcal{O}_{X^{\textnormal{rig}}_{0}}(D)$ is a $D(\Gamma)$-module and thus carries an action of the Lie algebra $\mathfrak{g}$ (cf. \cite{stladist}, Proposition 3.2). Thus, $i(\lambda\mathfrak{x})(f)=\lambda i(\mathfrak{x})(f)$ holds for all $f\in \mathcal{O}_{X^{\textnormal{rig}}_{0}}(D)$. As the $K$-linear inclusion $R^{\text{rig}}_{0}\hookrightarrow \mathcal{O}_{X^{\textnormal{rig}}_{0}}(D)$ is continuous, it is $\mathfrak{g}_{\mathbb{Q}_{p}}$-equivariant. Hence the equality $i(\lambda\mathfrak{x})(f)=\lambda i(\mathfrak{x})(f)$ is also true for all $f\in R^{\text{rig}}_{0}$.\\
\indent  Now it follows from \cite{schnfa}, Proposition 19.9 and the arguments proving the claim on page 98, that the $\breve{K}$-Fr\'{e}chet space $R^{\textnormal{rig}}_{0}$ is nuclear. Therefore, \cite{stladist}, Corollary 3.4 implies that the locally convex $\breve{K}$-vector space $(R^{\textnormal{rig}}_{0})'_{b}$ is of compact type and that the action of $\Gamma$ obtained by dualizing is locally $K$-analytic. 
\end{proof}
\indent  The preceding theorem can be generalized as follows. Let $\Gamma_{*}$ be a uniform pro-$p$ group contained in $\Gamma_{2n+1}$ for some positive integer $n$.
\begin{theorem}\label{genlaK0}
The action of $\Gamma_{\mathbb{Q}_{p}}$ on $M^{s}_{0}$ extends to a continuous action of the $\breve{K}$-Fr\'{e}chet algebra $D(\Gamma_{\mathbb{Q}_{p}})$, which then factors through a continuous action of $D(\Gamma)$ on $M^{s}_{0}$. Hence the action of $\Gamma$ on the strong continuous $\breve{K}$-linear dual $(M^{s}_{0})'_{b}$ of $M^{s}_{0}$ is locally $K$-analytic for any $s\in\mathbb{Z}$. 
\end{theorem}
\begin{proof}
Choose a generator $\delta$ of the $R_{0}$-module $\textnormal{Lie}(\mathbb{H}^{(0)})^{\otimes s}$. Then by (\ref{G_eqv_iso_eqn1}) and (\ref{G_eqv_iso_eqn3}), $M^{s}_{0}=R^{\textnormal{rig}}_{0}\delta$ and $M^{s}_{0,l}=R^{\textnormal{rig}}_{0,l}\delta$. The topology on $M^{s}_{0,l}$ is defined by the norm $\| f\delta\|_{l}:=\|f\|_{l}$. Let $\gamma(\delta)=f_{0}\delta$ then by $\Gamma$-equivariance, we have $\gamma(f\delta)=\gamma(f)\gamma(\delta)=\gamma(f)f_{0}\delta$ for all $f\delta\in M^{s}_{0}$. Hence \begin{align}\label{simplifying_eqn1}
\gamma(f\delta)-f\delta=(\gamma(f)f_{0}-f)\delta&=(\gamma(f)f_{0}-ff_{0}+ff_{0}-f)\delta\\&=((\gamma(f)-f)f_{0}+f(f_{0}-1))\delta\nonumber
\end{align} Now if $\gamma\in\Gamma_{*}\subseteq\Gamma_{2n+1}$ and if $f\delta\in M^{s}_{0}$, then $\|\gamma(f)-f\|_{l}\leq |\varpi|^{\frac{2n+1}{l}}\|f\|_{l}$ by Proposition \ref{mainprop} and $\gamma(\delta)-\delta=(f_{0}-1)\delta\in\mathfrak{m}^{n+1}_{R_{0}}\textnormal{Lie}(\mathbb{H}^{(0)})^{\otimes s}$ by Theorem \ref{ctsthm2}, i.e., $f_{0}-1\in\mathfrak{m}_{R_{0}}^{n+1}$. Since $\| y\|_{l}\leq\vert\varpi\vert^{1/l}$ for any $y\in\mathfrak{m}_{R_{0}}=(\varpi,u_{1},\cdots,u_{h-1})$, $\|f_{0}-1\|_{l}\leq|\varpi|^{\frac{n+1}{l}}$ and $\|f_{0}\|_{l}\leq\max\lbrace\|f_{0}-1\|_{l},1\rbrace=1$.  
Thus by the multiplicativity of the norm $\|\cdot\|_{l}$ on $R^{\textnormal{rig}}_{0}$ and by (\ref{simplifying_eqn1}), we have \begin{align*}
\|\gamma(f\delta)-f\delta\|_{l}&=\|(\gamma(f)-f)f_{0}+f(f_{0}-1)\|_{l}\\&\leq\max\lbrace\|(\gamma(f)-f)\|_{l}\|f_{0}\|_{l},\|f\|_{l}\|(f_{0}-1)\|_{l}\rbrace\\&\leq\max\lbrace|\varpi|^{\frac{2n+1}{l}}\|f\|_{l},|\varpi|^{\frac{n+1}{l}}\|f\|_{l}\rbrace\\&=|\varpi|^{\frac{n+1}{l}}\|f\|_{l}=|\varpi|^{\frac{n+1}{l}}\|f\delta\|_{l}
\end{align*}
The rest of the proof now proceeds along the same lines as for Theorem \ref{laK0}.
\end{proof}
\subsection{Local analyticity of the $\Gamma$-action on $M^{s}_{m}$ with $m>0$}
\label{step3}
\indent In this subsection, as the title indicates, we extend the theorems of the previous subsection to higher levels $m>0$. The following observation together with the continuity of the $\Gamma$-action on $R_{m}$ and on $R^{\text{rig}}_{0}$ (cf. Theorem \ref{ctsthm} and Proposition \ref{mainprop} respectively) allows us to show the continuity of the $\Gamma$-action on $R^{\text{rig}}_{m}$ for $m>0$.
\begin{lemma}\label{km}  For every $m\geq 0$, there exists a positive integer $k_{m}$ such that $\mathfrak{m}_{R_{m}}^{n}\subseteq\mathfrak{m}_{R_{0}}R_{m}$ for all $n\geq k_{m}$.
\end{lemma}
\begin{proof}
Since $R_{m}$ is a finite free module over $R_{0}$, $R_{m}/\mathfrak{m}_{R_{0}}R_{m}$ is a finite dimensional vector space over $R_{0}/\mathfrak{m}_{R_{0}}=\overline{k}$. Moreover, $R_{m}/\mathfrak{m}_{R_{0}}R_{m}$ is still a Noetherian local ring with the maximal ideal $\mathfrak{m}_{R_{m}}/\mathfrak{m}_{R_{0}}R_{m}$. The powers $(\mathfrak{m}_{R_{m}}/\mathfrak{m}_{R_{0}}R_{m})^{n}$, $n\in\mathbb{N}$, of the ideal $\mathfrak{m}_{R_{m}}/\mathfrak{m}_{R_{0}}R_{m}$ form a descending sequence of finite dimensional subspaces which eventually must become stationary. Let $k_{m}$ be a positive integer such that $(\mathfrak{m}_{R_{m}}/\mathfrak{m}_{R_{0}}R_{m})^{n+1}=(\mathfrak{m}_{R_{m}}/\mathfrak{m}_{R_{0}}R_{m})^{n}$ for all $n\geq k_{m}$. Then by Nakayama's lemma, for all $n\geq k_{m}$, $(\mathfrak{m}_{R_{m}}/\mathfrak{m}_{R_{0}}R_{m})^{n}=0$, in other words, $\mathfrak{m}_{R_{m}}^{n}\subseteq\mathfrak{m}_{R_{0}}R_{m}$.
\end{proof}
\begin{proposition}\label{mainprop2} Let $m$, $n$ and $l$ be integers with $m\geq 1$, $l\geq 1$ and $n\geq k_{m}-1$ where $k_{m}$ is as stated in Lemma \ref{km}. If $\gamma\in\Gamma_{n+m}$ and if $f\in R^{\textnormal{rig}}_{m}$, then $\|\gamma(f)-f\|_{l}\leq\vert \varpi\vert^{1/l}\| f\|_{l}$.
\end{proposition}
\begin{proof} Write $f=f_{1}e_{1}+\cdots +f_{r}e_{r}$ where $\lbrace e_{1},\cdots, e_{r}\rbrace$ is a basis of $R_{m}$ over $R_{0}$ and $f_{i}\in R^{\textnormal{rig}}_{0}$ for all $1\leq i\leq r$. Let $x_{i}:=\gamma(e_{i})-e_{i}$. Then $x_{i}\in\mathfrak{m}_{R_{m}}^{n+1}$ for all $1\leq i\leq r$ by Theorem \ref{ctsthm} and thus by Lemma \ref{km}, $x_{i}\in\mathfrak{m}_{R_{0}}R_{m}$ for all $1\leq i\leq r$. Since $\| y\|_{l}\leq\vert\varpi\vert^{1/l}$ for any $y\in\mathfrak{m}_{R_{0}}=(\varpi,u_{1},\cdots,u_{h-1})$, $\| x_{i}\|_{l}\leq\vert\varpi\vert^{1/l}$ for all $1\leq i\leq r$. Now note that
\begin{align*}
\|\gamma(f)-f\|_{l}&\leq\max_{1\leq i\leq r}\lbrace\|\gamma(f_{i}e_{i})-f_{i}e_{i}\|_{l}\rbrace =\max_{1\leq i\leq r}\lbrace\|\gamma(f_{i})\gamma(e_{i})-f_{i}e_{i}\|_{l}\rbrace\\&=\max_{1\leq i\leq r}\lbrace\|\gamma(f_{i})\gamma(e_{i})-f_{i}\gamma(e_{i})+f_{i}\gamma(e_{i})-f_{i}e_{i}\|_{l}\rbrace \\&=\max_{1\leq i\leq r}\lbrace\|(\gamma(f_{i})-f_{i})\gamma(e_{i})+(\gamma(e_{i})-e_{i})f_{i}\|_{l}\rbrace\\& =\max_{1\leq i\leq r}\lbrace\|(\gamma(f_{i})-f_{i})(e_{i}+x_{i})+x_{i}f_{i}\|_{l}\rbrace.  
\end{align*}
Then Lemma \ref{lnorm_is_algebra_norm} and Proposition \ref{mainprop} imply that for every $1\leq i\leq r$,
\begin{align*}\|(\gamma(f_{i})-f_{i})(e_{i}+x_{i})+x_{i}f_{i}\|_{l}&\leq\max\lbrace\|(\gamma(f_{i})-f_{i})(e_{i}+x_{i})\|_{l},\| x_{i}f_{i}\|_{l}\rbrace\\&\leq\max\lbrace\|(\gamma(f_{i})-f_{i})\|_{l},\| x_{i}\|_{l}\| f_{i}\|_{l}\rbrace\\&\leq\max\lbrace\vert\varpi\vert^{(n+m)/l}\| f_{i}\|_{l},\vert\varpi\vert^{1/l}\| f_{i}\|_{l}\rbrace\\&=\vert\varpi\vert^{1/l}\| f_{i}\|_{l}\end{align*} where we use that $e_{i}+x_{i}=\gamma(e_{i})\in R_{m}$ has $\|\cdot\|_{l}$-norm less than or equal to 1.
Therefore,  
$\|\gamma(f)-f\|_{l}\leq\max_{1\leq i\leq r}\lbrace\vert\varpi\vert^{1/l}\| f_{i}\|_{l}\rbrace=\vert \varpi\vert^{1/l}\| f\|_{l}$.
\end{proof}                                                                                                                                                                                              \indent We now arbitrarily fix a level $m\geq 1$. As before, we have a uniform pro-$p$ group $\Gamma_{o}$ of rank $t$ as an open subgroup of $\Gamma_{\mathbb{Q}_{p}}$. We also fix a positive integer $n\geq k_{m}-1$ such that $\Gamma_{n+m}\subseteq\Gamma_{o}$. Then $\Gamma_{n+m}$ contains $\Gamma_{*}:=P_{i}(\Gamma_{o})$ for some $i\geq 1$ which is also a uniform pro-$p$ group of rank $t$.\\
\indent  Let $\lbrace\gamma_{1},\dots ,\gamma_{t}\rbrace$ be an ordered basis of $\Gamma_{*}$ and let $b_{i}:=\gamma_{i}-1\in\Lambda(\Gamma_{*})$. Then as before, we equip the $\breve{K}$-algebra $\Lambda(\Gamma_{*})_{\breve{K}}$ with the sub-multiplicative norm $\|\cdot\|_{l}$ defined in (\ref{lnorm}) for every positive integer $l$. The natural inclusions $\Lambda(\Gamma_{*})_{\breve{K},l+1}\hookrightarrow\Lambda(\Gamma_{*})_{\breve{K},l}$ of $\breve{K}$-Banach completions  endow the projective limit $D(\Gamma_{*})=\varprojlim_{l}\Lambda(\Gamma_{*})_{\breve{K},l}$ with the structure of a $\breve{K}$-Fr\'{e}chet algebra which is equal to the algebra of $\breve{K}$-valued locally $K$-analytic distributions on $\Gamma_{*}$.
\begin{proposition}\label{lam} For any integer $l\geq 1$, the action of $\Gamma_{*}$ on $R^{\textnormal{rig}}_{m}$ extends to $R^{\textnormal{rig}}_{m,l}$ and makes $R^{\textnormal{rig}}_{m,l}$ a topological Banach module over the $\breve{K}$-Banach algebra $\Lambda(\Gamma_{*})_{\breve{K},l}$. The action of $\Gamma_{\mathbb{Q}_{p}}$ on $R^{\textnormal{rig}}_{m}$ extends to a continuous action of the $\breve{K}$-Fr\'{e}chet algebra $D(\Gamma_{\mathbb{Q}_{p}})$. 
\end{proposition}
\begin{proof} The proof is similar to that of Theorem \ref{laK0}. First, we prove by induction on $\vert \alpha\vert$ that $\| b^{\alpha}(f)\|_{l}\leq\| b^{\alpha}\|_{l}\| f\|_{l}$ for any $f\in R^{\textnormal{rig}}_{m}$. This is clear if $\vert \alpha\vert=0$. Let $|\alpha|>0$ and let $i$ be the minimal index such that $\alpha_{i}>0$. Define $\beta_{j}:=\alpha_{j}$ if $j\neq i$ and $\beta_{i}:=\alpha_{i}-1$. Then Proposition \ref{mainprop2} and the induction hypothesis imply \begin{align*}
\| b^{\alpha}(f)\|_{l}&=\| ((\gamma_{i}-1)b^{\beta})(f)\|_{l}=\| (\gamma_{i}-1)(b^{\beta}(f))\|_{l}\leq\vert\varpi\vert^{1/l}\| b^{\beta}(f)\|_{l}\\&\leq\vert\varpi\vert^{1/l}\| b^{\beta}\|_{l}\| f\|_{l}\leq\vert\varpi\vert^{1/l}\vert\varpi\vert^{\vert \beta\vert /l}\| f\|_{l}=\vert\varpi\vert^{(\vert \beta\vert+1) /l}\| f\|_{l}=\| b^{\alpha}\|_{l}\| f\|_{l}  
\end{align*} as required. By Remark \ref{iwasawa_alg_action}, this immediately gives $\|\mu(f)\|_{l}\leq\|\mu\|_{l}\| f\|_{l}$ for all $\mu\in\Lambda(\Gamma_{*})_{\breve{K}}$ and $f\in R_{m}[\frac{1}{\varpi}]=R_{m}\otimes_{\breve{\mathfrak{o}}}\breve{K}$. Hence the map $\Lambda(\Gamma_{*})_{\breve{K}}\times R_{m}[\frac{1}{\varpi}]\longrightarrow R_{m}[\frac{1}{\varpi}]$ $((\mu,f)\mapsto\mu(f))$ is continuous if $\Lambda(\Gamma_{*})_{\breve{K}}$ and $R_{m}[\frac{1}{\varpi}]$ are endowed with the respective $\|\cdot\|_{l}$-topologies and if the left hand side carries the product topology. Since $R_{m}[\frac{1}{\varpi}]$ is dense in $R^{\textnormal{rig}}_{m,l}$, we obtain a map $\Lambda(\Gamma_{*})_{\breve{K},l}\times R^{\textnormal{rig}}_{m,l}\longrightarrow R^{\textnormal{rig}}_{m,l}$ by passing to completions. By continuity, it gives $R^{\textnormal{rig}}_{m,l}$ the structure of a topological Banach module over the $\breve{K}$-Banach algebra $\Lambda(\Gamma_{*})_{\breve{K},l}$.\\
\indent  Taking projective limits over $l$, we obtain a continuous map $D(\Gamma_{*})\times R^{\textnormal{rig}}_{m}\longrightarrow R^{\textnormal{rig}}_{m}$ giving $R^{\textnormal{rig}}_{m}$ the structure of a continuous module over $D(\Gamma_{*})$. As $D(\Gamma_{\mathbb{Q}_{p}})$ is topologically isomorphic to the locally convex direct sum $\bigoplus_{\gamma\Gamma_{*}\in \Gamma_{\mathbb{Q}_{p}}/\Gamma_{*}}\delta_{\gamma} D(\Gamma_{*})$, $R^{\textnormal{rig}}_{m}$ is a continuous module over $D(\Gamma_{\mathbb{Q}_{p}})$.
 \end{proof}
\indent  We want to show that the $D(\Gamma_{\mathbb{Q}_{p}})$-action on $R^{\text{rig}}_{m}$ factors through a continuous $D(\Gamma)$-action. As mentioned in the introduction, the idea is to use the local $K$-analyticity of the $\Gamma$-action at level $m=0$ obtained in Theorem \ref{laK0} and the \'{e}taleness property of the extension $R^{\text{rig}}_{m}|R^{\text{rig}}_{0}$.\\
\indent  For a ring homomorphism $A\to B$ of commutative unital rings, let $\textnormal{Der}_{A}(B,B)$ denote the $B$-module of $A$-linear derivations from $B$ to $B$, and let $\Omega_{B/A}$ denote the $B$-module of differentials of $B$ over $A$. 
 \begin{lemma}\label{derivationextensionlemma} Any $\breve{K}$-linear derivation from $R^{\textnormal{rig}}_{0}$ to $R^{\textnormal{rig}}_{0}$ extends uniquely to a $\breve{K}$-linear derivation from $R^{\textnormal{rig}}_{m}$ to $R^{\textnormal{rig}}_{m}$.
\end{lemma}
\begin{proof}
Since $R_{m}[\frac{1}{\varpi}]$ is \'{e}tale over $R_{0}[\frac{1}{\varpi}]$ (cf. Theorem 2.1.2 (ii), \cite{strdeform}), $R^{\textnormal{rig}}_{m}\cong R_{m}[\frac{1}{\varpi}]\otimes_{R_{0}[\frac{1}{\varpi}]}R^{\textnormal{rig}}_{0}$ is \'{e}tale over $R^{\textnormal{rig}}_{0}$ by \cite[Tag 00U0]{stacks-project}, and so is formally \'{e}tale by \cite[Tag 00UR]{stacks-project}. Then, using \cite[Tag 031K]{stacks-project} and \cite[Tag 00UO]{stacks-project}, we get that $\Omega_{R^{\textnormal{rig}}_{m}/\breve{K}}\cong\Omega_{R^{\textnormal{rig}}_{0}/\breve{K}}\otimes_{R^{\textnormal{rig}}_{0}}R^{\textnormal{rig}}_{m}$. Therefore \begin{align*}
\textnormal{Der}_{\breve{K}}(R^{\textnormal{rig}}_{0},R^{\textnormal{rig}}_{0})&\cong\textnormal{Hom}_{R^{\textnormal{rig}}_{0}}(\Omega_{R^{\textnormal{rig}}_{0}/\breve{K}},R^{\textnormal{rig}}_{0})\\&\hookrightarrow\textnormal{Hom}_{R^{\textnormal{rig}}_{0}}(\Omega_{R^{\textnormal{rig}}_{0}/\breve{K}},R^{\textnormal{rig}}_{m})\\&\cong\textnormal{Hom}_{R^{\textnormal{rig}}_{m}}(\Omega_{R^{\textnormal{rig}}_{0}/\breve{K}}\otimes_{R^{\textnormal{rig}}_{0}}R^{\textnormal{rig}}_{m},R^{\textnormal{rig}}_{m})\\&\cong\textnormal{Hom}_{R^{\textnormal{rig}}_{m}}(\Omega_{R^{\textnormal{rig}}_{m}/\breve{K}},R^{\textnormal{rig}}_{m})\\&\cong\textnormal{Der}_{\breve{K}}(R^{\textnormal{rig}}_{m},R^{\textnormal{rig}}_{m}).
\end{align*}
\end{proof}
\begin{theorem}\label{laKm}
The action of $D(\Gamma_{\mathbb{Q}_{p}})$ on $R^{\textnormal{rig}}_{m}$ factors through a continuous action of $D(\Gamma)$ on $R^{\textnormal{rig}}_{m}$. Hence the action of $\Gamma$ on the strong continuous $\breve{K}$-linear dual $(R^{\textnormal{rig}}_{m})'_{b}$ of $R^{\textnormal{rig}}_{m}$ is locally $K$-analytic.
\end{theorem}
\begin{proof}
Recall the inclusion map $i:\mathfrak{g}_{\mathbb{Q}_{p}}\hookrightarrow D(\Gamma_{\mathbb{Q}_{p}})$ from the discussion before \ref{laK0}. For every $\mathfrak{x}\in\mathfrak{g}_{\mathbb{Q}_{p}}$, $i(\mathfrak{x})\in D(\Gamma_{\mathbb{Q}_{p}})$ acts on $R^{\textnormal{rig}}_{m}$ as a $\breve{K}$-linear derivation from $R^{\text{rig}}_{m}$ to $R^{\text{rig}}_{m}$. Let $\lambda\in K$ and $\mathfrak{x}\in\mathfrak{g}_{\mathbb{Q}_{p}}$ be arbitrary, and consider the distibution $i(\lambda\mathfrak{x})-\lambda i(\mathfrak{x})\in D(\Gamma_{\mathbb{Q}_{p}})$. It gives rise to a zero derivation on $R^{\textnormal{rig}}_{0}$ by Theorem \ref{laK0} and thus by Lemma \ref{derivationextensionlemma}, it is also zero on $R^{\textnormal{rig}}_{m}$. This means that the action of $D(\Gamma_{\mathbb{Q}_{p}})$ on $R^{\textnormal{rig}}_{m}$ factors through a continuous action of $D(\Gamma)$ on $R^{\textnormal{rig}}_{m}$.\\
\indent  As the $\breve{K}$-Fr\'{e}chet space $R^{\textnormal{rig}}_{m}$ is topologically isomorphic to $\bigoplus_{i=1}^{r} R^{\textnormal{rig}}_{0}$, it follows from \cite{schnfa}, Proposition 19.7, that $R^{\textnormal{rig}}_{m}$ is nuclear. Therefore, \cite{stladist}, Corollary 3.4 implies that the locally convex $\breve{K}$-vector space $(R^{\textnormal{rig}}_{m})'_{b}$ is of compact type and that the action of $\Gamma$ obtained by dualizing is locally $K$-analytic.
\end{proof}
\indent  Similar to the before, Theorem \ref{laKm} can be generalized as follows. Fix $m\geq 1$, $n\geq k_{m}-1$ and a uniform pro-$p$ group $\Gamma_{*}\subseteq\Gamma_{2n+m+1}$ with $k_{m}$ as in Lemma \ref{km}.
\begin{theorem}\label{genlaKm}
The action of $\Gamma_{\mathbb{Q}_{p}}$ on $M^{s}_{m}$ extends to a continuous action of the $\breve{K}$-Fr\'{e}chet algebra $D(\Gamma_{\mathbb{Q}_{p}})$, which then factors through a continuous action of $D(\Gamma)$ on $M^{s}_{m}$. Hence the action of $\Gamma$ on the strong continuous $\breve{K}$-linear dual $(M^{s}_{m})'_{b}$ of $M^{s}_{m}$ is locally $K$-analytic for any $s\in\mathbb{Z}$. 
\end{theorem}
\begin{proof}
Using Theorem \ref{ctsthm2}, Lemma \ref{km} and Proposition \ref{mainprop2}, the proof of the first part of the assertion is similar to that of Theorem \ref{genlaK0}.\\
\indent  Observe that the isomorphism $\text{Lie}(\mathbb{H}^{(m)})\cong\text{Lie}(\mathbb{H}^{(0)})\otimes_{R_{0}}R_{m}$ is $\Gamma$ -equivariant for the diagonal $\Gamma$-action on the right. Therefore, we have the following $\Gamma$-equivariant isomorphisms by (\ref{G_eqv_iso_eqn1}) 
\begin{align*}
M^{s}_{m}&\cong R^{\text{rig}}_{m}\otimes_{R_{m}}\textnormal{Lie}(\mathbb{H}^{(m)})\\&\cong R^{\text{rig}}_{m}\otimes_{R_{0}}\textnormal{Lie}(\mathbb{H}^{(0)})\\&\cong R^{\text{rig}}_{m}\otimes_{R^{\text{rig}}_{0}}R^{\text{rig}}_{0}\otimes_{R_{0}}\textnormal{Lie}(\mathbb{H}^{(0)})\\&\cong R^{\text{rig}}_{m}\otimes_{R^{\text{rig}}_{0}}M^{s}_{0}
\end{align*} with $\Gamma$-acting diagonally on all the tensor products. As a consequence, the $\mathfrak{g}_{\mathbb{Q}_{p}}$-action on $f\otimes\delta\in M^{s}_{m}$ is given by $\mathfrak{x}(f\otimes\delta)=\mathfrak{x}(f)\otimes\delta+f\otimes\mathfrak{x}(\delta)$. However, by Theorem \ref{genlaK0} and Theorem \ref{laKm}, $M^{s}_{0}$ and $R^{\text{rig}}_{m}$ are not only $\mathfrak{g}_{\mathbb{Q}_{p}}$-modules but also $\mathfrak{g}$-modules. Thus, it follows that the $D(\Gamma_{\mathbb{Q}_{p}})$-action on $M^{s}_{m}$ factors through a continuous $D(\Gamma)$-action as required. 
\end{proof}
\begin{remark} The method described above can also be used to show the local analyticity of the $\Gamma$-action on the dual space of the global sections of a large class of \emph{Drinfeld bundles} over the Lubin-Tate moduli space (cf. \cite{Koh11}, Section 3). If $V$ is a finite dimensional $\breve{K}$-linear smooth representation of $\Gamma$ then, for any $m\geq 0$, the free $R^{\text{rig}}_{m}$-module $R^{\text{rig}}_{m}\otimes_{\breve{K}}V$ with the diagonal $\Gamma$-action induces a $\Gamma$-equivariant vector bundle over $X^{\text{rig}}_{m}$ whose Fr\'{e}chet space of global sections is $R^{\text{rig}}_{m}\otimes_{\breve{K}}V$ (cf. \cite{Koh11}, Theorem 1.2 and Corollary A.3). By choosing a basis $\lbrace b_{1}, b_{2},\ldots,b_{d}\rbrace$ of $V$ over $\breve{K}$, the Fr\'{e}chet topology of $R^{\text{rig}}_{m}\otimes_{\breve{K}}V$ can be given by the family of the norms $\|f_{1}b_{1}+f_{2}b_{2}+\dots +f_{d}b_{d}\|_{l}:=\max_{1\leq i\leq d}\lbrace\|f_{i}\|_{l}\rbrace$ with $l\in\mathbb{N}$, where the $l$-norms on $R^{\text{rig}}_{m}$ are as defined in Section \ref{Rigidification and the equivariant vector bundles}. It then follows from Proposition \ref{mainprop} and Proposition \ref{mainprop2} that if $\gamma$ belongs to a small enough open subgroup of $\Gamma$ ($\subseteq$ intersection of stabilizers of $b_{i}$'s) then \begin{align*}
\|\gamma(f_{1}b_{1}+f_{2}b_{2}+\dots &+f_{d}b_{d})-(f_{1}b_{1}+f_{2}b_{2}+\dots +f_{d}b_{d})\|_{l}\\&=\|(\gamma(f_{1})-f_{1})b_{1}+(\gamma(f_{2})-f_{2})b_{2}+\dots +(\gamma(f_{d})-f_{d})b_{d}\|_{l}\\&=\max_{1\leq i\leq d}\lbrace\|\gamma(f_{i})-f_{i}\|_{l}\rbrace\\&\leq|\varpi|^{1/l}\max_{1\leq i\leq d}\lbrace\|f_{i}\|_{l}\rbrace\\&=|\varpi|^{1/l}\|f_{1}b_{1}+f_{2}b_{2}+\dots +f_{d}b_{d}\|_{l}.
\end{align*}
Therefore, similar to the the proof of Theorem \ref{laK0}, the $\Gamma$-action on $R^{\text{rig}}_{m}\otimes_{\breve{K}}V$ extends to a continuous action of the Fr\'{e}chet algebra $D(\Gamma_{\mathbb{Q}_{p}})$. Since $R^{\text{rig}}_{m}$ is a $\mathfrak{g}$-module and $V$ is annihilated by $\mathfrak{g}$, the $D(\Gamma_{\mathbb{Q}_{p}})$-action on $R^{\text{rig}}_{m}\otimes_{\breve{K}}V$ factors through a continuous action of $D(\Gamma)$ making its strong topological dual a locally $K$-analytic $\Gamma$-representation.  
\end{remark}
\section{Locally finite vectors in the global sections of equivariant vector bundles}
\indent  This section is devoted to study representation-theoretic aspects of the \linebreak $\Gamma$-representations $M^{s}_{m}$ which include a complete description of the $\Gamma$-locally finite (algebraic) vectors in $M^{s}_{m}$ for all $s\in\mathbb{Z}$ and $m\geq 0$.
\subsection{Locally finite vectors in the $\Gamma$-representations $M^{s}_{0}$}
\begin{definition} Let $G$ be a topological group and $V$ be a vector space over a field $F$ equipped with an $F$-linear $G$-action. We say that a vector $v\in V$ is \emph{locally finite} (or \emph{$G$-locally finite}) if there is an open subgroup $H$ of $G$ and a finite dimensional $H$-stable subspace $W$ of $V$ containing $v$.\footnote{In \cite{eme04}, Proposition-Definition 4.1.8, the notion of a locally finite vector is defined for the vector spaces over a complete non-archimedean field $F$, and requires locally finite vector $v$ to be contained in a \emph{continuous} finite dimensional $H$-representation $W$ for its natural Hausdorff topology as a finite dimensional $F$-vector space. Since all the $\Gamma$-representations we are concerned with in this section are continuous representations on $\breve{K}$-Fr\'{e}chet spaces, the continuity condition is automatically satisfied.} It follows easily that the set $V_{\textnormal{lf}}$ of all locally finite vectors of $V$ forms a $G$-stable subspace. We call $V$ a locally finite representation of $G$ if $V_{\textnormal{lf}}=V$. If $V$ and $W$ are $F$-linear $G$-representations, and if $f:V\longrightarrow W$ is an $F$-linear $G$-equivariant map, then clearly $f(V_{\textnormal{lf}})\subseteq W_{\textnormal{lf}}$.
\end{definition}
\indent  To calculate locally finite vectors, we make extensive use of the Lie algebra action. Let $U(\mathfrak{g})$ be the universal enveloping algebra of the Lie algebra $\mathfrak{g}$ of $\Gamma$ over $K$. Note that $\mathfrak{g}$ is isomorphic to the Lie algebra associated with the associative $K$-algebra $B_{h}$. Thus, $\mathfrak{g}\hookrightarrow\mathfrak{g}\otimes_{K}K_{h}\cong\mathfrak{gl}_{h}(K_{h})$. Denote by $\mathfrak{x}_{ij}\in\mathfrak{gl}_{h}(K_{h})$ the matrix with entry 1 at the place $(i,j)$ and zero everywhere else. By Theorem \ref{laofGonMD}, the $\Gamma$-representation $M^{s}_{D}\cong\mathcal{O}_{X^{\textnormal{rig}}_{0}}(D)\varphi_{0}^{s}$ carries a continuous linear action of $U(\mathfrak{g})\otimes_{K}\breve{K}\cong U(\mathfrak{gl}_{h}(K_{h}))\otimes_{K_{h}}\breve{K}\hookrightarrow D(\Gamma)$. Since $GL_{h}(K_{h})$ acts on the projective coordinates $\varphi_{0},\ldots,\varphi_{h-1}$ by fractional linear transformations, one can explicitly determine this Lie algebra action using the formula $\mathfrak{x}(f)=\frac{d}{dt}\text{exp}(t\mathfrak{x})(f)|_{t=0}$.
\begin{lemma}\label{explicit_g-action_lemma} Let $i$, $j$ and $s$ be integers with $0\leq i,j\leq h-1$. Put $w_{0}:=1$. If $f\in\mathcal{O}_{X^{\textnormal{rig}}_{0}}(D)$ then 
\begin{equation}\label{explicit_g-action_formulae}
\mathfrak{x}_{ij}(f\varphi^{s}_{0})=
\begin{cases}
w_{i}\frac{\partial f}{\partial w_{j}}\varphi^{s}_{0}, &\text{if $j\neq 0$;}\\
(sf-\sum_{l=1}^{h-1}w_{l}\frac{\partial f}{\partial w_{l}})\varphi^{s}_{0}, &\text{if $i=j=0$;}\\
w_{i}(sf-\sum_{l=1}^{h-1}w_{l}\frac{\partial f}{\partial w_{l}})\varphi^{s}_{0}, &\text{if $i>j=0$.}
\end{cases}
\end{equation}
\end{lemma}
\begin{proof}
This is exactly same as \cite{kohliwamo}, Lemma 4.1, which treats the case $K=\mathbb{Q}_{p}$. 
\end{proof} 
\indent  Given a Lie subalgebra $\mathfrak{h}\subseteq\mathfrak{gl}_{h}(K_{h})$, and a $\breve{K}$-linear $\mathfrak{gl}_{h}(K_{h})$-representation $W$, the $\breve{K}$-subspace of \emph{$\mathfrak{h}$-invariants of $W$} is the subspace \textnormal{$\lbrace w\in W|\mathfrak{x}(w)=0$ for all $\mathfrak{x}\in\mathfrak{h}\rbrace$} of $W$. Let us denote by $\mathfrak{n}$ the Lie subalgebra of $\mathfrak{gl}_{h}(K_{h})$ consisting of strictly upper triangular matrices. For later use, we calculate the $\mathfrak{g}$-invariants and the $\mathfrak{n}$-invariants of $\mathcal{O}_{X^{\textnormal{rig}}_{0}}(D)$ in the next lemma using the formulae (\ref{explicit_g-action_formulae}).
\begin{lemma}\label{gninvO(D)}
$\mathcal{O}_{X^{\textnormal{rig}}_{0}}(D)^{\mathfrak{g}=0}=\mathcal{O}_{X^{\textnormal{rig}}_{0}}(D)^{\mathfrak{n}=0}=\breve{K}$.
\end{lemma}
\begin{proof}
Since $\breve{K}\subseteq\mathcal{O}_{X^{\textnormal{rig}}_{0}}(D)^{\mathfrak{g}=0}\subseteq\mathcal{O}_{X^{\textnormal{rig}}_{0}}(D)^{\mathfrak{n}=0}$, it suffices to show that the latter is $\breve{K}$. Now if $f\in\mathcal{O}_{X^{\textnormal{rig}}_{0}}(D)^{\mathfrak{n}=0}$ then applying the formulae (\ref{explicit_g-action_formulae}), we get $\mathfrak{x}_{0j}(f)=\frac{\partial f}{\partial w_{j}}=0$ for all $1\leq j\leq h-1$. Therefore, $f$ must be a constant power series.
\end{proof}
\indent  We now compute the space $(M^{s}_{D})_{\text{lf}}$ of locally finite vectors in the $\Gamma$-representation $M^{s}_{D}$. The key step is the following lemma based on Lemma \ref{explicit_g-action_lemma}:
\begin{lemma}\label{polysubspace_contained_in_homopolygensubspace} The subspace $\breve{K}[w_{1},\dots ,w_{h-1}]\varphi^{s}_{0}$ of $M^{s}_{D}$ is contained in $(U(\mathfrak{g})\otimes_{K}\breve{K})(f\varphi_{0}^{s})$ for any non-zero homogeneous polynomial $f\in\breve{K}[w_{1},\dots ,w_{h-1}]$ of total degree $d>s$. 
\end{lemma} 
\begin{proof} Using (\ref{explicit_g-action_formulae}), we have for all $0<i,j\leq h-1$ \begin{align*}
&\mathfrak{x}_{0j}(f\varphi^{s}_{0})=\frac{\partial f}{\partial w_{j}}\varphi^{s}_{0}\\
&\mathfrak{x}_{i0}(f\varphi^{s}_{0})=(s-d)w_{i}f\varphi_{0}^{s}.
\end{align*} 
To obtain $g\varphi^{s}_{0}$ with a monomial $g$ of total degree $\leq s$, first reduce $f\varphi^{s}_{0}$ to $\varphi^{s}_{0}$ by applying suitable $\mathfrak{x}_{0j}$ $(j\neq 0)$ to it iteratively and then apply appropriate $\mathfrak{x}_{i0}$ $(i>0)$ to $\varphi^{s}_{0}$ to get the desired element $g\varphi^{s}_{0}$. To obtain $g\varphi^{s}_{0}$ with a monomial $g$ of total degree $>s$, reverse the procedure, i.e., first apply appropriate $\mathfrak{x}_{i0}$ $(i>0)$ to $f\varphi^{s}_{0}$ and then reduce the result to $g\varphi^{s}_{0}$ by applying suitable $\mathfrak{x}_{0j}$ $(j\neq 0)$ to it.
\end{proof} 
\indent  For any $s\in\mathbb{Z}$, we define the $\breve{K}$-subspace of $M^{s}_{D}$ by \begin{equation*}
V_{s}:=\sum_{\vert \alpha\vert\leq s}\breve{K}w^{\alpha}\varphi^{s}_{0}.
\end{equation*} Note that we have $V_{s}=0$ if $s<0$. For $s\geq 0$, it is easy to see that $V_{s}$ stable under the action of $\Gamma$. To see this, it is sufficient to prove that $\gamma(w^{\alpha}\varphi^{s}_{0})\in V_{s}$ for any $w^{\alpha}\varphi^{s}_{0}$ with $\vert \alpha\vert\leq s$ and for any $\gamma=\sum_{i=0}^{h-1}\lambda_{i}\Pi^{i}\in\Gamma$. In this case, using the action of the matrix (\ref{Gamma_as_a_subgroup_of_GLh}) on the projective coordinates $[\varphi_{0}:\ldots:\varphi_{h-1}]$, we find \begin{align*} &\gamma(w^{\alpha}\varphi^{s}_{0})=\gamma(w_{1}^{\alpha_{1}}\dots w_{h-1}^{\alpha_{h-1}}\varphi^{s}_{0})\\&=\gamma(\varphi_{1}^{\alpha_{1}}\dots \varphi_{h-1}^{\alpha_{h-1}}\varphi^{s-\vert \alpha\vert}_{0})\\&=\gamma(\varphi_{1})^{\alpha_{1}}\dots\gamma(\varphi_{h-1})^{\alpha_{h-1}}\gamma(\varphi_{0})^{s-\vert \alpha\vert}\\&=(\varpi\lambda_{1}\varphi_{0}+\dots +\varpi\lambda_{2}^{\sigma^{h-1}}\varphi_{h-1})^{\alpha_{1}}\dots(\varpi\lambda_{h-1}\varphi_{0}+\dots +\lambda_{0}^{\sigma^{h-1}}\varphi_{h-1})^{\alpha_{h-1}}\\&\hspace{7.5cm} (\lambda_{0}\varphi_{0}+\dots +\lambda_{1}^{\sigma^{h-1}}\varphi_{h-1})^{s-\vert \alpha\vert}\\&=(\varpi\lambda_{1}+\dots +\varpi\lambda_{2}^{\sigma^{h-1}}w_{h-1})^{\alpha_{1}}\dots(\varpi\lambda_{h-1}+\dots +\lambda_{0}^{\sigma^{h-1}}w_{h-1})^{\alpha_{h-1}}\\&\hspace{7.5cm} (\lambda_{0}+\dots +\lambda_{1}^{\sigma^{h-1}}w_{h-1})^{s-\vert \alpha\vert}\varphi^{s}_{0}\\&\in V_{s}.
\end{align*}  
\begin{theorem}\label{top_finiteness_of_MsD_thm} The $\Gamma$-representation $M^{s}_{D}$ is topologically irreducible if $s<0$, and if $s\geq 0$ then $V_{s}$ is a topologically irreducible sub-representation of $M^{s}_{D}$ with topologically irreducible quotient $M^{s}_{D}/V_{s}$.
\end{theorem} 
\begin{proof}
\textbf{Case $s<0$}: Let $V$ be a non-zero closed $\Gamma$-stable subspace of $M^{s}_{D}$. Let $f_{0}\varphi^{s}_{0}\in V$ where $f_{0}=\sum_{\alpha\in\mathbb{N}_{0}^{h-1}}c_{\alpha}w^{\alpha}\neq 0$ and $d$ $(\geq 0)$ be the smallest natural number such that $c_{\alpha}\neq 0$ for some $\alpha\in\mathbb{N}_{0}^{h-1}$ with $\vert \alpha\vert=d$. Thus, $f_{0}\varphi^{s}_{0}=\Big(\sum_{\vert \alpha\vert =d}c_{\alpha}w^{\alpha}+\sum_{i=1}^{\infty}\sum_{\vert \alpha\vert =d+i}c_{\alpha}w^{\alpha}\Big)\varphi^{s}_{0}\in V$ where $\sum_{\vert \alpha\vert =d}c_{\alpha}w^{\alpha}\neq 0$. For $n>0$, define a sequence of elements of $M^{s}_{D}$ inductively as follows: 
\begin{equation*}
f_{n}\varphi^{s}_{0}:=\frac{1}{n}\Big((d+n-s)f_{n-1}\varphi^{s}_{0}+\mathfrak{x}_{00}(f_{n-1}\varphi^{s}_{0})\Big).\end{equation*} Since $V$ is closed and $\Gamma$-stable, $V$ is stable under the action of the Lie algebra and thus $f_{n}\in V$ for all $n\in\mathbb{N}_{0}$.\\
\indent  We prove by induction on $n$ that
\begin{equation}\label{simplifying_eqn2}
f_{n}\varphi^{s}_{0}=\Bigg(\sum_{\vert \alpha\vert =d}c_{\alpha}w^{\alpha}+\sum_{i=1}^{\infty}\sum_{\vert \alpha\vert =d+i}(-1)^{n}\binom{i-1}{n}c_{\alpha}w^{\alpha}\Bigg)\varphi^{s}_{0}.
\end{equation} Here the generalized binomial coefficients are defined by $\binom{x}{n}:=\frac{x(x-1)\dots (x-n+1)}{n!}$ for any $x\in\mathbb{Z}$ and $n\in\mathbb{N}_{0}$. The case $n=0$ is true by definition. Assuming that (\ref{simplifying_eqn2}) holds for $n-1$, we compute using Lemma \ref{explicit_g-action_lemma} that
\begin{align*}
&f_{n}\varphi^{s}_{0}\\&=\frac{1}{n}\Big((d+n-s)f_{n-1}\varphi^{s}_{0}+\mathfrak{x}_{00}(f_{n-1}\varphi^{s}_{0})\Big)\\&=\frac{1}{n}\Bigg((d+n-s)\Big(\sum_{\vert \alpha\vert =d}c_{\alpha}w^{\alpha}+\sum_{i=1}^{\infty}\sum_{\vert \alpha\vert =d+i}(-1)^{n-1}\binom{i-1}{n-1}c_{\alpha}w^{\alpha}\Big)\varphi^{s}_{0}\\&\ \ \ \ \ \ \ \ +\mathfrak{x}_{00}\Bigg(\Big(\sum_{\vert \alpha\vert =d}c_{\alpha}w^{\alpha}+\sum_{i=1}^{\infty}\sum_{\vert \alpha\vert =d+i}(-1)^{n-1}\binom{i-1}{n-1}c_{\alpha}w^{\alpha}\Big)\varphi^{s}_{0}\Bigg)\Bigg)\\&=\frac{1}{n}\Bigg((d+n-s)\Big(\sum_{\vert \alpha\vert =d}c_{\alpha}w^{\alpha}+\sum_{i=1}^{\infty}\sum_{\vert \alpha\vert =d+i}(-1)^{n-1}\binom{i-1}{n-1}c_{\alpha}w^{\alpha}\Big)\varphi^{s}_{0}\\&\ \ \ \ \  +\Bigg(\sum_{\vert \alpha\vert =d}(s-d)c_{\alpha}w^{\alpha}+\sum_{i=1}^{\infty}\sum_{\vert \alpha\vert =d+i}(s-(d+i))(-1)^{n-1}\binom{i-1}{n-1}c_{\alpha}w^{\alpha}\Bigg)\varphi^{s}_{0}\Bigg)\\
\end{align*}
\begin{align*}
&=\Bigg(\sum_{\vert \alpha\vert =d}c_{\alpha}w^{\alpha}+\sum_{i=1}^{\infty}\sum_{\vert \alpha\vert =d+i}(n-i)(-1)^{n-1}\frac{1}{n}\binom{i-1}{n-1}c_{\alpha}w^{\alpha}\Bigg)\varphi^{s}_{0}\\&=\Bigg(\sum_{\vert \alpha\vert =d}c_{\alpha}w^{\alpha}+\sum_{i=1}^{\infty}\sum_{\vert \alpha\vert =d+i}(-1)^{n}\binom{i-1}{n}c_{\alpha}w^{\alpha}\Bigg)\varphi^{s}_{0}
\end{align*}
\indent  We now claim that the sequence $f_{n}\varphi^{s}_{0}$ converges to $\Big(\sum_{\vert \alpha\vert =d}c_{\alpha}w^{\alpha}\Big)\varphi^{s}_{0}$ as $n$ tends to $\infty$ with respect to the norm $\|\cdot\|_{M^{s}_{D}}$ defined in the proof of Theorem \ref{laofGonMD} As $f_{0}\in\mathcal{O}_{X^{\textnormal{rig}}_{0}}(D)$, we know that given $\varepsilon >0$, there exists $N_{\varepsilon}\in\mathbb{N}$ such that for all $\alpha\in\mathbb{N}_{0}^{h-1}$ with $\vert \alpha\vert > N_{\varepsilon}$, we have  $\vert c_{\alpha}\vert \vert\varpi\vert^{\sum_{i=1}^{h-1}\alpha_{i}(1-\frac{i}{h})}<\varepsilon$. Therefore,  $\sup_{\vert \alpha\vert >N_{\varepsilon}}\vert c_{\alpha}\vert \vert\varpi\vert^{\sum_{i=1}^{h-1}\alpha_{i}(1-\frac{i}{h})}<\varepsilon$. Now for all $n>N_{\varepsilon}-d$, using (\ref{simplifying_eqn2}), we have \begin{align*}\Big\| f_{n}\varphi^{s}_{0}-\Big(\sum_{\vert \alpha\vert =d}c_{\alpha}w^{\alpha}\Big)\varphi^{s}_{0}\Big\|_{M^{s}_{D}}&=\Big\|\Big(\sum_{i=1}^{\infty}\sum_{\vert \alpha\vert =d+i}(-1)^{n}\binom{i-1}{n}c_{\alpha}w^{\alpha}\Big)\varphi^{s}_{0}\Big\|_{M^{s}_{D}}\\&=\Big\|\sum_{i=1}^{\infty}\sum_{\vert \alpha\vert =d+i}(-1)^{n}\binom{i-1}{n}c_{\alpha}w^{\alpha}\Big\|_{D}\\&=\Big\|\sum_{i=n+1}^{\infty}\sum_{\vert \alpha\vert =d+i}(-1)^{n}\binom{i-1}{n}c_{\alpha}w^{\alpha}\Big\|_{D}\\&\leq\Big\|\sum_{\vert \alpha\vert >d+n}c_{\alpha}w^{\alpha}\Big\|_{D}\\&=\sup_{\vert \alpha\vert >d+n}\vert c_{\alpha}\vert \vert\varpi\vert^{\sum_{i=1}^{h-1}\alpha_{i}(1-\frac{i}{h})}\\&\leq\sup_{\vert \alpha\vert >N_{\varepsilon}}\vert c_{\alpha}\vert \vert\varpi\vert^{\sum_{i=1}^{h-1}\alpha_{i}(1-\frac{i}{h})}<\varepsilon
\end{align*} Hence, $f_{n}\varphi^{s}_{0}$ converges to $\Big(\sum_{\vert \alpha\vert =d}c_{\alpha}w^{\alpha}\Big)\varphi^{s}_{0}$ as $n\to\infty$ and $\Big(\sum_{\vert \alpha\vert =d}c_{\alpha}w^{\alpha}\Big)\varphi^{s}_{0}\in V$ because $V$ is closed. Since $\sum_{\vert \alpha\vert =d}c_{\alpha}w^{\alpha}$ is a non-zero homogeneous polynomial of total degree $d\geq 0>s$, Lemma \ref{polysubspace_contained_in_homopolygensubspace} implies that $\breve{K}[w_{1},\dots ,w_{h-1}]\varphi^{s}_{0}\subseteq V$. Since $\breve{K}[w_{1},\dots ,w_{h-1}]\varphi^{s}_{0}$ is dense in $M^{s}_{D}=\mathcal{O}_{X^{\textnormal{rig}}_{0}}(D)\varphi_{0}^{s}$ and $V$ is closed, it follows that $V=M^{s}_{D}$. Hence, $M^{s}_{D}$ is topologically irreducible for all $s<0$.\\
\indent  \textbf{Case $s\geq 0$}: Let $V$ be a non-zero closed $\Gamma$-stable subspace of $V_{s}$. Then $V$ is stable under the action of the Lie algebra $\mathfrak{g}$ and thus it becomes a module over $U(\mathfrak{g})\otimes_{K}\breve{K}$. As mentioned in the proof of Lemma \ref{polysubspace_contained_in_homopolygensubspace}, any non-zero element $f\varphi^{s}_{0}$ of $V$ can be reduced to $\varphi^{s}_{0}$ by applying suitable $\mathfrak{x}_{0j}$ $(j\neq 0)$ to it iteratively and then $\varphi^{s}_{0}$ can be converted into any monomial of total degree $\leq s$ multiplied with $\varphi^{s}_{0}$ by applying appropriate $\mathfrak{x}_{i0}$ $(i>0)$ to it. Therefore $V=V_{s}$ and $V_{s}$ is topologically irreducible.\\
\indent  Now let $\phi :M^{s}_{D}\longrightarrow M^{s}_{D}/V_{s}$ be the canonical surjective map and let $W\subset M^{s}_{D}/V_{s}$ be a non-zero, closed $\Gamma$-stable subspace. Then $\phi^{-1}(W)$ is a non-zero, closed $\Gamma$-stable subspace of $M^{s}_{D}$ not equal to $V_{s}$. Let $\Big(\sum_{\alpha\in\mathbb{N}_{0}^{h-1}}c_{\alpha}w^{\alpha}\Big)\varphi^{s}_{0}+V_{s}$ be a non-zero element of $W$. Then $f_{0}\varphi^{s}_{0}:=\Big(\sum_{\vert \alpha\vert>s}c_{\alpha}w^{\alpha}\Big)\varphi^{s}_{0}\neq 0\in\phi^{-1}(W)$. Let $d>s$ be the smallest natural number such that $c_{\alpha}\neq 0$ for some $\alpha\in\mathbb{N}_{0}^{h-1}$ with $\vert \alpha\vert=d$. Thus, \begin{equation*}
f_{0}\varphi^{s}_{0}=\Bigg(\sum_{\vert \alpha\vert =d}c_{\alpha}w^{\alpha}+\sum_{i=1}^{\infty}\sum_{\vert \alpha\vert =d+i}c_{\alpha}w^{\alpha}\Bigg)\varphi^{s}_{0}\in \phi^{-1}(W)
\end{equation*} where $\sum_{\vert \alpha\vert =d}c_{\alpha}w^{\alpha}\neq 0$.\\
\indent  As in the case of $s<0$, we define a sequence of elements in $\phi^{-1}(W)$ inductively for $n>0$ as follows: \begin{equation*}
f_{n}\varphi^{s}_{0}:=\frac{1}{n}\Big((d+n-s)f_{n-1}\varphi^{s}_{0}+\mathfrak{x}_{00}(f_{n-1}\varphi^{s}_{0})\Big).\end{equation*} Using exactly the same proof in the previous case of $s<0$, it can be shown that $f_{n}\varphi^{s}_{0}$ converges to $\Big(\sum_{\vert \alpha\vert =d}c_{\alpha}w^{\alpha}\Big)\varphi^{s}_{0}$ as $n\to\infty$ and $\Big(\sum_{\vert \alpha\vert =d}c_{\alpha}w^{\alpha}\Big)\varphi^{s}_{0}\in \phi^{-1}(W)$ because $\phi^{-1}(W)$ is closed. Since $\sum_{\vert \alpha\vert =d}c_{\alpha}w^{\alpha}$ is a non-zero homogeneous polynomial of total degree $d>s$, it follows from Lemma \ref{polysubspace_contained_in_homopolygensubspace} that $\breve{K}[w_{1},\dots ,w_{h-1}].\varphi^{s}_{0}\subseteq \phi^{-1}(W)$. Since $\breve{K}[w_{1},\dots ,w_{h-1}]\varphi^{s}_{0}$ is dense in $M^{s}_{D}=\mathcal{O}_{X^{\textnormal{rig}}_{0}}(D)\varphi_{0}^{s}$ and $\phi^{-1}(W)$ is closed, it follows that $\phi^{-1}(W)=M^{s}_{D}$. Hence $W=\phi(M^{s}_{D})=M^{s}_{D}/V_{s}$. Thus $M^{s}_{D}/V_{s}$ is topologically irreducible.
\end{proof}   
\begin{corollary}\label{lf0} For all $s\in\mathbb{Z}$, we have \begin{equation*}
(M^{s}_{D})_{\textnormal{lf}}=V_{s}.
\end{equation*} Thus, $(M^{s}_{D})_{\textnormal{lf}}$ is zero if $s<0$ and is a finite dimensional irreducible representation of $\Gamma$ if $s\geq 0$.
\end{corollary}
\begin{proof} Since $\Gamma$ is compact, any $v\in (M^{s}_{D})_{\textnormal{lf}}$ is contained in a finite dimensional $\Gamma$-subrepresentation of $M^{s}_{D}$. Now the corollary immediately follows from Theorem \ref{top_finiteness_of_MsD_thm} and the fact that $M^{s}_{D}$ is not a finite dimensional $\breve{K}$-vector space.
\end{proof}
\begin{remark}\label{highest_wt_of_Vs} For $s\geq 0$, the finite-dimensional $\Gamma$-representation $V_{s}$ is also a $\mathfrak{gl}(K_{h})$-module. Let $\mathfrak{t}\subset\mathfrak{sl}_{h}(K_{h})$ be the Cartan subalgebra of $\mathfrak{sl}_{h}(K_{h})$ consisting of diagonal matrices, and let $\lbrace\varepsilon_{1},\ldots,\varepsilon_{h-1}\rbrace$ be the basis of the root system $(\mathfrak{sl}_{h}(K_{h}),\mathfrak{t})$ given by $\varepsilon_{i}(\text{diag}(t_{0},\ldots,t_{h-1})):=t_{i-1}-t_{i}$. Define the fundamental dominant weight $\chi_{0}:=\frac{1}{h}\sum_{i=1}^{h-1}(h-i)\varepsilon_{i}\in\mathfrak{t}^{*}$. Then, by the same proof as in \cite{kohliwamo}, Proposition 4.3, it follows that $V_{s}$ is an irreducible $\mathfrak{sl}_{h}(K_{h})$-representation of highest weight $s\chi_{0}$. Although this is stronger than saying that $V_{s}$ is an irreducible $\Gamma$-representation, our result (Theorem \ref{top_finiteness_of_MsD_thm}) also gives information about the $\Gamma$-representation $M^{s}_{D}$ when $s<0$, and about the quotient $M^{s}_{D}/V_{s}$ when $s\geq 0$.
\end{remark}
\indent  The Corollary \ref{lf0} leads us to calculate locally finite vectors in the global sections $M^{s}_{0}$ over $X^{\text{rig}}_{0}$. Recall from \cite{bgr}, (9.3.4), Example 3, that the rigid analytic projective space $\mathbb{P}^{h-1}_{\breve{K}}$ has a finite admissible covering by the $(h-1)$-dimensional closed unit polydiscs $V_{i}:=\text{Sp}(\breve{K}\langle\frac{\varphi_{0}}{\varphi_{i}},\ldots,\frac{\varphi_{h-1}}{\varphi_{i}}\rangle)$, $0\leq i\leq h-1$. If $V_{ij}:=\text{Sp}(\breve{K}\langle\frac{\varphi_{0}}{\varphi_{i}},\ldots,\frac{\varphi_{h-1}}{\varphi_{i}},(\frac{\varphi_{j}}{\varphi_{i}})^{-1}\rangle)$ for $0\leq i,j\leq h-1$, then gluing the $V_{i}$'s along the identification $V_{ij}\cong V_{ji}$ of affinoid subdomains via 
$\breve{K}\langle\frac{\varphi_{0}}{\varphi_{i}},\ldots,\frac{\varphi_{h-1}}{\varphi_{i}},(\frac{\varphi_{j}}{\varphi_{i}})^{-1}\rangle=\breve{K}\langle\frac{\varphi_{0}}{\varphi_{j}},\ldots,\frac{\varphi_{h-1}}{\varphi_{j}},(\frac{\varphi_{i}}{\varphi_{j}})^{-1}\rangle$ gives the rigid analytic projective space $\mathbb{P}^{h-1}_{\breve{K}}$. The affinoid covering $\lbrace V_{i}\rbrace_{0\leq i\leq h-1}$ allows us to describe the construction of the line bundles $\mathcal{O}_{\mathbb{P}^{h-1}_{\breve{K}}}(s)$ on the rigid analytic projective space in a way analogous to the classical construction. For $s\geq 0$, define its sections over the affinoid space $V_{i}$; $\mathcal{O}_{\mathbb{P}^{h-1}_{\breve{K}}}(s)(V_{i}):=\breve{K}\Big\langle\frac{\varphi_{0}}{\varphi_{i}},\ldots,\frac{\varphi_{h-1}}{\varphi_{i}}\Big\rangle\varphi_{i}^{s}$; to be a free module of rank 1 generated by $\varphi_{i}^{s}$ over $\mathcal{O}_{\mathbb{P}^{h-1}_{\breve{K}}}(V_{i})=\breve{K}\langle\frac{\varphi_{0}}{\varphi_{i}},\ldots,\frac{\varphi_{h-1}}{\varphi_{i}}\rangle$, and the transition functions $\psi_{ij}:V_{ij}\iso V_{ji}$ induced by the homomorphisms of affinoid $\breve{K}$-algebras \begin{equation*}
\breve{K}\Big\langle\frac{\varphi_{0}}{\varphi_{j}},\ldots,\frac{\varphi_{h-1}}{\varphi_{j}},\Big(\frac{\varphi_{i}}{\varphi_{j}}\Big)^{-1}\Big\rangle\varphi_{j}^{s}\xrightarrow{\textnormal{multiply by}\hspace{.1cm}\frac{\varphi_{i}^{s}}{\varphi_{j}^{s}}}\breve{K}\Big\langle\frac{\varphi_{0}}{\varphi_{i}},\ldots,\frac{\varphi_{h-1}}{\varphi_{i}},\Big(\frac{\varphi_{j}}{\varphi_{i}}\Big)^{-1}\Big\rangle\varphi_{i}^{s}
\end{equation*} for all $0\leq i,j\leq h-1$. The above datum gives rise to a locally free $\mathcal{O}_{\mathbb{P}^{h-1}_{\breve{k}}}$-module $\mathcal{O}_{\mathbb{P}^{h-1}_{\breve{K}}}(s)$ of rank 1. For $s<0$, $\mathcal{O}_{\mathbb{P}_{\breve{K}}^{h-1}}(s)$
turns out to be the $\mathcal{O}_{\mathbb{P}^{h-1}_{\breve{k}}}$-linear dual of $\mathcal{O}_{\mathbb{P}_{\breve{K}}^{h-1}}(-s)$. It then follows easily from the above description that the global sections of $\mathcal{O}_{\mathbb{P}_{\breve{K}}^{h-1}}(s)$ are the $\breve{K}$-vector space of homogeneous polynomials of degree $s$ in the variables $\varphi_{i}$'s if $s\geq 0$, and are 0 otherwise. The line bundles $\mathcal{O}_{\mathbb{P}_{\breve{K}}^{h-1}}(s)$ carry a canonical action of $\Gamma$ induced by its action on the projective space $\mathbb{P}_{\breve{K}}^{h-1}$. \\
\indent  Now for any $\mathcal{O}_{X^{\textnormal{rig}}_{0}}$-module $\mathcal{F}$ and $\mathcal{O}_{\mathbb{P}^{h-1}_{\breve{K}}}-$module $\mathcal{G}$, there is a canonical bijection $\textnormal{Hom}_{\mathcal{O}_{X^{\textnormal{rig}}_{0}}-\textnormal{mod}}(\Phi^{*}\mathcal{G},\mathcal{F})\iso\textnormal{Hom}_{\mathcal{O}_{\mathbb{P}^{h-1}_{\breve{K}}}-\textnormal{mod}}(\mathcal{G},\Phi_{*}\mathcal{F})$, where $\Phi:X^{\textnormal{rig}}_{0}\longrightarrow\mathbb{P}^{h-1}_{\breve{K}}$ is the Gross-Hopkins' period morphism. The morphism id$_{\Phi^{*}\mathcal{G}}$ corresponds to the adjunction morphism $\textnormal{ad}:\mathcal{G}\longrightarrow\Phi_{*}\Phi^{*}\mathcal{G}$. Let $\mathcal{G}=\mathcal{O}_{\mathbb{P}^{h-1}_{\breve{K}}}(s)$ with $s\in\mathbb{Z}$. The period morphism $\Phi$ is constructed in such a way that $\Phi^{*}\mathcal{O}_{\mathbb{P}^{h-1}_{\breve{K}}}(s)\cong(\mathcal{M}^{s}_{0})^{\textnormal{rig}}$ (cf. Remark \ref{our_line_bundle_is_a_pullback_of_O(s)*}). This gives us a map $\textnormal{ad}:\mathcal{O}_{\mathbb{P}^{h-1}_{\breve{K}}}(s)\longrightarrow \Phi_{*}(\mathcal{M}^{s}_{0})^{\textnormal{rig}}$ of $\mathcal{O}_{\mathbb{P}^{h-1}_{\breve{K}}}$-modules. Taking global sections, we get a homomorphism of $\Gamma$-representations \begin{align*}
\textnormal{ad}_{\mathbb{P}^{h-1}_{\breve{K}}}:\mathcal{O}_{\mathbb{P}^{h-1}_{\breve{K}}}(s)(\mathbb{P}^{h-1}_{\breve{K}})\longrightarrow \Phi_{*}(\mathcal{M}^{s}_{0})^{\textnormal{rig}}(\mathbb{P}^{h-1}_{\breve{K}})&=(\mathcal{M}^{s}_{0})^{\textnormal{rig}}(\Phi^{-1}(\mathbb{P}^{h-1}_{\breve{K}}))\\&=(\mathcal{M}^{s}_{0})^{\textnormal{rig}}(X^{\textnormal{rig}}_{0})\\&=M^{s}_{0}.
\end{align*} 
\begin{lemma} The map $
\textnormal{ad}_{\mathbb{P}^{h-1}_{\breve{K}}}$ is injective.
\end{lemma}
\begin{proof}
The period morphism $\Phi$, when restricted to the affinoid subdomain $D$, is an isomorphism. Thus $(\mathcal{M}^{s}_{0})^{\textnormal{rig}}(D)\cong\Phi^{*}\mathcal{O}_{\mathbb{P}^{h-1}_{\breve{K}}}(s)(D)\cong\mathcal{O}_{\mathbb{P}^{h-1}_{\breve{K}}}(s)(\Phi(D))$. Also we have $(\mathcal{M}^{s}_{0})^{\textnormal{rig}}(D)\cong\mathcal{O}_{X^{\textnormal{rig}}_{0}}(D)\varphi^{s}_{0}\cong\mathcal{O}_{\mathbb{P}^{h-1}_{\breve{K}}}(\Phi(D))\varphi^{s}_{0}$. As a result, it follows from the preceding discussion on the line bundles that $\mathcal{O}_{\mathbb{P}^{h-1}_{\breve{K}}}(s)(\mathbb{P}^{h-1}_{\breve{K}})$ maps bijectively onto $V_{s}\subset\mathcal{O}_{\mathbb{P}^{h-1}_{\breve{K}}}(s)(\Phi(D))$ under the restriction map. The lemma now follows from the following commutative diagram with vertical restriction maps.
$$
\xymatrixcolsep{9pc}
\xymatrixrowsep{3pc}
\xymatrix{
\mathcal{O}_{\mathbb{P}^{h-1}_{\breve{K}}}(s)(\mathbb{P}^{h-1}_{\breve{K}}) \ar[r]^{\textnormal{ad}_{\mathbb{P}^{h-1}_{\breve{K}}}} \ar@{^{(}->}[d]
 & \Phi_{*}(\mathcal{M}^{s}_{0})^{\textnormal{rig}}(\mathbb{P}^{h-1}_{\breve{K}})=M^{s}_{0} \ar[d]\\
\mathcal{O}_{\mathbb{P}^{h-1}_{\breve{K}}}(s)(\Phi(D)) \ar[r]_{\textnormal{ad}_{\Phi(D)}}^{\cong}  & \Phi_{*}(\mathcal{M}^{s}_{0})^{\textnormal{rig}}(\Phi(D))=M^{s}_{D}
}
$$
\end{proof}
\begin{corollary}\label{lf1}
For all $s\in\mathbb{Z}$, we have an isomorphism of $\Gamma$-representations \begin{equation*}
(M^{s}_{0})_{\textnormal{lf}}=\mathcal{O}_{\mathbb{P}^{h-1}_{\breve{K}}}(s)(\mathbb{P}^{h-1}_{\breve{K}})\cong V_{s}.
\end{equation*} Thus, $(M^{s}_{0})_{\textnormal{lf}}$ is zero if $s<0$ and is a finite dimensional irreducible representation of $\Gamma$ if $s\geq 0$.
\end{corollary}
\begin{proof}
The inclusion $R^{\text{rig}}_{0}\hookrightarrow\mathcal{O}_{X^{\textnormal{rig}}_{0}}(D)$ gives rise to a $\Gamma$-equivariant inclusion $M^{s}_{0}\hookrightarrow M^{s}_{D}\cong\mathcal{O}_{X^{\textnormal{rig}}_{0}}(D)\otimes_{R_{0}}\text{Lie}(\mathbb{H}^{(0)})^{\otimes s}$ using (\ref{G_eqv_iso_eqn1}) and the freeness of $\text{Lie}(\mathbb{H}^{(0)})^{\otimes s}$ as an $R_{0}$-module. As $\mathcal{O}_{\mathbb{P}^{h-1}_{\breve{K}}}(s)(\mathbb{P}^{h-1}_{\breve{K}})$ is a finite dimensional $\breve{K}$-vector space, we have for $s\geq 0$, $\mathcal{O}_{\mathbb{P}^{h-1}_{\breve{K}}}(s)(\mathbb{P}^{h-1}_{\breve{K}})\subseteq (M^{s}_{0})_{\textnormal{lf}}\subseteq (M^{s}_{D})_{\textnormal{lf}}=V_{s}$, where the first and the last $\breve{K}$-vector spaces are isomorphic as mentioned in the proof of the previous lemma. 
\end{proof} 
\begin{remark}\label{Vs_is_sym-s-part_of_V1} From now on, we identify the subrepresentation $\mathcal{O}_{\mathbb{P}^{h-1}_{\breve{K}}}(s)(\mathbb{P}^{h-1}_{\breve{K}})$ of $M^{s}_{0}$ with $V_{s}$. For $s=1$, the $\Gamma$-locally finite subrepresentation $V_{1}$ of $M^{1}_{0}$ is the representation $\mathbb{V}$ mentioned in the construction of the period morphism $\Phi$ (cf. the paragraph after Corollary \ref{Lie(Ems)_generically_flat}), and thus is isomorphic to the $h$-dimensional $\Gamma$-representation $B_{h}\otimes_{K_{h}}\breve{K}$. Since $\mathcal{O}_{\mathbb{P}^{h-1}_{\breve{K}}}(s)(\mathbb{P}^{h-1}_{\breve{K}})$ is same as the $s$-th symmetric power $\text{Sym}^{s}(\mathcal{O}_{\mathbb{P}^{h-1}_{\breve{K}}}(1)(\mathbb{P}^{h-1}_{\breve{K}}))$ of $\mathcal{O}_{\mathbb{P}^{h-1}_{\breve{K}}}(1)(\mathbb{P}^{h-1}_{\breve{K}})$, we obtain the isomorphism, \begin{equation*}
(M^{s}_{0})_{\textnormal{lf}}= V_{s}\cong \textnormal{Sym}^{s}(B_{h}\otimes_{K_{h}}\breve{K})
\end{equation*} of $\Gamma$-representations for all $s\geq 0$. 
\end{remark}
\subsection{Locally finite vectors in the $\Gamma$-representations $M^{s}_{m}$ with $m>0$}
We compute the locally finite vectors in two parts: $s\leq0$ and $s>0$. The idea here is to use the commuting actions of $\Gamma$ and the finite group $G_{0}/G_{m}$ on $M^{s}_{m}$.
\begin{center}
\textbf{Part I : $s\leq0$}
\end{center}
\begin{lemma}\label{finiteextlemma} Let $G$ be a finite group acting on an integral domain $R$ by ring automorphisms such that the subring of $G$-invariants $R^{G}$ is a perfect field $F$. Then $R$ is a field and the extension $R/F$ is finite. 
\end{lemma}
\begin{proof} 
If $\alpha\in R$ then $\prod_{\sigma\in G}(t-\sigma(\alpha))$ is a monic polynomial of degree $|G|$ with coefficients in $R^{G}=F$, and has $\alpha$ as a root. This implies that every nonzero $\alpha$ has a unique inverse, since $R$ is an integral domain. The second assertion now follows from \cite{langalgebra}, Chapter VI, Lemma 1.7.   
\end{proof}

\indent  Let $\breve{K}_{m}$ denote the $m$-th Lubin-Tate extension of $\breve{K}$. This is a finite Galois extension of $\breve{K}$ obtained by adjoining $\varpi^{m}$-torsion points of any Lubin-Tate formal $\mathfrak{o}$-module over $\mathfrak{o}$ to it. It is a non-trivial result of M. Strauch (cf. \cite{strgeo}, Corollary 3.4 (ii)) that $\breve{K}_{m}\subset R^{\textnormal{rig}}_{m}$. In fact, $\breve{K}_{m}$ is stable under the actions of $G_{0}/G_{m}$ and $\Gamma$ on $ R^{\textnormal{rig}}_{m}$. For $g\in G_{0}/G_{m}$, $\gamma\in\Gamma$ and $\alpha\in\breve{K}_{m}$, these actions are given by $g(\alpha)=\textnormal{det}(g)^{-1}(\alpha)$ and $\gamma(\alpha)=\textnormal{Nrd}(\gamma)(\alpha)$ viewing $\breve{K}$ as a left $\mathfrak{o}^{\times}$-module via the map $\mathfrak{o}^{\times}\twoheadrightarrow(\mathfrak{o}/\varpi^{m}\mathfrak{o})^{\times}\cong\text{Gal}(\breve{K}_{m}/\breve{K})$ (cf. \cite{strgeo}, Theorem 4.4).
\begin{theorem}\label{lfinRm} For all $m\geq0$, $(M^{0}_{m})_{\textnormal{lf}}=(R^{\textnormal{rig}}_{m})_{\textnormal{lf}}=\breve{K}_{m}$.  
\end{theorem}
\begin{proof}
The kernel of the composition map $\Gamma\xrightarrow{\textnormal{Nrd}}\mathfrak{o}^{\times}\longrightarrow(\mathfrak{o}/\varpi^{m}\mathfrak{o})^{\times}$ is an open subgroup of $\Gamma$ which acts trivially on $\breve{K}_{m}$. Thus $\breve{K}_{m}\subseteq (R^{\textnormal{rig}}_{m})_{\textnormal{lf}}$. Notice that $(R^{\textnormal{rig}}_{m})_{\textnormal{lf}}$ is a subring of $R^{\textnormal{rig}}_{m}$ and is stable under the action of $G_{0}/G_{m}$. To see the stability, let $f\in (R^{\textnormal{rig}}_{m})_{\textnormal{lf}}$ and $V$ be a finite dimensional $H$-subrepresentation of $R^{\textnormal{rig}}_{m}$ containing $f$ for some open subgroup $H$ of $\Gamma$. Let $g\in G_{0}/G_{m}$. Then the $\breve{K}$-vector space $gV$ is $H$-stable since the actions of $H$ and $G_{0}/G_{m}$ on $V$ commute. Thus $gV$ is a finite dimensional $H$-subrepresentation of $R^{\textnormal{rig}}_{m}$ containing $gf$ implying that $gf$ is locally finite. Now it follows from \cite{Koh11} Theorem 1.4 (i) and Corollary \ref{lf1} that $(R^{\textnormal{rig}}_{m})_{\textnormal{lf}}^{G_{0}/G_{m}}=((R^{\textnormal{rig}}_{m})^{G_{0}/G_{m}})_{\textnormal{lf}}=(R^{\textnormal{rig}}_{0})_{\textnormal{lf}}=\breve{K}$. As $(R^{\textnormal{rig}}_{m})_{\textnormal{lf}}$ is an integral domain due to \cite{Koh11}, Theoem 1.2 (i) and $G_{0}/G_{m}$ is finite, $(R^{\textnormal{rig}}_{m})_{\textnormal{lf}}$ is a finite field extension of $\breve{K}$ by Lemma \ref{finiteextlemma}. So it is also finite over $\breve{K}_{m}$. However, Strauch's result that $X^{\text{rig}}_{m}$ is geometrically connected over $\breve{K}_{m}$ implies that $\breve{K}_{m}$ is separably closed in $R^{\textnormal{rig}}_{m}$ (cf. \cite{Koh11}, Theorem 1.4). Therefore $(R^{\textnormal{rig}}_{m})_{\textnormal{lf}}=\breve{K}_{m}$.	
\end{proof}
\begin{remark}\label{ginvinRm} By Theorem \ref{laKm}, we have a $\mathfrak{g}$-action on $R^{\textnormal{rig}}_{m}$. The subspace of $\mathfrak{g}$-invariants $(R^{\textnormal{rig}}_{m})^{\mathfrak{g}=0}$ forms a subring of $R^{\textnormal{rig}}_{m}$, and is stable under the action of $G_{0}/G_{m}$ because the $G_{0}/G_{m}$-action on $R^{\textnormal{rig}}_{m}$ is continuous and commutes with that of $\Gamma$. As said in the proof of Theorem \ref{lfinRm}, the kernel of the composition map $\Gamma\xrightarrow{\textnormal{Nrd}}\mathfrak{o}^{\times}\longrightarrow(\mathfrak{o}/\varpi^{m}\mathfrak{o})^{\times}$ is an open subgroup of $\Gamma$ which acts trivially on $\breve{K}_{m}$. Thus, $\breve{K}_{m}\subseteq (R^{\textnormal{rig}}_{m})^{\mathfrak{g}=0}$. Proceeding similarly as above, we have $((R^{\textnormal{rig}}_{m})^{\mathfrak{g}=0})^{G_{0}/G_{m}}=((R^{\textnormal{rig}}_{m})^{G_{0}/G_{m}})^{\mathfrak{g}=0}=(R^{\textnormal{rig}}_{0})^{\mathfrak{g}=0}=\breve{K}$ (cf. Lemma \ref{gninvO(D)}). Then by the same arguments as above, we get $(R^{\textnormal{rig}}_{m})^{\mathfrak{g}=0}=\breve{K}_{m}$.
\end{remark}
\indent  For all integers $s$, the $\Gamma$-equivariant isomorphism $M^{s}_{m}\cong R^{\text{rig}}_{m}\otimes_{R^{\text{rig}}_{0}}M^{s}_{0}$ (cf. proof of Theorem \ref{genlaKm}) and the freeness of the $R^{\text{rig}}_{0}$-module $M^{s}_{0}$ give rise to a $\Gamma$-equivariant inclusion $M^{s}_{0}\subset M^{s}_{m}$ of $\breve{K}$-vector spaces. Consequently, we have $(M^{s}_{0})_{\textnormal{lf}}\subseteq (M^{s}_{m})_{\textnormal{lf}}$. Using the above theorem, we see that $(M^{s}_{m})_{\textnormal{lf}}$ is a module over $(R^{\textnormal{rig}}_{m})_{\textnormal{lf}}=\breve{K}_{m}$, and thus we obtain a natural map \begin{equation*}
\breve{K}_{m}\otimes_{\breve{K}}(M^{s}_{0})_{\textnormal{lf}}\longrightarrow (M^{s}_{m})_{\textnormal{lf}}
\end{equation*} of $\breve{K}$-vector spaces. Our objective is to show that this map is an isomorphism of $\breve{K}[\Gamma]$-modules for all $s$. 
\begin{lemma}\label{lflemma} Suppose $V$ and $W$ are two representations of a topological group $G$ over a field $F$ such that one of them, say $W$, is finite dimensional. Consider the representation $V\otimes_{F}W$ with diagonal $G$-action. Then $(V\otimes_{F}W)_{\textnormal{lf}}=V_{\textnormal{lf}}\otimes_{F}W$.
\end{lemma}
\begin{proof}
We omit the subscript $F$ in $\otimes_{F}$, as all the tensor products are over $F$. The inclusion $V_{\textnormal{lf}}\otimes W\subseteq (V\otimes W)_{\textnormal{lf}}$ is clear. Let $W^{*}$ be the $F$-linear dual of $W$ equipped with the contragredient $G$-action, i.e., $(gf)(w)=f(g^{-1}w)$ for all $g\in G, w\in W$ and $f\in W^{*}$. Choose an $F$-basis $\lbrace w_{1},\ldots ,w_{d}\rbrace$ of $W$, and let $\lbrace f_{1},\ldots ,f_{d}\rbrace$ be the dual basis of $W^{*}$ (i.e. $f_{i}(w_{j})=\delta_{ij}$). Then the natural evaluation map $W\otimes W^{*}\longrightarrow F$ $(w\otimes f\mapsto f(w))$ is $G$-equivariant for the diagonal $G$-action on the left and for the trivial $G$-action on the right. Tensoring both sides with $V$, we get a $G$-equivariant map $\phi:V\otimes W\otimes W^{*}\longrightarrow V$ for the diagonal $G$-action on the left, sending $v\otimes w\otimes f$ to $f(w)v$. Because of its $G$-equivariance, $\phi$ maps locally finite vectors to locally finite vectors. Now let $x\in (V\otimes W)_{\textnormal{lf}}$. Then $x$ can be uniquely written as $x=\sum_{i=1}^{d}x_{i}\otimes w_{i}$ for some $x_{1},\ldots ,x_{d}\in V$. Since $W^{*}$ is finite dimensional, $x\otimes f_{i}\in (V\otimes W)_{\textnormal{lf}}\otimes (W^{*})_{\textnormal{lf}}\subseteq (V\otimes W\otimes W^{*})_{\textnormal{lf}}$ for all $1\leq i\leq d$. Hence, $\phi(x\otimes f_{i})=x_{i}\in V_{\textnormal{lf}}$ for all $1\leq i\leq d$. Therefore, $x\in V_{\textnormal{lf}}\otimes W$.         	
\end{proof}
\begin{theorem}\label{lf2} For all $s<0$ and for all $m\geq 0$, $(M^{s}_{m})_{\textnormal{lf}}\cong\breve{K}_{m}\otimes_{\breve{K}}(M^{s}_{0})_{\textnormal{lf}}=0$.
\end{theorem}
\begin{proof} Recall from Corollary \ref{Lie(Ems)_generically_flat} that we have an isomorphism \begin{equation*}
R^{\textnormal{rig}}_{m}\otimes_{R_{m}}\textnormal{Lie}(\mathbb{E}^{(m)})^{\otimes s}\cong R^{\textnormal{rig}}_{m}\otimes_{\breve{K}}(B_{h}\otimes_{K_{h}}\breve{K})^{\otimes s}
\end{equation*} of $\Gamma$-representations. As a result, using Lemma \ref{lflemma} together with Theorem \ref{lfinRm}, we obtain locally finite vectors in the global sections of $\mathcal{L}\text{ie}(\mathbb{E}^{(m)})^{\otimes s}$, \begin{equation*}
(R^{\textnormal{rig}}_{m}\otimes_{R_{m}}\textnormal{Lie}(\mathbb{E}^{(m)})^{\otimes s})_{\textnormal{lf}}\cong \breve{K}_{m}\otimes_{\breve{K}}(B_{h}\otimes_{K_{h}}\breve{K})^{\otimes s}.
\end{equation*}
Then, since $s<0$, the $(\Gamma\times(G_{0}/G_{m}))$-equivariant inclusion \begin{equation*}
M^{s}_{m}\subset R^{\textnormal{rig}}_{m}\otimes_{R_{m}}\textnormal{Lie}(\mathbb{E}^{(m)})^{\otimes s}
\end{equation*} from (\ref{exact_seq_of_global_sections_of_eqv_bundles}) gives rise to a $(\Gamma\times (G_{0}/G_{m}))$-equivariant inclusion \begin{equation*}
(M^{s}_{m})_{\textnormal{lf}}\subseteq\breve{K}_{m}\otimes_{\breve{K}}(B_{h}\otimes_{K_{h}}\breve{K})^{\otimes s}
\end{equation*}  of $\breve{K}$-vector spaces. As the action of $SL_{h}(\mathfrak{o}/\varpi^{m}\mathfrak{o})\subset G_{0}/G_{m}$ on the right hand side above is trivial, we get $(M^{s}_{m})_{\textnormal{lf}}=\big(M^{s}_{m}\big)_{\textnormal{lf}}^{SL_{h}(\mathfrak{o}/\varpi^{m}\mathfrak{o})}=\Big(\big(M^{s}_{m}\big)^{SL_{h}(\mathfrak{o}/\varpi^{m}\mathfrak{o})}\Big)_{\textnormal{lf}}$, where the latter equality is due to the fact that the both group actions on $M^{s}_{m}$ commute. Therefore,
\begin{align*}(M^{s}_{m})_{\textnormal{lf}}=\Big(\big(M^{s}_{m}\big)^{SL_{h}(\mathfrak{o}/\varpi^{m}\mathfrak{o})}\Big)_{\textnormal{lf}}&\cong\Big(\big(R^{\textnormal{rig}}_{m}\otimes_{R^{\textnormal{rig}}_{0}}M^{s}_{0}\big)^{SL_{h}(\mathfrak{o}/\varpi^{m}\mathfrak{o})}\Big)_{\textnormal{lf}}\\&\cong\Big(\big(R^{\textnormal{rig}}_{m}\big)^{SL_{h}(\mathfrak{o}/\varpi^{m}\mathfrak{o})}\otimes_{R^{\textnormal{rig}}_{0}}M^{s}_{0}\Big)_{\textnormal{lf}}\\&\cong\big((\breve{K}_{m}\otimes_{\breve{K}}R^{\textnormal{rig}}_{0})\otimes_{R^{\textnormal{rig}}_{0}}M^{s}_{0}\big)_{\textnormal{lf}}\\&\cong\big(\breve{K}_{m}\otimes_{\breve{K}}M^{s}_{0}\big)_{\textnormal{lf}}=\breve{K}_{m}\otimes_{\breve{K}}(M^{s}_{0})_{\textnormal{lf}}=0
\end{align*}
where the second isomorphism holds because $M^{s}_{0}$ is free over $R^{\text{rig}}_{0}$ with trivial $G_{0}/G_{m}$-action, and the third isomorphism holds because $\big(R^{\textnormal{rig}}_{m}\big)^{SL_{h}(\mathfrak{o}/\varpi^{m}\mathfrak{o})}$ is Galois over $R^{\textnormal{rig}}_{0}$ with the Galois group isomorphic to $\frac{G_{0}/G_{m}}{SL_{h}(\mathfrak{o}/\varpi^{m}\mathfrak{o})}\cong(\mathfrak{o}/\varpi^{m}\mathfrak{o})^{\times}\cong\textnormal{Gal}(\breve{K}_{m}/\breve{K})$ and $\breve{K}_{m}\subseteq R^{\text{rig}}_{m}$. For the second last equality in the above, we use Lemma \ref{lflemma} again. The final result then follows from Corollary \ref{lf1}.
\end{proof}
\begin{center}
\textbf{Part II : $s>0$}
\end{center}
\indent  To compute the locally finite vectors in $M^{s}_{m}$ for $s>0$, we make use of the action of the group $G^{0}:=\lbrace g\in GL_{h}(K)|\hspace*{.1cm}\textnormal{det}(g)\in\mathfrak{o}^{\times}\rbrace$ on \textit{the Lubin-Tate tower} $\big(X_{m}^{\textnormal{rig}}\big)_{m\in\mathbb{N}_{0}}$ described in \cite{strdeform}, Section 2.2.2. Given $g\in G^{0}$ and $m\geq 0$, for every $m'\geq m$ sufficiently large (depending on $g$), there is a morphism $g_{m',m}:X_{m'}^{\textnormal{rig}}\longrightarrow X_{m}^{\textnormal{rig}}$ of rigid analytic spaces satisfying the following properties:
\begin{enumerate}[leftmargin=*]
\item[1.] For all $g\in G^{0}$ and for all $n\geq m''\geq m'\geq m$, we have $g_{n,m}=\pi_{m',m}\circ g_{m'',m'} \circ \pi_{n,m''}$, where recall that $\pi_{m',m}:X_{m'}^{\textnormal{rig}}\longrightarrow X_{m}^{\textnormal{rig}}$ denotes the covering morphism. In particular, if $g=e$, and if $m=m'=m''$, then we get $e_{n,m}=\pi_{n,m}$ for all $n\geq m$ because $e_{m,m}=\text{id}_{ X_{m}^{\textnormal{rig}}}$ by definition (cf. \cite{strdeform}, Section 2.2.2).
\item[2.] $(gh)_{m'',m}=g_{m',m}\circ h_{m'',m}$ for all $g,h\in G^{0}$ and for all $m''\geq m' \geq m$.
\item[3.] Set $\Phi_{m}:=\Phi\circ\pi_{m,0}:X^{\textnormal{rig}}_{m}\longrightarrow\mathbb{P}^{h-1}_{\breve{K}}$. Then $\Phi_{m'}=\Phi_{m}\circ g_{m',m}$ for all $g\in G^{0}$, $m'\geq m$.
\item[4.] All $g_{m',m}$ are $\Gamma$-equivariant morphisms.
\item[5.] For $g\in GL_{h}(\mathfrak{o})$ and $m\geq 0$, $g_{m,m}$ is defined. The gives an action of $GL_{h}(\mathfrak{o})$ on $X_{m}^{\textnormal{rig}}$ which factors through $GL_{h}(\mathfrak{o}/\varpi^{m}\mathfrak{o})=G_{0}/G_{m}$. The induced $G_{0}/G_{m}$-action coincides with the $G_{0}/G_{m}$-action introduced in Section \ref{the_group_actions}.
\end{enumerate}

\indent  Let $D_{m}:=\pi_{m,0}^{-1}(D)$ where $D$ is the Gross-Hopkins fundamental domain $D$ in $X^{\textnormal{rig}}_{0}$. The admissible open $D_{m}$ is a $\Gamma$-stable affinoid subdomain because $\pi_{m,0}$ is a finite, $\Gamma$-equivariant morphism, and $D$ is $\Gamma$-stable. For every $g\in G^{0}$ and $m\geq 0$, we define a $g$-translate of $D_{m}$ \hspace{0cm} as $gD_{m}:=g_{m',m}(D_{m'})$ by choosing $m'\geq m$ large enough. Note that this definition is independent of the choice of $m'$, since by property 1, for $m''\geq m'\geq m$,
$g_{m'',m}(D_{m''})=g_{m',m}(\pi_{m'',m'}(D_{m''}))=g_{m',m}(\pi_{m'',m'}(\pi_{m'',0}^{-1}(D)))=g_{m',m}(\pi_{m'',m'}(\pi_{m'',m'}^{-1}(\pi_{m',0}^{-1}(D))))=g_{m',m}(D_{m'})$, using that $\pi_{m'',m'}$ is surjective. 
\begin{proposition}\label{gDmisacovering} The set $\lbrace gD_{m}\rbrace_{g\in G^{0}}$ forms an admissible affinoid covering of $\Phi_{m}^{-1}(\Phi(D))$ consisting of $\Gamma$-stable affinoid subdomains.
\end{proposition}
\begin{proof}
This is a part of the cellular decomposition of the Lubin-Tate tower in \cite{fgl}, Proposition I.7.1 relying on \cite{gh}, Corollary 23.26. The $\Gamma$-stability of $gD_{m}$ follows from (iv) and that of $D_{m'}$.
\end{proof}
\begin{lemma}\label{inclusionofaffisubd} The maps $\mathcal{O}_{\mathbb{P}^{h-1}_{\breve{K}}}(\Phi(D))\longrightarrow\mathcal{O}_{X^{\textnormal{rig}}_{m}}(gD_{m})\longrightarrow\mathcal{O}_{X^{\textnormal{rig}}_{m'}}(D_{m'})$ of affinoid $\breve{K}$-algebras induced by the morphisms $D_{m'}\xrightarrow{g_{m',m}}gD_{m}\xrightarrow{\Phi_{m}}\Phi(D)$ are injective  for all $g\in G^{0}$ and $m'\geq m$. 
\end{lemma}
\begin{proof}
By property 3, the composition $\Phi_{m}\circ g_{m',m}=\Phi_{m'}=\Phi\circ\pi_{m',0}$ is flat because $\Phi$ and $\pi_{m',0}$ are flat. Hence the composition map $\mathcal{O}_{\mathbb{P}^{h-1}_{\breve{K}}}(\Phi(D))\longrightarrow\mathcal{O}_{X^{\textnormal{rig}}_{m'}}(D_{m'})$ of affinoid $\breve{K}$-algebras is flat. Since $\mathcal{O}_{\mathbb{P}^{h-1}_{\breve{K}}}(\Phi(D))\cong\mathcal{O}_{X^{\textnormal{rig}}_{0}}(D)$ is an integral domain (cf. \cite{bgr}, (6.1.5), Proposition 2), the map $\mathcal{O}_{\mathbb{P}^{h-1}_{\breve{K}}}(\Phi(D))\longrightarrow\mathcal{O}_{X^{\textnormal{rig}}_{m'}}(D_{m'})$ is injective.\\
\indent  To show that the other map $\mathcal{O}_{X^{\textnormal{rig}}_{m}}(gD_{m})\longrightarrow\mathcal{O}_{X^{\textnormal{rig}}_{m'}}(D_{m'})$ is injective, choose $m''\geq m'$ large enough so that $g^{-1}_{m'',m'}:X^{\textnormal{rig}}_{m''}\longrightarrow X^{\textnormal{rig}}_{m'}$ is defined. Using properties 1 and 2, we have $g^{-1}_{m'',m'}(gD_{m''})=g^{-1}_{m'',m'}(g_{n,m''}(D_{n}))=e_{n,m'}(D_{n})=\pi_{n,m'}(D_{n})=D_{m'}$ and thus $gD_{m}=g_{m',m}(D_{m'})=g_{m',m}(g^{-1}_{m'',m'}(gD_{m''}))$ = $e_{m'',m}(gD_{m''})=\pi_{m'',m}(gD_{m''})$. In other words, \begin{equation*}
\big(g_{m',m}\circ g^{-1}_{m'',m'}\big)\big|_{gD_{m''}}=\pi_{m'',m}.\end{equation*} 
Hence the induced composition $\mathcal{O}_{X^{\textnormal{rig}}_{m}}(gD_{m})\longrightarrow\mathcal{O}_{X^{\textnormal{rig}}_{m'}}(D_{m'})\longrightarrow\mathcal{O}_{X^{\textnormal{rig}}_{m''}}(gD_{m''})$ of the maps of affinoid $\breve{K}$-algebras is flat. Now it is not clear if the algebra $\mathcal{O}_{X^{\textnormal{rig}}_{m}}(gD_{m})$ is an integral domain. However, we can decompose $gD_{m}$ into its finitely many disjoint connected components $gD_{m}=\bigsqcup_{i=1}^{r} U_{i}$ so that each $\mathcal{O}_{X^{\textnormal{rig}}_{m}}(U_{i})$ is an integral domain (cf. discussion after \cite{bgr}, (9.1.4), Proposition 8 as well as \cite{con}, Lemma 2.1.4). This decomposition also gives a decomposition $gD_{m''}=\bigsqcup_{i=1}^{r}(\pi_{m'',m}|_{gD_{m''}})^{-1}(U_{i})$ of $gD_{m''}$ into disjoint admissible open subsets. By the same argument as in the first paragraph, each map $\mathcal{O}_{X^{\textnormal{rig}}_{m}}(U_{i})\longrightarrow\mathcal{O}_{X^{\textnormal{rig}}_{m''}}((\pi_{m'',m}|_{gD_{m''}})^{-1}(U_{i}))$ is injective. As a consequence, the composition \linebreak $\mathcal{O}_{X^{\textnormal{rig}}_{m}}(gD_{m})\longrightarrow\mathcal{O}_{X^{\textnormal{rig}}_{m''}}(gD_{m''})$ is injective since it is the finite direct product of all these maps.
\end{proof} 
\begin{remark}\label{O(gDm)isDmod_part1}
The affinoid subdomain $D_{m}$, by definition, is the same as the fibre product $X^{\textnormal{rig}}_{m}\times_{X^{\textnormal{rig}}_{0}}D$ for the maps $\pi_{m,0}:X^{\textnormal{rig}}_{m}\longrightarrow X^{\textnormal{rig}}_{0}$ and $D\hookrightarrow X^{\textnormal{rig}}_{0}$. Thus, we have an isomorphism $\mathcal{O}_{X^{\textnormal{rig}}_{m}}(D_{m})\cong R^{\textnormal{rig}}_{m}\otimes_{R^{\textnormal{rig}}_{0}}\mathcal{O}_{X^{\textnormal{rig}}_{0}}(D) $ because $R^{\textnormal{rig}}_{m}\big|R^{\textnormal{rig}}_{0}$ is finite. The Galois group $G_{0}/G_{m}=\textnormal{Gal}(R^{\textnormal{rig}}_{m}\big|R^{\textnormal{rig}}_{0})$ acts on $\mathcal{O}_{X^{\textnormal{rig}}_{m}}(D_{m})$ via $\sum_{i=1}^{r}f_{i}\otimes f'_{i}\mapsto\sum_{i=1}^{r}g(f_{i})\otimes f'_{i}$ for $g\in G_{0}/G_{m}$, which gives an action on $\mathcal{O}_{X^{\textnormal{rig}}_{m}}(D_{m})$ by $\mathcal{O}_{X^{\textnormal{rig}}_{0}}(D)$-linear automorphisms. Hence the extension $\mathcal{O}_{X^{\textnormal{rig}}_{m}}(D_{m})|\mathcal{O}_{X^{\textnormal{rig}}_{0}}(D)$ is finite Galois with the Galois group $G_{0}/G_{m}$. Consequently, for all $m\geq 0$, the extension of coordinate rings $\mathcal{O}_{X^{\textnormal{rig}}_{m}}(D_{m})|\mathcal{O}_{\mathbb{P}^{h-1}_{\breve{K}}}(\Phi(D))$ induced by the map $\Phi_{m}$ is finite Galois with the same Galois group. 
\end{remark}
\begin{remark}\label{O(gDm)isDmod} As both $R^{\textnormal{rig}}_{m}$ and $\mathcal{O}_{X^{\textnormal{rig}}_{0}}(D)$ are $\mathfrak{g}$-modules (cf. Proposition \ref{laofG}, Theorem \ref{laKm}), we have a $\mathfrak{g}$-action on $\mathcal{O}_{X^{\textnormal{rig}}_{m}}(D_{m})\cong R^{\textnormal{rig}}_{m}\otimes_{R^{\textnormal{rig}}_{0}}\mathcal{O}_{X^{\textnormal{rig}}_{0}}(D)$. Namely, if $\mathfrak{x}\in\mathfrak{g}$ then on simple tensors, $\mathfrak{x}(f\otimes f')=\mathfrak{x}(f)\otimes f'+f\otimes\mathfrak{x}(f')$. The $\mathfrak{g}$-action on $\mathcal{O}_{X^{\textnormal{rig}}_{m'}}(D_{m'})$ restricts to the subalgebra $\mathcal{O}_{X^{\textnormal{rig}}_{m}}(gD_{m})$, because by Remark \ref{O(gDm)isDmod_part1}, $\mathcal{O}_{X^{\textnormal{rig}}_{m}}(gD_{m})$ is a $\Gamma$-stable submodule of the finitely generated  $\mathcal{O}_{\mathbb{P}^{h-1}_{\breve{K}}}(\Phi(D))$-module $\mathcal{O}_{X^{\textnormal{rig}}_{m'}}(D_{m'})$ and hence is closed in $\mathcal{O}_{X^{\textnormal{rig}}_{m'}}(D_{m'})$ by \cite{bgr}, (3.7.3), Proposition 1. Denoting by $\text{Ad}_{\gamma}$ the adjoint automorphism of $\mathfrak{g}$ corresponding to $\gamma\in\Gamma$, we remark that the actions of $\Gamma$ and $\mathfrak{g}$ on $\mathcal{O}_{X^{\textnormal{rig}}_{m}}(gD_{m})$ are compatible in the sense that $\gamma(\mathfrak{x}(f))=\text{Ad}_{\gamma}(\mathfrak{x})(\gamma(f))$, since the Lie algebra action comes from the action of the distribution algebra $D(\Gamma)$ on $R^{\text{rig}}_{m}$ and $\mathcal{O}_{X^{\textnormal{rig}}_{0}}(D)$. Using the isomorphism $(\mathcal{M}^{s}_{m})^{\text{rig}}(gD_{m})\cong\mathcal{O}_{X^{\textnormal{rig}}_{m}}(gD_{m})\otimes_{R^{\text{rig}}_{m}}M^{s}_{m}$ and Theorem \ref{genlaKm}, one obtains that $(\mathcal{M}^{s}_{m})^{\text{rig}}(gD_{m})$ carries compatible actions of $\Gamma$ and $\mathfrak{g}$ for all $s\in\mathbb{Z}$.  
\end{remark}
\begin{proposition}\label{ninvariantsequalginvariants} For all $g\in G^{0}$ and $m\geq 0$, $\mathcal{O}_{X^{\textnormal{rig}}_{m}}(gD_{m})_{\textnormal{lf}}=\mathcal{O}_{X^{\textnormal{rig}}_{m}}(gD_{m})^{\mathfrak{g}=0}=\mathcal{O}_{X^{\textnormal{rig}}_{m}}(gD_{m})^{\mathfrak{n}=0}$. All these $\breve{K}$-vector spaces are finite dimensional. 
\end{proposition}
\begin{proof}
Let $g\in G^{0}, m\geq 0$ be arbitrary, and $m'\geq m$ so that $g_{m',m}$ is defined. As seen in the proof of Lemma \ref{inclusionofaffisubd}, the composition $\mathcal{O}_{\mathbb{P}^{h-1}_{\breve{K}}}(\Phi(D))\hookrightarrow\mathcal{O}_{X^{\textnormal{rig}}_{m}}(gD_{m})\hookrightarrow\mathcal{O}_{X^{\textnormal{rig}}_{m'}}(D_{m'})$ is induced by $\Phi_{m'}$. The $\Gamma$-equivariance of $g_{m',m}$ and of $\Phi_{m}$ yields the inclusions $\mathcal{O}_{\mathbb{P}^{h-1}_{\breve{K}}}(\Phi(D))_{\textnormal{lf}}\hookrightarrow\mathcal{O}_{X^{\textnormal{rig}}_{m}}(gD_{m})_{\textnormal{lf}}\hookrightarrow\mathcal{O}_{X^{\textnormal{rig}}_{m'}}(D_{m'})_{\textnormal{lf}}$ of $\breve{K}$-algebras. The Galois action on $\mathcal{O}_{X^{\textnormal{rig}}_{m'}}(D_{m'})$ commutes with the $\Gamma$-action. As a result, $\mathcal{O}_{X^{\textnormal{rig}}_{m'}}(D_{m'})_{\textnormal{lf}}$ is stable under the Galois action and $\big(\mathcal{O}_{X^{\textnormal{rig}}_{m'}}(D_{m'})_{\textnormal{lf}}\big)^{G_{0}/G_{m'}}\\=\big(\mathcal{O}_{X^{\textnormal{rig}}_{m'}}(D_{m'})^{G_{0}/G_{m'}}\big)_{\textnormal{lf}}=\mathcal{O}_{\mathbb{P}^{h-1}_{\breve{K}}}(\Phi(D))_{\textnormal{lf}}=\mathcal{O}_{X^{\textnormal{rig}}_{0}}(D)_{\textnormal{lf}}=\breve{K}$ (cf. Corollary \ref{lf0}). Since $G_{0}/G_{m'}$ is finite, $\mathcal{O}_{X^{\textnormal{rig}}_{m'}}(D_{m'})_{\textnormal{lf}}$ is integral over $\breve{K}$, and thus $\mathcal{O}_{X^{\textnormal{rig}}_{m}}(gD_{m})_{\textnormal{lf}}$ is integral over $\breve{K}$. As before, we write $gD_{m}=\bigsqcup_{i=1}^{r} U_{i}$ where $U_{i}$s are the connected components of $gD_{m}$. Let $\Gamma_{i}$ be the stabilizer of $U_{i}$ in $\Gamma$; then each $\Gamma_{i}$ has a finite index in $\Gamma$. Let $\Gamma_{o}\subseteq\Gamma$ be an open subgroup which is a uniform pro-$p$ group. Then for every $i$, the intersection $\Gamma_{i}\cap\Gamma_{o}$ has a finite index in $\Gamma_{o}$, and thus is open in $\Gamma_{o}$ by \cite{ddms}, Theorem 1.17. As a result, $\Gamma_{i}\cap\Gamma_{o}$ is open in $\Gamma$, and \begin{equation*}
\Gamma_{i}=\bigcup_{\bar{\gamma}\in\Gamma_{i}/\Gamma_{i}\cap\Gamma_{o}}\gamma(\Gamma_{i}\cap\Gamma_{o})
\end{equation*} implies that $\Gamma_{i}$ is open in $\Gamma$ for all $i$. Hence their intersection $\Gamma':=\bigcap_{i=1}^{r}\Gamma_{i}$ is again an open subgroup of $\Gamma$.\\
\indent  Now the decomposition $\mathcal{O}_{X^{\textnormal{rig}}_{m}}(gD_{m})\cong\prod_{i=1}^{r}\mathcal{O}_{X^{\textnormal{rig}}_{m}}(U_{i})$ of $\breve{K}$-algebras is $\Gamma'$-equivariant for the componentwise $\Gamma'$-action on the right. Thus the compactness of $\Gamma$ gives the decomposition  $\mathcal{O}_{X^{\textnormal{rig}}_{m}}(gD_{m})_{\textnormal{lf}}=\mathcal{O}_{X^{\textnormal{rig}}_{m}}(gD_{m})_{\Gamma'-\textnormal{lf}}\cong\prod_{i=1}^{r}\mathcal{O}_{X^{\textnormal{rig}}_{m}}(U_{i})_{\Gamma'-\textnormal{lf}}$ of locally finite vectors. Denote by $K_{i}$ the integral closure of $\breve{K}$ in the integral domain $\mathcal{O}_{X^{\textnormal{rig}}_{m}}(U_{i})$ for each $i$. It then follows that $K_{i}$ is a field extension of $\breve{K}$. Since every projection $\mathcal{O}_{X^{\textnormal{rig}}_{m}}(gD_{m})_{\textnormal{lf}}\longrightarrow\mathcal{O}_{X^{\textnormal{rig}}_{m}}(U_{i})_{\Gamma'-\textnormal{lf}}$ is a surjective $\breve{K}$-algebra homomorphism, the integrality of $\mathcal{O}_{X^{\textnormal{rig}}_{m}}(gD_{m})_{\textnormal{lf}}$ over $\breve{K}$ implies that $\mathcal{O}_{X^{\textnormal{rig}}_{m}}(U_{i})_{\Gamma'-\textnormal{lf}}$ is integral over $\breve{K}$ for all $i$. Therefore, $\mathcal{O}_{X^{\textnormal{rig}}_{m}}(U_{i})_{\Gamma'-\textnormal{lf}}\subseteq K_{i}$ for all $i$. On the other hand, for each $i$, $K_{i}$ is $\Gamma'$-stable as $\Gamma'$ acts $\breve{K}$-linearly on $\mathcal{O}_{X^{\textnormal{rig}}_{m}}(U_{i})$. Now for any classical point $x\in U_{i}$, the composition map $K_{i}\hookrightarrow\mathcal{O}_{X^{\textnormal{rig}}_{m}}(U_{i})\twoheadrightarrow\kappa(x)$ is injective, and because $\kappa(x)|\breve{K}$ is finite, $K_{i}|\breve{K}$ is a finite extension. This gives the other inclusion $K_{i}\subseteq\mathcal{O}_{X^{\textnormal{rig}}_{m}}(U_{i})_{\Gamma'-\textnormal{lf}} $ for all $i$. Thus we have $\mathcal{O}_{X^{\textnormal{rig}}_{m}}(gD_{m})_{\textnormal{lf}}=\prod_{i=1}^{r}K_{i}$ with each $K_{i}$ a finite field extension of $\breve{K}$.\\
\indent  We now claim that $\mathcal{O}_{X^{\textnormal{rig}}_{m}}(gD_{m})^{\mathfrak{g}=0}=\mathcal{O}_{X^{\textnormal{rig}}_{m}}(gD_{m})^{\mathfrak{n}=0}=\prod_{i=1}^{r}K_{i}$. Note that $\mathcal{O}_{X^{\textnormal{rig}}_{m}}(U_{i})$ is $\mathfrak{g}$-stable for all $i$ because the projection map $\mathcal{O}_{X^{\textnormal{rig}}_{m}}(gD_{m})\twoheadrightarrow\mathcal{O}_{X^{\textnormal{rig}}_{m}}(U_{i})$ of affinoid $\breve{K}$-algebras is surjective, continuous and $\Gamma'$-equivariant. Then all arguments in the last two paragraphs carry over to these cases since $\mathcal{O}_{\mathbb{P}^{h-1}_{\breve{K}}}(\Phi(D))^{\mathfrak{g}=0}=\mathcal{O}_{\mathbb{P}^{h-1}_{\breve{K}}}(\Phi(D))^{\mathfrak{n}=0}=\breve{K}$ (cf. Lemma \ref{gninvO(D)}). The only thing that remains to be shown is $K_{i}\subseteq\mathcal{O}_{X^{\textnormal{rig}}_{m}}(U_{i})^{\mathfrak{g}=0}$ for all $i$ : Write $K_{i}=\breve{K}[\alpha_{i}]$. Then the set $\lbrace\gamma(\alpha_{i})\rbrace_{\gamma\in\Gamma'}$ is finite as $\Gamma'$ takes $\alpha_{i}$ to its conjugates. Therefore, the stabilizer $\Gamma'_{i}$ of $\alpha_{i}$ in $\Gamma'$ has a finite index in $\Gamma'$, and thus we obtain the open subgroup $\Gamma'_{i}$ of $\Gamma$ which acts trivially on $K_{i}$.
\end{proof}
\indent  The embedding $\Gamma\hookrightarrow GL_{h}(K_{h})$ in (\ref{Gamma_as_a_subgroup_of_GLh}) extends to an embedding $B_{h}^{\times}\hookrightarrow GL_{h}(K_{h})$ of locally $K$-analytic groups via the same map. This yields an action of $B_{h}^{\times}$ on $\mathbb{P}_{\breve{K}}^{h-1}$. The $\Gamma$-action on $X^{\textnormal{rig}}_{m}$ is extended to the full group $B_{h}^{\times}$ by letting $b\in B_{h}^{\times}$ act by the action of $(1,b,\tilde{\sigma}^{-\text{val}(\text{Nrd}(b))})\in GL_{h}(K)\times B_{h}^{\times}\times W_{K}$ given on page 20 of \cite{car}. Here $\tilde{\sigma}$ denotes a lift of the Frobenius in the Weil group $W_{K}$ and $\text{val}$ is the normalized valuation of $K$. The maps $\Phi_{m}$ are equivariant for the extended  $B_{h}^{\times}$-action for all $m$.
\begin{lemma}\label{coveringofXm} The set $\lbrace\Pi^{i}\Phi(D)\rbrace_{0\leq i\leq h-1}$ forms an admissible covering of $\mathbb{P}_{\breve{K}}^{h-1}$. Thus, $X^{\textnormal{rig}}_{m}$ has an admissible covering $\lbrace \Pi^{i}\Phi_{m}^{-1}(\Phi(D))\rbrace_{0\leq i\leq h-1}$ for all $m\geq 0$.
\end{lemma} 
\begin{proof} This is proved as a part of \cite{gh}, Corollary 23.21.
\end{proof}
\indent  For $0\leq i\leq h-1$, $s\geq 0$, $m\geq 0$, define $N^{s}_{m}(i):=(\mathcal{M}^{s}_{m})^{\textnormal{rig}}(\Pi^{i}\Phi^{-1}_{m}(\Phi(D)))$ and $A_{m}(i):=N^{0}_{m}(i)=\mathcal{O}_{X^{\textnormal{rig}}_{m}}(\Pi^{i}\Phi^{-1}_{m}(\Phi(D)))$. Note that each $\Pi^{i}\Phi_{m}^{-1}(\Phi(D))$ is $\Gamma$-stable because the conjugation by $\Pi^{-i}$ $(\gamma\mapsto\Pi^{-i}\gamma\Pi^{i})$ is an automorphism of $\Gamma$. Therefore, all $A_{m}(i)$ and $N^{s}_{m}(i)$ are $\Gamma$-representations. Moreover, they are also $\mathfrak{g}$-modules as explained below.\\
\indent  Because of Proposition \ref{gDmisacovering}, we have the exact diagram
\[
\xymatrix{
A_{m}(0)\ar[r]^-{r}&\prod_{g\in G^{0}}\mathcal{O}_{X^{\textnormal{rig}}_{m}}(gD_{m})\ar@<-.6ex>[r]_-{r_{2}}\ar@<.6ex>[r]^-{r_{1}}&\prod_{g,g'\in G^{0}}\mathcal{O}_{X^{\textnormal{rig}}_{m}}(gD_{m}\cap g'D_{m})
}
\]      
with maps given by $r(f)=\big(f\big|_{gD_{m}}\big)_{g\in G^{0}}$, $r_{1}((f_{g})_{g\in G^{0}})=\big(f_{g}\big|_{gD_{m}\cap g'D_{m}}\big)_{g,g'\in G^{0}}$, and $r_{2}((f_{g})_{g\in G^{0}})=\big(f_{g'}\big|_{gD_{m}\cap g'D_{m}}\big)_{g,g'\in G^{0}}$. The continuity of the restriction maps $\mathcal{O}_{X^{\textnormal{rig}}_{m}}(gD_{m})\longrightarrow\mathcal{O}_{X^{\textnormal{rig}}_{m}}(gD_{m}\cap g'D_{m})$ between affinoid $\breve{K}$-algebras implies that the maps $r_{1}$ and $r_{2}$ are continuous for the product topology on their source and target. The Remark \ref{O(gDm)isDmod} allows us to view $\prod_{g\in G^{0}}\mathcal{O}_{X^{\textnormal{rig}}_{m}}(gD_{m})$ as a $\mathfrak{g}$-module with the componentwise $\mathfrak{g}$-action. Now, $A_{m}(0)$ can be identified with the kernel of the continuous map 
\begin{align*}
r_{1}-r_{2}:\prod_{g\in G^{0}}\mathcal{O}&_{X^{\textnormal{rig}}_{m}}(gD_{m})\longrightarrow\prod_{g,g'\in G^{0}}\mathcal{O}_{X^{\textnormal{rig}}_{m}}(gD_{m}\cap g'D_{m})\\&(f_{g})_{g\in G^{0}}\longmapsto r_{1}((f_{g})_{g\in G^{0}})-r_{2}((f_{g})_{g\in G^{0}}).
\end{align*} Hence, $A_{m}(0)$ is a closed $\Gamma$-stable subspace of $\prod_{g\in G^{0}}\mathcal{O}_{X^{\textnormal{rig}}_{m}}(gD_{m})$ as $r$ is $\Gamma$-equivariant. Consequently, $A_{m}(0)$ is stable under the induced $\mathfrak{g}$-action.\\
\indent  Observe that the isomorphism \begin{equation*}
(\mathcal{M}^{s}_{m})^{\text{rig}}(gD_{m}\cap g'D_{m})\cong\mathcal{O}_{X^{\textnormal{rig}}_{m}}(gD_{m}\cap g'D_{m})\otimes_{R^{\text{rig}}_{m}}M^{s}_{m}
\end{equation*} yields a $\mathfrak{g}$-action on $(\mathcal{M}^{s}_{m})^{\text{rig}}(gD_{m}\cap g'D_{m})$ (cf. Theorem \ref{genlaKm} and Remark \ref{O(gDm)isDmod}). The restriction maps $(\mathcal{M}^{s}_{m})^{\text{rig}}(gD_{m})\longrightarrow(\mathcal{M}^{s}_{m})^{\text{rig}}(gD_{m}\cap g'D_{m})$ are continuous for the topology of finitely generated Banach modules. Then by the similar argument as in the last paragraph, $N^{s}_{m}(0)$ carries a $\mathfrak{g}$-module structure. The $\mathfrak{g}$-action and the $\Gamma$-action on $A_{m}(0)$ and on $N^{s}_{m}(0)$ are compatible with each other (cf. Remark \ref{O(gDm)isDmod}). The action of $\Pi^{i}$ on $\Phi^{-1}_{m}(\Phi(D))$ induces an isomorphism $\Pi^{i}:N^{s}_{m}(i)\iso N^{s}_{m}(0)$ of the sections of $(\mathcal{M}^{s}_{m})^{\text{rig}}$. The $\mathfrak{g}$-action on $N^{s}_{m}(i)$ is then given by $\mathfrak{x}(n)=\Pi^{-i}(\text{Ad}_{\Pi^{i}}(\mathfrak{x})(\Pi^{i}(n)))$ for $\mathfrak{x}\in\mathfrak{g}, n\in N^{s}_{m}(i)$ and $1\leq i\leq h-1$.\\
\indent  Now since $M^{s}_{m}$ is generated over $R^{\textnormal{rig}}_{m}$ by $V_{s}$ (cf. (\ref{exact_seq_of_global_sections_of_eqv_bundles}), Proposition \ref{Lie(E0)_is_generically_flat}), $N^{s}_{m}(i)$ is generated by $V_{s}$ as an $A_{m}(i)$-module for all $0\leq i\leq h-1$. Let $A_{m}(i)^{\mathfrak{g}=0}V_{s}$ and $A_{m}(i)^{\mathfrak{g}=0}\varphi_{0}^{s}$ denote the $A_{m}(i)^{\mathfrak{g}=0}$-submodules of $N^{s}_{m}(i)$ generated by $V_{s}$ and $\varphi_{0}^{s}$ respectively.
\begin{proposition}\label{keyprop} For all $0\leq i\leq h-1$, $s\geq 0$, $m\geq 0$, we have $N^{s}_{m}(i)_{\textnormal{lf}}\subseteq A_{m}(i)^{\mathfrak{g}=0}V_{s}$ and $(N^{s}_{m}(i)_{\textnormal{lf}})^{\mathfrak{n}=0}\subseteq A_{m}(i)^{\mathfrak{g}=0}\varphi_{0}^{s}$. \end{proposition}
\begin{proof}
We first show that $N^{s}_{m}(0)_{\textnormal{lf}}\subseteq A_{m}(0)^{\mathfrak{g}=0}V_{s}$. Noticing $\varphi_{0}\in\mathcal{O}_{\mathbb{P}_{\breve{K}}^{h-1}}(\Phi(D))^{\times}$ $\hookrightarrow$ $A_{m}(0)^{\times}$ implies that $\varphi^{s}_{0}$ alone generates $N_{m}^{s}(0)$ as a free $A_{m}(0)$-module of rank one. Now let $W\subseteq N^{s}_{m}(0)$ be a finite dimensional $\Gamma$-stable subspace. As an $\mathfrak{sl}_{h}(K_{h})$-representation, $W$ decomposes as a direct sum of simple $\mathfrak{sl}_{h}(K_{h})$-modules by Weyl's complete reducibility theorem. From highest weight theory, we know that each simple module in the decomposition is generated by an element annihilated by the subalgebra $\mathfrak{n}$ of strictly upper triangular matrices. Now $W^{\mathfrak{n}=0}\subseteq N^{s}_{m}(0)^{\mathfrak{n}=0}=(A_{m}(0)\varphi^{s}_{0})^{\mathfrak{n}=0}=A_{m}(0)^{\mathfrak{n}=0}\varphi^{s}_{0}$ because $\mathfrak{n}\varphi^{s}_{0}=0$ (cf. Lemma \ref{explicit_g-action_lemma}). Let $f\in A_{m}(0)^{\mathfrak{n}=0}$, then $f\big|_{gD_{m}}\in \mathcal{O}_{X^{\textnormal{rig}}_{m}}(gD_{m})^{\mathfrak{n}=0}=\mathcal{O}_{X^{\textnormal{rig}}_{m}}(gD_{m})^{\mathfrak{g}=0}$ for all $g\in G^{0}$ by Proposition \ref{ninvariantsequalginvariants}. The $\mathfrak{g}$-linear injection $A_{m}(0)\hookrightarrow\prod_{g\in G^{0}}\mathcal{O}_{X^{\textnormal{rig}}_{m}}(gD_{m})$ of $\breve{K}$-algebras induces an equality \begin{equation*}
A_{m}(0)^{\mathfrak{g}=0}=A_{m}(0)\cap\prod_{g\in G^{0}}\mathcal{O}_{X^{\textnormal{rig}}_{m}}(gD_{m})^{\mathfrak{g}=0}.
\end{equation*} Therefore, $f\in A_{m}(0)^{\mathfrak{g}=0}$, and hence $A_{m}(0)^{\mathfrak{g}=0}=A_{m}(0)^{\mathfrak{n}=0}$. This gives us \linebreak $(N^{s}_{m}(0)_{\textnormal{lf}})^{\mathfrak{n}=0}\subseteq A_{m}(0)^{\mathfrak{g}=0}\varphi_{0}^{s}$. As explained earlier, $N^{s}_{m}(0)_{\textnormal{lf}}$ is generated as an $\mathfrak{sl}_{h}(K_{h})$-module by its $\mathfrak{n}$-invariants $(N^{s}_{m}(0)_{\textnormal{lf}})^{\mathfrak{n}=0}$. Thus  
\begin{equation*}
N^{s}_{m}(0)_{\textnormal{lf}}=\mathfrak{sl}_{h}(K_{h}) .(N^{s}_{m}(0)_{\textnormal{lf}})^{\mathfrak{n}=0}\subseteq \mathfrak{sl}_{h}(K_{h}) .(A_{m}(0)^{\mathfrak{g}=0}\varphi^{s}_{0})=A_{m}(0)^{\mathfrak{g}=0}V_{s}.
\end{equation*} The last equality follows from the fact that $\varphi^{s}_{0}$ is a highest weight vector in $V_{s}$ (cf. Remark \ref{highest_wt_of_Vs}).\\
\indent  If the $\Gamma$-action on $\Phi_{m}^{-1}(\Phi(D))$ is changed via the automorphism $\gamma\mapsto\Pi^{-i}\gamma\Pi^{i}$, then the map $\Pi^{i}:\Phi_{m}^{-1}(\Phi(D))\iso\Pi^{i}\Phi_{m}^{-1}(\Phi(D))$ is a $\Gamma$-equivariant isomorphism. We note that the new $\Gamma$-action does not change the locally finite vectors in $N^{s}_{m}(0)$. Writing $\varphi_{h}:=\varphi_{0}$ formally, we have an induced isomorphism $\Pi^{i}:(N^{s}_{m}(i))_{\textnormal{lf}}\iso (N^{s}_{m}(0))_{\textnormal{lf}}$ mapping $\varphi_{0}^{s}$ to $\varphi_{h-i}^{s}$, and the $\mathfrak{n}$-invariants onto the $\mathfrak{n}_{i}:=\textnormal{Ad}_{\Pi^{i}}(\mathfrak{n})$-invariants for all $0\leq i\leq h-1$. Therefore,
\begin{align*}
(N^{s}_{m}(i)_{\textnormal{lf}})^{\mathfrak{n}=0}=(\Pi^{i})^{-1}((N^{s}_{m}(0)_{\textnormal{lf}})^{\mathfrak{n}_{i}=0})&\subseteq (\Pi^{i})^{-1}((A_{m}(0)^{\mathfrak{g}=0}V_{s})^{\mathfrak{n}_{i}=0})\\&=(\Pi^{i})^{-1}(A_{m}(0)^{\mathfrak{g}=0}V_{s}^{\mathfrak{n}_{i}=0})\\&=(\Pi^{i})^{-1}(A_{m}(0)^{\mathfrak{g}=0}\varphi^{s}_{h-i})\\&=A_{m}(i)^{\mathfrak{g}=0}\varphi_{0}^{s}.
\end{align*}
As before, this also implies $N^{s}_{m}(i)_{\textnormal{lf}}\subseteq A_{m}(i)^{\mathfrak{g}=0}V_{s}$ for all $0\leq i\leq h-1$.
\end{proof}
\begin{theorem}\label{lf3} For all $s\geq 0$, $m\geq 0$, we have an isomorphism 
\begin{equation*}
(M^{s}_{m})_{\textnormal{lf}}\cong\breve{K}_{m}\otimes_{\breve{K}}V_{s}\cong\breve{K}_{m}\otimes_{\breve{K}}\mathcal{O}_{\mathbb{P}_{\breve{K}}^{h-1}}(s)(\mathbb{P}_{\breve{K}}^{h-1})\cong\breve{K}_{m}\otimes_{\breve{K}}\textnormal{Sym}^{s}(B_{h}\otimes_{K_{h}}\breve{K})
\end{equation*} of $\Gamma$-representations for the diagonal $\Gamma$-action on the tensor products. The representation $(M^{s}_{m})_{\textnormal{lf}}$ is a finite dimensional semi-simple locally algebraic representation of $\Gamma$. 
\end{theorem}
\begin{proof}
As before, $(M^{s}_{m})_{\textnormal{lf}}$ is generated as an $\mathfrak{sl}_{h}(K_{h})$-module by its $\mathfrak{n}$-invariants. Let $x\in((M^{s}_{m})_{\textnormal{lf}})^{\mathfrak{n}=0}$. Then, using the preceding proposition, $x\big|_{\Pi^{i}\Phi_{m}^{-1}(\Phi(D))}\in(N^{s}_{m}(i)_{\textnormal{lf}})^{\mathfrak{n}=0}\subseteq A_{m}(i)^{\mathfrak{g}=0}\varphi_{0}^{s}$ for all $0\leq i\leq h-1$. Let $Y_{i}:=\Pi^{i}\Phi_{m}^{-1}(\Phi(D))$, and write $x\big|_{Y_{i}}=f_{i}\varphi^{s}_{0}$ with $f_{i}\in A_{m}(i)^{\mathfrak{g}=0}$.\\
\indent  For all $0\leq i,j\leq h-1$, we have $\big(f_{i}\big|_{Y_{i}\cap Y_{j}}-f_{j}\big|_{Y_{i}\cap Y_{j}}\big)\varphi^{s}_{0}=x\big|_{Y_{i}\cap Y_{j}}-x\big|_{Y_{i}\cap Y_{j}}=0$. Now $M^{s}_{m}$ is free over the integral domain $R^{\textnormal{rig}}_{m}$, and contains $\varphi^{s}_{0}\neq 0$. Hence the map $(r\mapsto r\varphi^{s}_{0})$ from $R^{\textnormal{rig}}_{m}$ to $M^{s}_{m}$ is injective and remains injective after any flat base change. In particular, the map $(r\mapsto r\varphi^{s}_{0}):\mathcal{O}_{X^{\textnormal{rig}}_{m}}(Y_{i}\cap Y_{j})\longrightarrow(\mathcal{M}^{s}_{m})^{\textnormal{rig}}(Y_{i}\cap Y_{j})$ is injective, and thus $f_{i}\big|_{Y_{i}\cap Y_{j}}=f_{j}\big|_{Y_{i}\cap Y_{j}}$ for all $0\leq i,j\leq h-1$. Therefore, by the sheaf axioms, the functions $(f_{i})_{i}$ glue together to a global section $f\in R^{\textnormal{rig}}_{m}$ and $x=f\varphi^{s}_{0}$. Since $f\big|_{Y_{i}}=f_{i}\in A_{m}(i)^{\mathfrak{g}=0}$ for all $i$, and the map $R^{\textnormal{rig}}_{m}\hookrightarrow\prod_{i=0}^{h-1}A_{m}(i)$ is $\mathfrak{g}$-equivariant, $f\in(R^{\textnormal{rig}}_{m})^{\mathfrak{g}=0}=\breve{K}_{m}$ (cf. Remark \ref{ginvinRm}). Hence $x\in\breve{K}_{m}\varphi^{s}_{0}$. As a result, $(M^{s}_{m})_{\textnormal{lf}}\subseteq\mathfrak{sl}_{h}(K_{h}).(\breve{K}_{m}\varphi^{s}_{0})=\breve{K}_{m}V_{s}$. The other inclusion $\breve{K}_{m}V_{s}\subseteq (M^{s}_{m})_{\textnormal{lf}}$ is easy to see as $(M^{s}_{m})_{\textnormal{lf}}$ is a module over $(R^{\textnormal{rig}}_{m})_{\textnormal{lf}}=\breve{K}_{m}$, and $V_{s}=(M^{s}_{0})_{\textnormal{lf}}\subseteq (M^{s}_{m})_{\textnormal{lf}}$.\\
\indent  Now to justify the isomorphism $\breve{K}_{m}\otimes_{\breve{K}}V_{s}\cong\breve{K}_{m}V_{s}$, it is enough to show that the natural map 
\begin{align*}
& \hspace{1.6cm} \breve{K}_{m}\otimes_{\breve{K}}V_{s}\longrightarrow \breve{K}_{m}V_{s}\\&
\sum_{0\leq|\alpha|\leq s}c_{\alpha}(1\otimes w^{\alpha}\varphi_{0}^{s})\longmapsto\sum_{0\leq|\alpha|\leq s}c_{\alpha} w^{\alpha}\varphi_{0}^{s}
\end{align*}
is injective. Here the set $\lbrace 1\otimes w^{\alpha}\varphi_{0}^{s}\rbrace_{0\leq|\alpha|\leq s}$ forms a $\breve{K}_{m}$-basis of $\breve{K}_{m}\otimes_{\breve{K}}V_{s}$. By Lemma \ref{explicit_g-action_lemma}, we have $\mathfrak{x}_{00}(w^{\alpha}\varphi_{0}^{s})=(s-|\alpha|)w^{\alpha}\varphi_{0}^{s}$ and $\mathfrak{x}_{ii}(w^{\alpha}\varphi_{0}^{s})=\alpha_{i}w^{\alpha}\varphi_{0}^{s}$ for all $1\leq i\leq h-1$. Since $\mathfrak{g}$ annihilates $\breve{K}_{m}$, if $\sum_{0\leq|\alpha|\leq s}c_{\alpha} w^{\alpha}\varphi_{0}^{s}=0$, one can use the above actions of the diagonal matrices iteratively to deduce that each summand $c_{\alpha} w^{\alpha}\varphi_{0}^{s}$ is zero, and therefore $c_{\alpha}=0$ for all $0\leq |\alpha|\leq s$.\\
\indent  Unlike $(M^{s}_{0})_{\text{lf}}=V_{s}$, the space of locally finite vectors $(M^{s}_{m})_{\text{lf}}\cong\breve{K}_{m}\otimes_{\breve{K}}V_{s}$ at level $m>0$ is not an irreducible $\Gamma$-representation as it properly contains the representation $V_{s}$. However, it is semi-simple and this can be seen as follows. The action of $\Gamma$ on $\breve{K}_{m}$ factors through a finite group. As a result, $\breve{K}_{m}$ decomposes into a direct sum $\breve{K}_{m}\cong\bigoplus_{i=1}^{n}W_{i}$ of irreducible representations. This gives us a decomposition \begin{equation}\label{ss_decomp}
 \breve{K}_{m}\otimes_{\breve{K}}V_{s}\cong \bigoplus_{i=1}^{n}(W_{i}\otimes_{\breve{K}}V_{s}).
\end{equation} Now we note that $V_{s}\cong\textnormal{Sym}^{s}(B_{h}\otimes_{K_{h}}\breve{K})$ is an irreducible algebraic representation of $\Gamma\cong\mathfrak{o}_{B_{h}}^{\times}$ (cf. Theorem \ref{top_finiteness_of_MsD_thm} and \cite{kohliwamo}, Remark 4.4), and $\breve{K}_{m}$ is a smooth representation of $\Gamma$ by Remark \ref{ginvinRm}. Thus every direct summand in (\ref{ss_decomp}) is a tensor product of a smooth irreducible representation and an irreducible algebraic representation of $\Gamma$. Such a product is an irreducible locally algebraic representation by \cite{stug}, Appendix by Dipendra Prasad, Theorem 1. As a consequence, $(M^{s}_{m})_{\text{lf}}$ is a semi-simple locally algebraic representation of $\Gamma$ and exhausts all locally algebraic vectors in $M^{s}_{m}$ as every locally algebraic vector is locally finite by definition (cf. \cite{eme04}, paragraph after Definition 4.2.1). 
\end{proof}

\end{document}